\documentclass{puthesis-UG}

\usepackage{style}

\author{Alan Yan}
\adviser{Professor June E Huh}
\title{Log-concavity in Combinatorics}

\abstract{In this thesis we survey some of the mechanisms used to prove that naturally defined sequences in combinatorics are log-concave. Among these mechanisms are Alexandrov's inequality for mixed discriminants, Alexandrov's Fenchel inequality for mixed volumes, Lorentzian polynomials, and the Hard Lefschetz theorem. We use these mechanisms to prove some new log-concavity and extremal results related to partially ordered sets and matroids. We present joint work with Ramon van Handel and Xinmeng Zeng to give a complete characterization for the extremals of the Kahn-Saks inequality. We extend Stanley's inequality for regular matroids to arbitrary matroids using the technology of Lorentzian polynomials. As a result, we provide a new proof of the weakest Mason conjecture. We also prove necessary and sufficient conditions for the Gorenstein ring associated to the basis generating polynomial of a matroid to satisfy Hodge-Riemann relations of degree one on the facets of the positive orthant.}

\acknowledgements{First and foremost, I would like to thank June Huh for his mentorship throughout this senior thesis project. Through our numerous discussions, I have grown both as a mathematician and as a researcher. Your attentiveness and patience faciliated an environment where I could explore research ideas freely. For that, I am immensely grateful. This thesis would not have been as successful without your help.

I would also like to thank Ramon van Handel for introducing me to the topic of convex bodies and the Alexandrov-Fenchel inequality. The Kahn-Saks project that you suggested was an excellent way for me to enter the study of log-concavity and algebraic combinatorics. Thank you for your thorough proofreading and suggestions. 

I am grateful to Evita Nestoridi for suggesting algebraic combinatorics as a possible field of interest to me. Although our work in Markov chain mixing times is orthogonal to the topic of the thesis, I was introduced to a rich variety of algebraic objects and techniques which informed my decision to pursue a senior thesis topic in algebraic combinatorics. 

I am grateful to Christopher Eur for our discussions in the Harvard common room. I am grateful to Connor Simpson for providing a counterexample to Conjecture~\ref{conj:my-conjecture}.

I would like to thank the department of mathematics for providing the environment, resources, and opportunity to complete a senior thesis. Throughout the four years of undergraduate education, I never once felt unwelcome or pressured by departmental graduation requirements. 

Thank you to Christopher Guan, Daniel Hu, Josh Lau, Cindy Li, Elie Svoll, Erin Watson, and all other Tora Taiko members (there are a lot of you now!) It was truly a pleasure to hang out and play music together. After all of us veterans graduate, I believe you are all a part of a Golden Age for Tora Taiko. I'm looking forward to witnessing your growth as an ensemble!

I would like to sincerely thank Ollie Thakar and Ben Bobell for their friendship and camaraderie. Thank you for being the best people you could be. I deeply admire your kindness, and I wish you two the best in your future endeavors. In general, I would like to thank the entire mathematical community localized at the Fine Hall common room. Even though I couldn't show up as frequently this year, the mathematical community remained a welcoming space whenever I returned. 

I would like to express my deepest appreciation to Karen Li. Thank you for putting up with many of my late nights and early mornings during writing process. Since meeting you, my life has brightened in ways that I would have never imagined. 

Finally, I want to thank my family for all of their love and support throughout the years. 
}

\date{May 1st, 2023}

\begin{document}

\chapter*{Introduction}
\addcontentsline{toc}{section}{Introduction}

Given a sequence of non-negative real numbers $a_0, \ldots, a_n, a_{n+1}$, we say that the sequence is log-concave if and only if $a_i^2 \geq a_{i-1} a_{i+1}$ for all $i \in [n]$. If a sequence of non-negative numbers is log-concave, then it is also unimodal. There is a phenomenon in mathematics that many naturally defined sequences in algebra, combinatorics, and geometry satisfy some version of log-concavity. To explain this phenomenon, there is a rich variety of methods that have been developed to prove log-concavity problems in combinatorics. Our goal in this thesis is to survey a few of log-concavity mechanisms. Specifically, we will survey some of the theory behind mixed discriminants, mixed volumes, Lorentzian polynomials, and Hodge-Riemann relations. Note that this is only a small subset of the techniques to prove sequences are log-concave. In particular, we do not mention analytic techniques, linear algebraic techniques, or real-rootedness techniques. For a more thorough treatment of these techniques, we refer the reader the excellent survey paper by Richard Stanley \cite{Stanley1989LogConcaveAU}. We also do not mention the recent technology of the combinatorial atlas defined by Swee Hong Chan and Igor Pak. For background on the combinatorial atlas, we refer the reader to the original sources \cite{logconcave-poset-inequalities,combinatorial-atlas}. \\

We will be especially interested in the log-concavity of sequences related to partially ordered sets and matroids. The thesis topic was motivated by a project recommended by Ramon van Handel about the extremals of the Kahn-Saks inequality. In a paper by him and Yair Shenfeld, they give a characterization for equality to hold in the Alexandrov-Fenchel inequality when the mixed area measure is associated to polytopes. Using this characterization, they were able to give a characterization for the equality cases of the Stanley poset inequality \cite{STANLEY}. As an extension of the result, we tried to give a similar characterization for the Kahn-Saks inequality. Given a finite poset $P$ and two distinguished elements $x, y \in P$, we can count $N_k$ the number of linear extensions $\sigma$ satisfying $\sigma (y) - \sigma (x) = k$. In \cite{balancing-poset-extensions}, Jeff Kahns and Michael Saks prove that this sequence is log-concave using the theory of mixed volumes and the Alexandrov-Fenchel inequality. Through joint work with Ramon van Handel and Xinmeng Zeng, we were able to get a complete combinatorial characterization for the Kahn-Saks inequality. We present these results in Section~\ref{kahn-saks-inequality}. \\

As an extension of the Kahn-Saks project, we considered the second inequality proved by Stanley in \cite{STANLEY} for matroids. Given a regular matroid $M = (E, \mcI)$ of rank $n$ and a partition $E = R \sqcup Q$ of the ground set, we define $B_k$ to be the number of bases of $M$ which share $k$ elements with $R$. Using the theory of mixed volumes, Stanley proves that the sequence $B_k$ is ultra-log-concave. In his paper, Stanley gives a characterization for the extremals of the slightly weaker inequality $B_1^n \geq B_0^{n} B_{n-1}$ based on the equality cases of the Minkowski inequality. In \cite{bapat_raghavan_1997}, Bapat and Raghavan give a proof of Stanley's matroid inequality by associating each number $B_k$ with a mixed discriminant rather than a mixed volume. Suppose our sequence $B_k$ consists of only positive integers. From our better understanding of the equality cases of Alexandrov's inequality for mixed discriminants, the extremals $B_k^2 \geq B_{k-1} B_{k+1}$ are exactly the extremals of the weaker inquality $B_{1}^n \geq B_0^{n-1} B_n$. This gives a satisfactory characterization of extremals of Stanley's matroid inequality for regular matroids. At this point, we were interested in two questions. Is Stanley's inequality true for arbitrary matroids? If Stanley's inequality is true for arbitrary matroids, will the characterization for the extremals be the same? In Section~\ref{subsection-matroid-lorentzian-stanley-general}, we prove that the answer to the first question is in the affirmative. The proof uses the fact that the basis generating polynomial of a matroid is Lorentzian. For the second question, we proved that the combinatorial characterization guarentees equality. However, we made little progress for the reverse implication. \\

In our attempt to combinatorial characterize the extremals of Stanley's matroid inequality for arbitrary matroids, we began exploring the Hodge theory of matroids. In Section~\ref{sec:gorenstein-ring-of-matroid}, we define a ring $\A(M)$ which we call the the Gorenstein ring associated to the basis generating polynomial of the matroid $M$. In their paper \cite{MNY}, Murai-Nagaoka-Yazawa prove that the ring $\A(M)$ satisfies the Hard Lefschetz property and the Hodge-Riemann relations in degree $1$. They conjecture that the ring satisfies the Hard Lefschetz property in all degrees $k \leq \frac{\rank (M)}{2}$. Our current goal is to prove the stronger conjecture that $\A(M)$ satisfies the K\"ahler package. Our progress towards the general conjecture is outlined in Chapter~\ref{chap:hodge-theory-for-matroids}. 

\chapter{Combinatorial Structures and Convex Geometry}

In this chapter, we review some basic structures from combinatorics and convex geometry. We begin in Section~\ref{sec:posets} by reviewing the notions of a partially ordered set and related concepts. We will not cover any deep concepts in order theory and will content ourselves in reviewing the basic definitions of linear extensions, lattices, and covering relations. In Section~\ref{sec:graph-theory}, we will go over basic notions in graph theory and spectral graph theory. In this section, the notion of the Laplacian of a loopless graph will be important. In Section~\ref{sec:matroids}, we go over the definitions and notions associated to matroids. The properties of matroids will be at center stage for many of our applications in later chapters. In Section~\ref{sec:mixed-discriminants}, we will cover the notion of mixed discriminants. These objects arise as the polarization form of the determinant. Mixed discriminants were introduced by Alexandrov in \cite{aleksandrov} to study the mixed volumes of convex bodies. Finally, in Section~\ref{sec:convex-bodies}, we outline notions in convex geometry and Brunn-Minkowski theory including the notion of mixed volumes. Our goal with this chapter is not to provide an exhaustive overview of the topics mentioned, but to cover the definitions, results, and applications needed for the rest of the thesis. For the reader who wishes to dive more deeply into these individual topics, we refer to more in-depth treatments of the concepts within each section. 

\section{Partially Ordered Sets} \label{sec:posets}

In this section, we review basic notions in the theory of partially ordered sets (posets). Our treatment of posets is similar to that of \cite{ordered-sets} and \cite{Rota1964}. Given a finite set $P$, we abstractly define a binary relation on $P$ as simply a subset of $P \times P$. A \textbf{partial order} on a set is a binary relation that is reflexive, antisymmetric, and transitive. These properties are described in Definition~\ref{def:poset}. 

\begin{defn}[Definition 1.1.1 in \cite{ordered-sets}]\label{def:poset}
	A \textbf{partially ordered set} is an ordered pair $(P, \leq)$ of a set $P$ and a binary relation $\leq$ on $P$ which is reflexive, antisymmetric, and transitive. Explicitly, we have the following conditions. 
	\begin{enumerate}[label = (\alph*)]
		\item $x \leq x$ for all $x \in P$.
		\item If $x \leq y$ and $y \leq x$, then $x = y$.
		\item If $x \leq y$ and $y \leq z$ then $x \leq z$.
	\end{enumerate}
	When $x \leq y$ and $x \neq y$, we can also write $x < y$ or $y > x$.
\end{defn}

In this thesis, we will only consider posets where the ground set is finite. Let $(P, \leq)$ be a finite partially ordered set. We call two elements $x, y \in P$ \textbf{comparable} if and only if $x \leq y$ or $y \leq x$. We write $x \sim y$ if and only if $x$ and $y$ are comparable. If $x$ and $y$ are not comparable, we say that they are incomparable. We call a subset $C \subseteq P$ a \textbf{chain} if every pair of elements in $C$ are comparable. In a finite poset, a chain will always have the form $\{x_1 < \ldots < x_k\}$. For every pair of elements $x, y$, we can define the closed interval $[x, y] := \{z \in P : x \leq z \leq y\}$. We say that $y$ \textbf{covers} $x$ (or $x$ is covered by $y$) if $[x, y] = \{x, y\}$. In this case, we call $x \lessdot y$ a \textbf{covering relation}, and we write $x \lessdot y$ or $y \gtrdot x$. We can diagrammatically visualize posets by drawing each element as a point and drawing the covering relations as edges. Such a diagram is called a \textbf{Hasse diagram}. For an example of a Hasse diagram, see Figure~\ref{fig:example-of-a-is-two}. It is not difficult to see that the covering relations of a poset determine the poset uniquely. A poset element is called a \textbf{minimal} element if it is not greater than any other element. Similarly, we call an element a \textbf{maximal} element if it is not less than any other element. Every pair of elements $x, y \in P$ satisfying $x \leq y$ has a maximal chain $C$ with $x$ as the minimal element in the chain and $y$ as the maximal element in the chain. Any maximal chain from $x$ to $y$ will be of the form 
\[
	x = z_0 \lessdot z_1 \lessdot \ldots \lessdot z_{n-1} \lessdot z_n = y. 
\]
Posets appear all throughout mathematics. For example, the set of integers $\ZZ$ can be equipped with divisibility to give it a partially ordered set structure. The M\"obius function associated with the lattice of integers is a common object in the study of analytic number theory (see \cite{Apostol2013-an}). In category theory, posets are examples of the most basic form of categories. In particular, they are categories where the morphisms between any two objects consists of a single element. Given a collection of sets, they can be given a poset structure with set inclusion as the partial order. Finally, in graph theory, we can equip the vertices of a directed graph with a poset structure where two vertices are comparable if and only if one can be reached from the other. 

\subsection{Linear Extensions}

Given two posets $(P_1, \leq_1)$ and $(P_2, \leq_2)$, we define say a map $f : P_1 \to P_2$ is \textbf{order-preserving} if $f(x) \leq_2 f(y)$ whenever $x, y \in P_1$ satisfy $x \leq_1 y$. If $|P| = n$, we call any bijective order-preserving map $f : P \to [n]$ a linear extension where $[n]$ is equipped with its natural total order. 

\begin{defn} \label{defn:linear-extension}
	Let $(P, \leq)$ be a poset on $n$ elements. A \textbf{linear extension} is any bijective map $f : P \to [n]$ such that $f(x) < f(y)$ whenever $x, y \in P$ satisfies $x < y$.
\end{defn}

Recall that a partial order on $P$ is called a $\textbf{total order}$ if every pair of elements is comparable. Every partial order can be viewed a total order where some comparability information is missing. A linear extension is simply a way to extend a partial order to a total order. In later sections, we will be interested in the log-concavity of sequences which enumerate linear extensions of a poset.

\subsection{Lattices}

Let $P$ be a poset and let $x, y \in P$ be two arbitrary elements. We say that $z \in P$ is a \textbf{least upper bound} or \textbf{join} of $x$ and $y$ if $z \geq x$, $z \geq y$, and for any $w \in P$ satisfying $w \geq x, w \geq y$ we have that $w \geq z$. Similarly, we say $z$ is a \textbf{greatest lower bound} or \textbf{meet} if $z \leq x$, $z \leq y$, and for any $w \in P$ satisfying $w \leq x, w \leq y$, we have that $w \leq z$. If the meet or join of two elements exist, they must be unique. In a general poset, the meet and join of two elements does not necessarily exist. When they do exist for every pair of elements, we call the poset a lattice. 

\begin{defn} \label{defn:lattice}
	Let $\mcL$ be a finite poset. We say $\mcL$ is a lattice if every pair of elements in $\mcL$ has a meet and a join. When $x, y \in \mcL$, we let $x \wedge y$  and $x \vee y$ denote the meet and join of $x$ and $y$. 
\end{defn}

Given a lattice, it is not hard to show that the meet and join operations are commutative and associative. The partially ordered set of natural numbers equipped with divisibility forms a lattice under the greatest common denominator and the least common multiple. A lattice will automatically have a unique minimal element which is a global minimum and a unique maximal element which is a global maximum. If a poset has a global minimum, we call this element the $0$ element. If a poset has a global maximum, we call this element the $1$ element. Let $x \in P$ be an element which covers the $0$ element. In this case, we call $x$ an \textbf{atom}. Dually, if $x$ is covered by the $1$ element, then we call $x$ a \textbf{co-atom}. For some examples of lattices, we will be introduced to the lattice of faces of a polytope and the lattices of flats of a matroid. The latter example satisfies extra conditions which makes it \textbf{geometric lattice}. For a thorough treatment of geometric lattices and their connections to matroids, we refer the reader to \cite{10.5555/1197093}. 

\section{Graph Theory} \label{sec:graph-theory}

In this section, we briefly review some notions in graph theory. We assume that the reader has some basic background knowledge on graph theory such as the definitions of connected components, paths, trees, etc. In Definition~\ref{def:graph}, we provide the definition of a graph that we will use in the thesis. Note that to each graph $G$ we attach some arbitrary total ordering on the vertices. For a thorough reading of graph theory, we refer the reader to \cite{diestel}. For notions in spectral graph theory, we refer the reader to \cite{chung-spectral-graph-theory}. 

\begin{defn} \label{def:graph}
	A \textbf{graph} is an ordered pair $(V, E)$ of vertices and edges such that each edge is associated with either two distinct vertices or one vertex. If an edge is associated with two distinct vertices, then we call it a \textbf{simple edge}. If an edge is associated with one vertex, then we call it a \textbf{loop}. We also equip $V$ with an arbitrary total ordering. 
\end{defn}

The role of the total ordering in Definition~\ref{def:graph} will show up in Definition~\ref{def:incidence-matrix} when we define the incidence matrix. When two vertices in $G$ contain an edge, we say that they are \textbf{adjacent}. If two vertices $v$ and $w$ are adjacent, we write $v \sim w$. This corresponds to comparability in the reachability poset of the graph. If an edge contains a vertex, we say that the edge is \textbf{incident} to the vertex. Given a connected graph $G$, we say a subgraph $T \subseteq G$ is a \textbf{spanning tree} if it is a tree that is incident to all vertices of $G$. In general, when $G$ is a graph (not necessarily connected), we call $T$ a \textbf{spanning forest} if it is a forest that is incident to all vertices of $G$. 

\subsection{Spectral Graph Theory} \label{sec:spectral-graph-theory}

In this section, we will assume that our graph $G$ is loopless. For every vertex $v \in V$, we define $\deg(v)$ to be the number of edges incident to $v$. For any pair of distinct vertices $v, w \in V$, we define $e(v, w)$ to be the number of edges between $v$ and $w$. In particular, we have that 
\[
	\deg (v) = \sum_{\substack{w \in V \\ w \neq v}} e(v, w).
\]

\begin{defn}
	Let $G = (V, E)$ be a (loopless) graph where $V = \{v_1, \ldots, v_n\}$. We define its Laplacian matrix $L := L_G$ to be the $n \times n$ matrix where the $(i, j)$ entry is given by 
	\[
		L_{i, j} := \begin{cases}
			\deg (v_i) & \text{if $i = j$} \\
			-E_{v_i, v_j} & \text{if $i \neq j$ and $v_i \sim v_j$}
		\end{cases}
	\]
	where $E_{v_i, v_j}$ is the number of edges between $v_i$ and $v_j$. 
\end{defn}

\begin{defn} \label{def:incidence-matrix}
	Let $G$ be a graph and let $V = \{v_1 < \ldots < v_n\}$ be an arbitrary ordering of the vertices. We define a $|V| \times |E|$ matrix $B := B_G$ called the \textbf{incidence matrix} where the entry indexed by the vertex $v$ and the edge $e = \{v_i < v_j\} \in E$ is equal to 
	\[
		B_{ve} = \begin{cases}
			1, & \text{if $v = v_i$} \\
			-1, & \text{if $v = v_j$} \\
			0, & \text{otherwise.}
		\end{cases}
	\] 
\end{defn}

The matrix $C_G$ which is obtained by removing the last row of $B_G$ is called the \textbf{reduced incidence matrix}. The incidence matrix and reduced incidence matrix both satisfy Proposition~\ref{incidence-reduced-incidence-matrix-is-good}. In this sense, these two matrices capture the property of being a cycle in the graph $G$.  

\begin{prop} \label{incidence-reduced-incidence-matrix-is-good}
	For a graph $G = (V, E)$ the incidence matrix $B_G$ and reduced incidence matrix $C_G$ both satisfy the property that a set of columns is linearly dependent if and only if the graph formed by the corresponding edges contains a cycle. 
\end{prop}

\begin{proof}
	See Example 5.4 of \cite{bapat_raghavan_1997}. 
\end{proof}

The incidence matrix and Laplacian of a matrix are related by Proposition~\ref{incidence-and-laplacian-relation-prop}. This relation will reappear in Section~\ref{sec:stanley-matroid-inequality} when we prove Theorem~\ref{stanley-equality-cases-matroid-thm} in the case of graphic matroids. 

\begin{prop} \label{incidence-and-laplacian-relation-prop}
	For a graph $G = (V, E)$, let $L$ be its Laplacian matrix and let $B$ be its incidence matrix. Then, we have that $L = BB^T$. 
\end{prop}

\begin{proof}
	For every $e = \{v_i, v_j\} \in E$ with $v_i < v_j$, we define $\pi_e(v_i) = -1$ and $\pi_e (v_j) = 1$. The function $\pi_e$ indicates which of the two vertices in an edge is the smaller vertex with respect to the total ordering on the vertices. We can compute that 
	\[
		\left ( BB^T \right )_{v, w} = \sum_{e \in E : v, w \in e} \pi_e(v) \pi_e(w). 
	\]
	When $v = w$, then each summand is equal to $1$ and we get exactly $\deg (v_i)$. If $v$ and $w$ are not adjacent, then the sum is empty and is trivially is equal to zero. Otherwise, each summand is $-1$ and there are $E_{v_i, v_j}$ elements in this sum. This suffices for the proof. 
\end{proof}

\section{Matroids} \label{sec:matroids}

Matroids are combinatorial objects which abstract and generalize several properties in linear algebra, graph theory, and geometry. For example, it generalizes the notion of cyclelessness in graphs, the notion of linear independence in vector spaces, the poset structure of linear subspaces in a vector space, and concurrence in configurations of points and lines. Despite the seemingly limited conditions imposed on a matroid, matroids successfully describe many objects relevant to other areas in mathematics such as topology \cite{gelfand}, graph theory \cite{milnor-numbers}, combinatorial optimization \cite{optimization}, algebraic geometry \cite{schubert-cell}, and convex geometry \cite{matroid-polytope}. In this section, we provide an introduction to the basic notions in matroid theory that we will need in the remainder of this thesis. We use the excellent monographs \cite{10.5555/1197093} and \cite{welsh} as our main references for this theory. We begin by describing matroids as a set with a collection of independent sets. 

\subsection{Independent Sets} \label{sec:independent-sets}

As motivation for the definition of a matroid, we first describe some properties of linearly independent vectors in a vector space. Let $V$ be a (finite-dimensional) vector space and let $S \subseteq V$ be a subset of linearly independent vectors. Note that any subset of $S$ will also consist of linearly independent vectors. If $T \subseteq V$ is another set of linearly independent vectors with $|S| < |T|$, then there always exists a vector $v \in T \backslash S$ such that $S \cup v$ is linearly independent. To prove this fact, suppose for the sake of contradiction that there is not $v \in T \backslash S$ for which $S \cup v$ is linearly independent. This implies that $\text{span} (T) \subseteq \text{span} (S)$. Since our ambient vector space is finite-dimensional, by comparing dimensions we reach a contradiction. We abstract these properties in Defnition~\ref{def:matroid-independent-sets} and call the resulting object a matroid.

\begin{defn} \label{def:matroid-independent-sets}
	A \textbf{matroid} is an ordered pair $M = (E, \mcI)$ consisting of a finite set $E$ and a collection of subsets $\mathcal{I} \subseteq 2^E$ which satisfy the following three properties:
	\begin{enumerate}
		\item[(I1)] $\emptyset \in \mathcal{I}$.
		\item[(I2)] If $X \subseteq Y$ and $Y \in \mathcal{I}$, then $X \in \mathcal{I}$.
		\item[(I3)] If $X, Y \in \mathcal{I}$ and $|X| > |Y|$, then there exists some element $e \in X \backslash Y$ such that $Y \cup \{e\} \in \mathcal{I}$.
	\end{enumerate}
\end{defn}

The set $E$ is called the \textbf{ground set} of the matroid and the collection of subsets $\mathcal{I}$ are called \textbf{independent sets}. This terminology is motivated by Example~\ref{linear-matroid} where the indepedent sets consist exactly of linearly independent subsets of our ground set. Condition (I1) is referred to as the non-emptiness axiom, condition (I2) is referred to as the hereditary axiom, and condition (I3) is referred to as the exchange axiom. We say two matroids $M_1$ and $M_2$ are isomorphic if there is a bijection between their ground sets which induces a one-to-one correspondence between their independent sets. 

\begin{example} [Linear Matroids] \label{linear-matroid}
	Let $V$ be a $k$-vector space and let $E = \{v_1, \ldots, v_n\}$ be a finite set of vectors from $V$. Let $\mcI$ consist of all subsets of $E$ which are linearly independent. Then, the ordered pair $(E, \mcI)$ forms a matroid. For a set of linearly independent vectors or equivalently a matrix $A$, we let $M(A)$ denote the linear matroid generated by $A$. We call a matroid $M$ a \textbf{linear matroid} if there exists a matrix $A$ such that $M \cong M(A)$. When there is a $k$-vector space $V$ such that the matroid $M$ is generated by a set of vectors in $V$, we say that $M$ is \textbf{representable} over $k$. 
\end{example}

\begin{example} [Graphic Matroids] \label{example:graphic-matroid}
	Let $G = (V, E)$ be a graph and let $\mcI$ be the collection of subsets of $E$ which consist of edges such that the subgraph on $V$ with these edges is a forest (contains no cycles). The ordered pair $(E, \mcI)$ forms a matroid called the \textbf{cycle matroid} of the graph $G$. If $G$ is a graph, we let $M(G)$ be the cycle matroid associated to the graph $G$. We call any matroid isomorphic to $M(G)$ for some graph $G$ a \textbf{graphic matroid}. 
\end{example}

It is not difficult to directly show that the graphic matroid associated to a graph as in Example~\ref{example:graphic-matroid} is a matroid. We can show this fact indirectly using Proposition~\ref{graphic-are-linear}. Specifically, this proposition show that graphic matroids are also linear matroids.

\begin{prop}[Proposition 1.2.9 in \cite{10.5555/1197093}] \label{prop:graphic-matroids-connected}
	Let $M$ be a graphic matroid. Then $M \cong M(G)$ for some connected graph $G$. 
\end{prop}
\begin{proof}
	Since $M$ is graph, there exists a graph $H$ (not necessarily connected) such that $M \cong M(H)$. Take the connected components of $H$ and pick a one vertex from each of them. By identifying these vertices, we get a connected graph $G$ such that $M \cong M(G)$. This suffices for the proof. 
\end{proof}

\begin{prop} \label{graphic-are-linear}
	Let $G = (V, E)$ be a graph and let $A$ be its incidence matrix or reduced incidence matrix. Then $M(G) \cong M(A)$. 
\end{prop}

\begin{proof}
	This follows immediately from Proposition~\ref{incidence-reduced-incidence-matrix-is-good}. 
\end{proof}

\begin{example}[Uniform Matroids]
	For any integers $0 \leq k \leq n$, we can define the \textbf{uniform matroid} $U_{k, n}$ which is the matroid on $[n]$ where the independent sets consist of all subsets of $[n]$ of size at most $k$. The matroid $U_{n, n}$ is called the \textbf{boolean matroid} or \textbf{free matroid} on $n$ elements. 
\end{example}

Given a matroid $M = (E, \mathcal{I})$, we call a subset $X \subseteq E$ a \textbf{dependent set} if and only if $X \notin \mathcal{I}$. Any minimal dependent set a \textbf{circuit}. An alternative way to define matroids is through circuits. The collection of circuits of a matroid satisfy some properties, and any collection of subsets which satisfy these properties will be the collection of circuits of a unique matroid (see Corollary 1.1.5 in \cite{10.5555/1197093}). There are many other cryptomorphic definitions for matroids. In this thesis, we will not concern ourselves with proving the equivalence between these definitions. In the next section, we will define a dual notion of circuits called bases.

\subsection{Bases} \label{sec:bases}

We call an independent set $B \in \mathcal{I}$ a \textbf{basis} if it is a maximal independent set. From the properties of independent sets of a matroid, we can deduce that all bases have the same number of elements. Indeed, if $B_1$ and $B_2$ are bases satisfying $|B_1| < |B_2|$, then there must exist some element $e \in B_2 \backslash B_1$ satisfying $B_1 \cup \{e\} \in \mcI$. But, this means that $B_1 \cup \{e\}$ is an independent set strictly larger than $B_1$. This contradicts the maximality of $B_1$ and implies that all bases contain the same number of elements. 

\begin{prop} \label{prop:matroid-bases}
	Let $M = (E, \mathcal{I})$ be a matroid and let $\mcB$ be the collection of bases. The collection $\mcB$ satisfies the following three properties:
	\begin{enumerate}[label = (\alph*)]
		\item $\mcB$ is non-empty.
		\item If $B_1$ and $B_2$ are members of $\mcB$ and $x \in B_1 \backslash B_2$, then there is an element $y$ of $B_2 \backslash B_1$ such that $(B_1 - x) \cup y \in \mcB$. 
		\item If $B_1$ and $B_2$ are members of $\mcB$ and $x \in B_1 \backslash B_2$, then there is an element of $y \in B_2 \backslash B_1$ such that $(B_2 - y) \cup x \in \mcB$. 
	\end{enumerate}
\end{prop}

\begin{proof}
	See Lemma 1.2.2 in \cite{10.5555/1197093}. 
\end{proof}

For any matroid $M = (E, \mcB)$ where $\mcB = \mcB(M)$ are the bases of $M$, we can define the \textbf{basis generating polynomial} of a matroid $M = (E, \mcB)$ by
\[
	f_M (x) := \sum_{B \in \mcB} x^B \in \RR[x_e : e \in E].
\]
The polynomial $f_M$ is a homomogeneous polynomial of degree $d$ where $d$ is the size of a basis in $M$. In Section~\ref{sec:rank}, we define the number $d$ as the \textbf{rank} or \textbf{dimension} of the matroid $M$. 

\subsection{Rank Functions} \label{sec:rank}

Recall the motivating example of a matroid as a subset of vectors $S \subseteq V$ where vectors are independent if and only if they are linearly independent. In this example, there is a natural notion of dimension or rank. For any set of vectors, we can define the rank of this set to be the dimension of the vector subspace spanned by these vectors. This defines a function from the subsets of $S$ to the non-negative integers with a few properties. We will abstract these properties in Definition~\ref{def:rank}. 

\begin{defn} \label{def:rank}
	For any matroid $M = (E, \mcI)$, we define its \textbf{rank} function $\rank_M : 2^E \to \NN$ to be equal to
	\[
		\rank_M(X) := \max \{|I| : I \in \mathcal{I}, I \subseteq X\}.
	\]
	When the matroid $M$ is clear from context, we will sometimes write $\rank := \rank_M$. The rank function satisfies the following three properties:
	\begin{enumerate}
		\item[(R1)] If $X \subseteq E$, then $0 \leq r(X) \leq |X|$. 
		\item[(R2)] If $X \subseteq Y \subseteq E$, then $r(X) \leq r(Y)$. 
		\item[(R3)] If $X$ and $Y$ are subsets of $E$, then $r(X \cup Y) + r(X \cap Y) \leq r(X) + r(Y)$.
	\end{enumerate}
\end{defn}

For a proof of (R1), (R2), and (R3) in Definition~\ref{def:rank}, we refer the reader to Lemma 1.3.1 in \cite{10.5555/1197093}. Many properties of rank functions which are satisfied in the case of linear matroids and the vector space picture are satisfied in general. For example, for any subset $S \subset E$ and $e \in E \backslash S$, we can prove that $\rank_M(S \cup e) \in \{\rank_M(S), \rank_M(S) + 1\}$. 

\subsection{Closure and Flats}

In the case of a linear matroid, the rank of a set of vectors is equal to the dimension of the vector subspace spanned by our vectors. We can then study our matroid through the subspaces spanned by its vectors. We define a subset of vectors to be closed if the span of these vectors contain no other vectors in our ground set. In this sense, the vector space spanned by a set of vectors is the closure of the set. We abstract the properties of this closure operation in Definition~\ref{def:closure}. 

\begin{defn} \label{def:closure}
	For any matroid $M = (E, \mcI)$, we define its closure operator $\clo_M : 2^E \to 2^E$ to be 
	\[
		\clo_M (X) := \overline{X} = \{x \in E : \rank (X \cup \{x\}) = \rank (X) \}.
	\]
	When the underlying matroid $M$ is clear from context, we also write $\clo := \clo_M$. The closure operator satisfies the following four properties:
	\begin{enumerate}
		\item[(C1)] If $X \subseteq E$, then $X \subseteq \clo_M (X)$. 
		\item[(C2)] If $X \subseteq Y \subseteq E$, then $\clo_M (X) \subseteq \clo_M (Y)$. 
		\item[(C3)] If $X \subseteq E$, then $\clo_M (\clo_M(X)) = \clo_M(X)$. 
		\item[(C4)] If $X \subseteq E$ and $x \in E$, and $y \in \clo_M(X \cup \{x\}) \backslash \clo_M(X)$, then $x \in \clo_M(X \cup \{y\})$. 
	\end{enumerate}
\end{defn}

For a proof of (C1), (C2), (C3), and (C4) in Definition~\ref{def:closure}, see Lemma 1.4.3 in \cite{10.5555/1197093}. From this definition, it is not difficult to check that the closure operation in a linear matroid returns the collection of all vectors contained in a given subspace. If $X \subseteq E$ satisfies $X = \clo_M(X)$, then we say $X$ is a \textbf{closed set} or \textbf{flat}. For any matroid $M = (E, \mcI)$, let $\mcL(M)$ denote the partially ordered set consisting of the flats of $M$ equipped with set inclusion. From Theorem 1.7.5 in \cite{10.5555/1197093}, we have that $\mcL(M)$ is a geometric lattice with join, meet, and rank function (as a graded poset) given by 
\begin{align*}
	X \vee Y & := \clo_M(X \cup Y) \\
	X \wedge Y & := X \cap Y \\
	\rank_{\mcL} (X) & := \rank_M(X).
\end{align*}
In fact, from Theorem 1.7.5 in \cite{10.5555/1197093}, a lattice is geometric if and only if it is the lattice of flats of a matroid. A geometric lattice determines a matroid up to simplification. We define the notion of simplification in Section~\ref{sec:loops-parallel-simple}.

\subsection{Loops, Parallelism, and Simplification} \label{sec:loops-parallel-simple}

Let $M = (E, \mcI)$ be a matroid. We say an element of the ground set $e \in E$ is a \textbf{loop} in $M$ if $\{e\}$ is a dependent set. Equivalently, $e \in E$ is a loop if $\rank (\{e\}) = 0$. In a graphic matroid, this corresponds to a loop in the underlying graph. We define $E_0$ as the set of loops. An element $e \in E$ in the ground set is called a \textbf{coloop} if it is contained in every basis. We say that two elements $x_1, x_2 \in E \backslash E_0$ are \textbf{parallel} if and only if $\{x_1, x_2\} \notin \mcI$. When this happens, we write $x_1 \sim_M x_2$. Equivalently, $x_1 \sim_M x_2$ if and only if $\rank (\{x_1, x_2\}) = 1$. For a thorough treatment of loops and parallelism, we refer the reader to Section 1.4 of \cite{welsh}.

\begin{prop} \label{parallelism-is-equivalence-relation}
	Let $M = (E, \mcI)$ be a matroid. Then $\sim_M$ is an equivalence relation on $E \backslash E_0$. 
\end{prop}

\begin{proof}
	For any $x \in E \backslash E_0$, we have $\rank (\{x, x\}) = \rank (\{x\}) = 1$ since $x \notin E_0$. Thus $x \sim_M x$. For $x, y \in E \backslash E_0$, we have $\rank (\{x, y\}) = \rank (\{y, x\})$. This proves that $\sim_M$ is reflexive. To prove transitivity, suppose that we have elements $x, y, z \in E$ that satisfy $x \sim_M y$ and $y \sim_M z$. If $x = y$ or $y = z$, then we automatically get $x \sim_M z$. Suppose that they are all distinct. Then, we have $\{x, y\}, \{y, z\} \not\in \mcI$. For the sake of contradiction, suppose that $\{x, z\} \in \mcI$. From (I3) applied to $\{x, z\}$ and $\{y\}$, we have that either $\{x, y\} \in \mcI$ or $\{y, z\} \in \mcI$. This is a contradiction. This proves that $\sim_M$ is transitive and is an equivalence relation. 
\end{proof}

From Proposition~\ref{parallelism-is-equivalence-relation}, for any non-loop $e \in E \backslash E_0$, we can consider its equivalence class $[e] := [e]_M$ under the equivalence relation $\sim_M$. We call the equivalence class $[e]$ the \textbf{parallel class} of $e$. Then, we can partition the ground set $E$ into $E = E_1 \sqcup E_2 \sqcup \ldots \sqcup E_s \sqcup E_0$ where $E_1, \ldots, E_s$ are the distinct parallel classes and $E_0$ are the loops. In Proposition~\ref{characterization-of-atoms-of-flats}, we give a characterization of the atoms of the lattice of flats in terms of the parallel classes and loops. 

\begin{prop} \label{characterization-of-atoms-of-flats}
	Let $M = (E, \mcI)$ be a matroid and let $e \in E$ be an element of the matroid that is not a loop. Then, $\overline{e} = [e] \cup E_0$.
\end{prop}

\begin{proof}
	Since any independent set contains no loops, we know that $E_0 \subseteq \overline{e}$. For any $f \in [e]$, we have that $\rank (\{e, f\}) = 1 = \rank (\{e\})$ by definition of $\sim_M$. This proves that $[e] \subseteq \overline{e}$. To prove the opposite inclusion, let $f \in \overline{e}$. Then $\rank(\{e, f\}) = 1$. If $f$ is a loop, then $f \in E_0$. Otherwise, $f \sim_M e$ and $f \in [e]$. This suffices for the proof of the proposition. 
\end{proof}

We call a matroid $\textbf{simple}$ if it contains no loops and no parallel elements. To every matroid $M$, we can associate a simple matroid $\widetilde{M}$ called the \textbf{simplification} of the matroid $M$. For any matroid $M = (E, \mcI)$, let $E(\widetilde{M})$ be the set of rank $1$ flats of $M$. In particular, if $x_1, \ldots, x_s$ are the representatives for the parallel classes in $M$, then we can define the ground set of $\widetilde{M}$ as 
\[
	E \left ( \widetilde{M} \right ) := \{[x_1], \ldots, [x_s]\}.	
\] 
We can define a map $\pi_M : E \backslash E_0 \to E(\widetilde{M})$ to be the map which sends $e \in E \backslash E_0$ to the rank one flat $\overline{e}$. For any $\alpha \in E(\widetilde{M})$, we define $\text{fiber}(\alpha) := \pi_e^{-1}(\alpha)$. In other words, this consists of the elements of $E$ in the parallel class $\alpha$. We can define the following collection of subsets of $E(\widetilde{M})$:
\[
	\mcI \left ( \widetilde{M} \right ) := \left \{ \{[x_{i_1}], \ldots, [x_{i_l}] \} : \{x_{i_1}, \ldots, x_{i_l}\} \in \mcI(M) \right \}.
\]
We would like $\mcI(\widetilde{M})$ to be a collection of independent sets for $\widetilde{M}$. To prove that this is the case, it suffices to prove Proposition~\ref{prop-parallelism-and-independence}. 

\begin{prop} \label{prop-parallelism-and-independence}
	Let $M = (E, \mcI)$ be a matroid and let $e \sim f$ be two parallel elements. If $I \in \mcI$ is an independent set which contains $e$, then $(I \backslash e) \cup f \in \mcI$. In other words, in an independent set we can freely replace elements by parallel ones without breaking the independence structure. 
\end{prop}

\begin{proof}
	By definition, $e$ and $f$ are not loops. By applying the exchange axiom for independent sets repeatedly for $f$ and $I$, we must have that $(I \backslash e) \cup f$ is independent. This is because we can never exchange $e$ from $I$ to the independent set containing $f$ due to the fact that $\{e, f\}$ is dependent. 
\end{proof}

As a consequence of Proposition~\ref{prop-parallelism-and-independence}, we know that $\mcI(\widetilde{M})$ provides a well-defined collection of indepedent sets on the ground set $E(\widetilde{M})$. Given any matroid $M = (E, \mcI)$, we define the matroid $\widetilde{M} = (E(\widetilde{M}), \mcI(\widetilde{M}))$ to be the \textbf{simplification} of $M$.

\begin{remark}
	According to James Oxley in \cite{10.5555/1197093}, the notation $\widetilde{M}$ for the simplification of a matroid is defunct. Currently, the convention is to use $\text{si}(M)$ for the simplification.  
\end{remark} 

\subsection{Restriction, Deletion, and Contraction} \label{sec:contraction-deletion}

Given a graph, there are well-defined notions of edge deletion and edge contraction. In this section, we generalize these graph operations to general matroids. When we apply this generalization of deletion and contraction to graphic matroids, we recover the graph-theoretic model of deletion and contraction. 

\begin{defn}[Restriction, Deletion, and Contraction] \label{contraction-deletion}
	Let $M = (E, \mcI)$ be a matroid and let $T \subseteq E$ be a subset. We call $M|_T$ the \textbf{restriction} of $M$ on $T$. This is defined as the matroid on $T$ with independent sets given by 
	\[
		\mcI (M|_T) := \{I \in \mcI (M) : I \subseteq T\}.
	\]
	We call the matroid $M \backslash T$ the \textbf{deletion} of $T$ from $M$. This is defined to be $M|_{E \backslash T}$, the restriction of $M$ on $E \backslash T$. Let $B_T$ be a basis of $M|_T$. We call $M / T$ the \textbf{contraction} of $M$ by $T$. This is the matroid on $E \backslash T$ with independent sets given by
	\begin{align*}
		\mcI(M/T) & := \{I \subset E \backslash T: I \cup B_T \in \mcI(M) \}.
	\end{align*}
\end{defn}

The definition of the matroid contraction is well-defined because it is independent from our choice of basis of $T$. In fact, if we know about matroid duals, we can define contraction without choosing a basis: the contraction $M / T$ is defined to be $(M^* \backslash T)^*$. We will not concern ourselves with the dual matroid and refer the interested reader to Section 3.1 in \cite{10.5555/1197093}. In Lemma~\ref{lem-contraction-deletion-basis-generating-polynomial}, we describe how contraction and deletion effect the basis generating polynomial. 

\begin{lem} \label{lem-contraction-deletion-basis-generating-polynomial}
	Let $M = (E, \mcB)$ be a matroid with basis generating polynomial $f_M$.
	\begin{enumerate}[label = (\alph*)]
		\item If $e \in E$ is a loop, then $\partial_e f_M = 0$. 
		\item If $e \in E$ is not a loop, then $\partial_e f_M = f_{M/e}$. 
		\item If $e \in E$ is not a coloop, then $f_M = x_e f_{M/e} + f_{M \backslash e}$
	\end{enumerate}
\end{lem}

\begin{proof}
	Since every basis contains no loops, we have $\partial_e f_M = 0$ whenever $e$ is a loop. The second claim follows from the fact that the remaining monomials in $\partial_e f_M$ will correspond to sets of the form $B \backslash e$ where $e \in B$ and $B$ is a basis of $M$. This is exactly the set of bases of $M / e$. For the third claim, this follows from the fact that the monomials in $f_M$ which do not contain $x_e$ will correspond to bases of $M$ which do not contain $e$. These are exactly the bases of $M \backslash e$. This implies that we can write $f_M = x_e p + f_{M \backslash e}$ where $p \in \RR[x_e : e \in E]$ and $p$ contains no monomial with $x_e$. Taking the partial derivative with respect to $x_e$, we get $p = f_{M / e}$ from (a). This suffices for the proof. 
\end{proof}

\subsection{Matroid Sum and Truncation}

Given two matroids, there is a notion of adding these two matroids to produce another. There is also a notion of truncating a matroid so that the rank lowers by $1$. The first operation is called the matroid sum of two matroids while the second operation is called the truncation of a matroid. \\

To describe the matroid sum, let $M = (E, \mathcal{I}_M)$ and $N = (F, \mathcal{I}_N)$ be matroids. We define the \textbf{matroid sum} of $M$ and $N$ to be the matroid $M \oplus N$ on the set $E \sqcup F$ such that the independent sets of $M\oplus N$ are sets of $E \sqcup F$ of the form $I \cup J$ where $I \in \mcI_M$ and $J \in \mcI_N$. In other words, we define
\[
	\mcI(M+N) := \{I \cup J : I \in \mcI(M), J \in \mcI(N)\}.
\]
It is not difficult to see that this produces a well-defined collection of independent sets on the set $E \sqcup F$. Now, we describe the truncation of a matroid. For a matroid $M = (E, \mcI)$, we define $TM$ to be the \textbf{truncation} of $M$. This is the matroid on the same ground set $E$ with independent sets given by 
\[
	\mcI(TM) := \{I \in \mcI (M) : |I| \leq \rank_M(M) - 1\}.
\]
This collection of sets inherits the properties of independent sets from $\mcI(M)$. Thus, our definition $TM$ gives a well-defined matroid. We can repeatedly apply our truncation operation to get a matroid $T^k(M)$ which lowers the rank of $M$ by $k$. 

\subsection{Regular Matroids} \label{sec:regular-matroids}

In this section, we study a subclass of matroids called regular matroids. As a preliminary definition, these are the matroids which are representable over any field. The content of Definition~\ref{def:regular-matroid} illustrates the many different equivalent definitions of regular matroids. 

\begin{defn} \label{def:regular-matroid}
	We say that a matroid $M$ is \textbf{regular} if it satisfies any of the following equivalent conditions:
	\begin{enumerate}[label = (\alph*)]
		\item $M$ is representable over any field. 
		\item $M$ is $\FF_2$ and $\FF_3$ representable. 
		\item $M$ is representable over $\FF_2$ and $k$ where $k$ is any field of characteristic other than $2$.
		\item $M$ is representable over $\RR$ by a totally unimodular matrix. 
	\end{enumerate}
\end{defn}

For a proof of the equivalences of the conditions in Definition~\ref{def:regular-matroid}, we refer the reader to Theorem 5.16 in \cite{10.5555/1197093}. In the definition, we define a \textbf{totally unimodular matrix} to be a real matrix for which every square submatrix has determinant in the set $\{0, \pm 1\}$. In this thesis, we use the terminology totally unimodular and unimodular interchangeably even though they have different meanings in the literature. The main property that we will use for regular matroids is (d) in Definition~\ref{def:regular-matroid}. Note that when we represent $M$ by a unimodular matrix over $\RR$, property (d) does not indicate how small the dimension of the ambient space can be made. Ideally, we want the column vectors of the matrix to lie in a $d$-dimensional real vector space where $d$ is the rank of the matroid. This is the minimum possible dimension of an ambient vector space for which a matroid can be embedded. Fortunately, Theorem~\ref{thm-representing-regular-in-d-dim} implies that we can achieve this minimum dimension.

\begin{thm}[Lemma 2.2.21 in \cite{10.5555/1197093}] \label{thm-representing-regular-in-d-dim}
	Let $\{e_1, \ldots, e_r\}$ be a basis of a matroid $M$ of non-zero rank. Then $M$ is regular if and only if there is a totally unimodular matrix $[I_r | D]$ representing $M$ over $\RR$ whose first $r$ columns are labelled, in order, $e_1, e_2, \ldots, e_r$. 
\end{thm}

Thus, for any regular matroid $M$ of rank $n$, there is an injection $v : E(M) \to \RR^n$ where the columns $\{v(e) : e \in M\}$ form a totally unimodular matrix. We call such a map a \textbf{unimodular coordinatization} of $M$. Not only does it assign unimodular coordinates to our matroid, but it does so in a way which minimizes the dimension of the ambient space. From Theorem~\ref{thm-representing-regular-in-d-dim}, we know that such a coordinatization exists for all regular matroids. 

\begin{prop} \label{graphic-matroid-is-regular-prop}
	Let $G = (V, E)$ be a graph and let $M = M(G)$ be the graphic matroid associated with $G$. Then, $M$ is also the linear matroid generated by the incidence matrix and reduced incidence matrix of the graph. As a consequence, graphic matroids are regular.
\end{prop}

\begin{proof}
	From Proposition~\ref{incidence-reduced-incidence-matrix-is-good}, it suffices to prove that the incidence matrix is totally unimodular. But this follows from Lemma 4.4 in \cite{fujishige}. 
\end{proof}

In the proof of Proposition~\ref{graphic-matroid-is-regular-prop}, we have used the fact that the incidence and reduced incidence matrix of a graph are totally unimodular. For any graphic matroid $M$, Proposition~\ref{prop:graphic-matroids-connected} gives us a connected graph $G$ so that $M \cong M(G)$. In this case, the rank of $M$ is $|V(G)| - 1$. Hence, the reduced incidence matrix of $G$ is a unimodular coordinatization of $M$.  

\section{Mixed Discriminants} \label{sec:mixed-discriminants}

In this section, we discuss a symmetric multilinear form called the mixed discriminant which arises as the polarization form of the determinant. The notion of mixed discriminants was introduced by Alexandrov in his paper \cite{aleksandrov} where he uses mixed discriminants to study the Alexandrov-Fenchel inequality. Our treatment of mixed discriminants is inspired by the exposition in \cite{bapat_raghavan_1997}. Using Cholesky factorization, we will show how to compute mixed discriminants for rank $1$ matrices. This will allow us to extend the computation to all positive semi-definite matrices. 

\begin{defn} \label{defn-mixed-discriminant}
	let $n \geq 1$ be a positive integer. Suppose that for each $k \in [n]$, we are given real matrix $A_k := \left (a_{ij}^{(k)} \right )_{i, j = 1}^n \in \RR^{n \times n}$. Then, we define the \textbf{mixed discriminant} of the collection $(A_1, \ldots, A_n)$ to be 
	\begin{equation} \label{eqn:mixed}
		\mathsf{D}(A^1, \ldots, A^n) := \frac{1}{n!} \sum_{\sigma \in S_n} 
		\det 
		\begin{bmatrix} 
			a_{11}^{\sigma(1)} & \ldots & a_{1n}^{\sigma(n)} \\
			\vdots & \ddots & \vdots \\
			a_{n1}^{\sigma(1)} & \ldots & a_{nn}^{\sigma(n)}
		\end{bmatrix} = \frac{1}{n!} \sum_{\sigma \in S_n} \text{Det} \left (v^{\sigma(1)}_{1}, \ldots, v_{n}^{\sigma(n)} \right ).
	\end{equation}
	In Equation~\ref{eqn:mixed}, the group $S_n$ is the symmetric group on $n$ letters and $\text{Det}$ refers to the determinant as a multilinear $n$-form on $(\RR^n)^n$. 
\end{defn}

From Definition~\ref{defn-mixed-discriminant}, the mixed discriminant is multilinear and symmetric in its entries. Moreover, for any symmetric matrix $A$, we have $\mathsf{D}(A, \ldots, A) = \det (A)$. 

\subsection{Polarization Form of a Homogeneous Polynomial} \label{sec:polarization-form}

Let $f \in \RR[x_1, \ldots, x_n]$ be an arbitrary homogeneous polynomial of degree $d$. For any choice of vectors $v_1, \ldots, v_d \in \RR^n$ we can study the coefficients of $f(\lambda_1 v_1 + \ldots + \lambda_d v_d)$ as a polynomial in $\lambda_1, \ldots, \lambda_d$. In Definition~\ref{def:polarization-form}, we define the polarization form associated to a homogeneous polynomial. In the literature, this form is also called the complete homogeneous form. These objects are briefly mentioned in Section 3.2 of \cite{lie-groups}, Section 5.5 of \cite{schneider_2013}, and Section 4.1 of \cite{lorentzian-polynomials}. In this section, we hope to provide an accessible and self-contained exposition of the properties of the polarization form.

\begin{defn} \label{def:polarization-form}
	Let $k$ be a field of characteristic $0$. Let $f \in k[x_1, \ldots, x_n]$ be a homogeneous polynomial of degree $d$. We define the \textbf{polarization form} or \textbf{complete homogeneous form} of $f$ to be the function $F_f : (k^n)^d \to k$ defined by 
	\[
		F_f (v_1, \ldots, v_d) := \frac{1}{d!} \frac{\partial}{\partial x_1} \ldots \frac{\partial}{\partial x_d} f(x_1 v_1 + \ldots + x_d v_d).
	\] 
\end{defn}

From the definition, we can see that the form is $k$-multilinear, symmetric, and $F_f(v, \ldots, v) = f(v)$ for all $v \in k^n$. Let $f \in k[x_1, \ldots, x_n]$ be a homogeneous polynomial of degree $d$. Then, we can write this polynomial in the form 
\[
	f(x_1, \ldots, x_n) = \sum_{\alpha_1, \ldots, \alpha_d = 1}^n c_{\alpha_1, \ldots, \alpha_d} \cdot x_{\alpha_1} \ldots x_{\alpha_d}.
\]
where $c_{\alpha_1, \ldots, \alpha_d}$ is symmetric in $\alpha_1, \ldots, \alpha_d$. Let $v_1, \ldots, v_d\in k^n$ be vectors where $v_i = v_1^{(i)} e_1 + \ldots + v_n^{(i)} e_n$ for all $1 \leq i \leq d$ and $e_1, \ldots, e_n$ is the standard basis in $k^n$. We can define $y_i := \sum_{j = 1}^d x_j v_i^{(j)}$ for all $1 \leq i \leq n$. Then, we have $f(x_1 v_1 + \ldots + x_d v_d) = f(y_1, \ldots, y_n)$. This allows us to compute
\begin{align*}
	F_f(v_1, \ldots, v_d) & = \frac{1}{d!}[x_1 \ldots x_d] f(y_1, \ldots, y_n) \\
	& = \frac{1}{d!} [x_1 \ldots x_d] \sum_{\alpha_1, \ldots, \alpha_d = 1}^n c_{\alpha_1, \ldots, \alpha_d} y_{\alpha_1} \ldots y_{\alpha_d} \\
	& = \frac{1}{d!} \sum_{\sigma \in S_n} \sum_{\alpha_1, \ldots, \alpha_d = 1}^n c_{\alpha_{\sigma(1)}, \ldots, \alpha_{\sigma(d)}} v_{\alpha_{\sigma(1)}}^{(1)} \ldots v_{\alpha_{\sigma(d)}}^{(d)} \\
	& = \sum_{\alpha_1, \ldots, \alpha_d = 1}^n c_{\alpha_1, \ldots, \alpha_d} v_{\alpha_1}^{(1)} \ldots v_{\alpha_d}^{(d)}.
\end{align*}

This gives us a formula for the polarization form explicitly in the vectors $v_1, \ldots, v_d$ and the coefficients of $f$. Using this explicit formula, we prove the well-known polarization formula given in Theorem~\ref{polariziation-identity}.

\begin{thm}[Polarization Identity] \label{polariziation-identity}
	Let $k$ be a field of characteristic $0$. Let $v_1, \ldots, v_m \in k^n$ be arbitrary vectors. Then, we have the identity
	\[
		f(x_1 v_1 + \ldots + x_m v_m ) = \sum_{i_1, \ldots, i_d = 1}^m F_f(v_{i_1}, \ldots, v_{i_d}) \cdot x_{i_1} \ldots x_{i_d}.
	\]
\end{thm}
\begin{proof}
	For all $1 \leq i \leq n$, we have $v_i = v_i^{(1)} e_1 + \ldots + v_i^{(n)} e_n$ for some constants $v_i^{(j)}$. Let $y_i = \sum_{j = 1}^d x_j v_i^{(j)}$ for all $1 \leq i \leq n$. Then, we have that 
	\begin{align*}
		f(x_1 v_1 + \ldots + v_m v_m) & = f(y_1, \ldots, y_m) \\
		& = \sum_{\alpha_1, \ldots, \alpha_d = 1}^n c_{\alpha_1, \ldots, \alpha_d} \sum_{i_1, \ldots, i_d = 1}^m x_{i_1} \ldots x_{i_d} \cdot v_{\alpha_1}^{(i_1)} \ldots v_{\alpha_d}^{(i_d)} \\
		& = \sum_{i_1, \ldots, i_d = 1}^m \left ( \sum_{\alpha_1, \ldots, \alpha_d = 1}^n c_{\alpha_1, \ldots, \alpha_d} v_{\alpha_1}^{(i_1)} \ldots v_{\alpha_d}^{(i_d)}\right ) x_{i_1} \ldots x_{i_d} \\
		& = \sum_{i_1, \ldots, i_d = 1}^m F_f(v_{i_1}, \ldots, v_{i_d}) \cdot x_{i_1} \ldots x_{i_d}.
	\end{align*}
	This suffices for the proof. 
\end{proof}

\subsection{Polarization for Mixed Discriminants}

We can view the mixed discriminant as the polarization form of the determinant function. This fact is the content of Theorem~\ref{polarization-mixed-discriminant}. For proofs of this result in other sources, we refer the reader to \cite{Zhao2015-kf} or \cite{schneider_2013}. 

\begin{thm} \label{polarization-mixed-discriminant}
	For $n \times n$ matrices $A_1, \ldots, A_m$ and $\lambda_1, \ldots, \lambda_m \in \RR$, the determinant of the linear combination $\lambda_1 A_1 + \ldots + \lambda_m A_m$ is a homogeneous polynomial of degree $n$ in the $\lambda_i$ and is given by 
	\[
		\det (\lambda_1 A_1 + \ldots + \lambda_m A_m) = \sum_{1 \leq i_1, \ldots, i_n \leq m} \mathsf{D}(A_{i_1}, \ldots, A_{i_n}) \lambda_{i_1} \ldots \lambda_{i_n}. 
	\]
\end{thm}

\begin{proof}
	If $v_1^{(i)}, \ldots$ and $v_n^{(i)}$ are the columns of $A_i$ for $1 \leq i \leq m$, we have that 
	\begin{align*}
		\det \left ( \sum_{i = 1}^m \lambda_i A_i \right ) & = \text{Det} \left ( \sum_{i = 1}^m \lambda_i v_1^{(i)}, \ldots, \sum_{i = 1}^m \lambda_i v_n^{(i)} \right ) \\
		& = \sum_{i_1, \ldots, i_n = 1}^m \text{Det} (\lambda_{i_1} v_{1}^{(i_1)}, \ldots, \lambda_{i_n} v_{n}^{(i_n)} ) \\
		& = \sum_{i_1, \ldots, i_n = 1}^m \lambda_{i_1} \ldots \lambda_{i_n} \cdot \text{Det}(v_{1}^{(i_1)}, \ldots, v_{n}^{(i_n)}).
	\end{align*}
	Looking at the coefficient in front of $\lambda_1^{r_1} \ldots \lambda_m^{r_m}$ where $r_1 + \ldots + r_m = n$, it is equal to 
	\begin{align*}
		[\lambda_1^{r_1} \ldots \lambda_M^{r_m}] \det \left ( \sum_{i = 1}^m \lambda_i A_i \right ) & = \frac{1}{(r_1)! \ldots (r_m)!} \sum_{\sigma \in S_n} \text{Det}(v_1^{i_{\sigma(1)}}, \ldots, v_n^{i_{\sigma(n)}}) \\
		& = \binom{n}{r_1,\ldots, r_m} \mathsf{D}(A_1[r_1], \ldots, A_m[r_m])
	\end{align*}
	where the multiset $\{i_1, \ldots, i_m\}$ is equal to $\{1[r_1], \ldots, m[r_m]\}.$ This coincides with the right hand side in Theorem~\ref{polarization-mixed-discriminant}.
\end{proof}

As an application of the polarization identity, we will compute the mixed discriminants of rank $1$ matrices. Since positive semi-definite matrices can be written as the sum of rank 1 matrices, this computation will extend to the mixed discriminants of positive semi-definite matrices. 

\begin{example}[Mixed discriminants of rank 1 matrices] \label{mixed-discriminant-calculation-for-rank-1}
	Let $x_1, \ldots, x_n \in \RR^n$ be real vectors. For any $\lambda_1, \ldots, \lambda_n > 0$, define the vectors $y_i = \sqrt{\lambda_i} x_i$. Let $X$ be the matrix with the $x_i$ as column vectors and let $Y$ be the matrices with the $y_i$ as column vectors. Then, we have that 
	\[
		\det \left ( \sum_{i = 1}^n \lambda_i x_i x_i^T \right ) = \det \left ( \sum_{i = 1}^n y_i y_i^T \right ) = \det (YY^T) = (\det (Y))^2 = \lambda_1 \ldots \lambda_n (\det (X))^2.
	\]
	From Theorem~\ref{polarization-mixed-discriminant}, we get the identity:
	\begin{equation} \label{eqn:mixed-discriminant-of-rank-1-matrices}
		\mathsf{D} (x_1x_1^T, \ldots, x_nx_n^T) = \frac{1}{n!} \left [ \text{Det} (x_1, \ldots, x_n) \right ]^2.
	\end{equation}
	From Equation~\ref{eqn:mixed-discriminant-of-rank-1-matrices}, we see that the mixed discriminant of rank 1 matrices is exactly a determinant of vectors generating the rank 1 matrices. In particular, it can serve as an indicator for when a collection of vectors forms a basis. This fact will be used in Section~\ref{sec:mixed-discriminant-stanley-basis} when we study Stanley's matroid inequality. 
\end{example}

\subsection{Positivity}

To extend the calculation in Example~\ref{mixed-discriminant-calculation-for-rank-1} to positive semi-definite matrices, we first recall the following linear algebra factorization result. 

\begin{thm} [Cholesky Factorization, Theorem 4.2.5 in \cite{GoluVanl96}] \label{cholesky}
	If $A \in \RR^{n \times n}$ is a symmetric positive definite matrix, then there exists a unique lower triangular $L \in \RR^{n \times n}$ with positive diagonal entries such that $A = LL^T$. When $A$ is positive semi-definite, then there exists a (not necessarily unique) lower triangular $L \in \RR^{n \times n}$ with $A = LL^T$. 
\end{thm}

Let $A$ be a positive semi-definite matrix. Theorem~\ref{cholesky} implies that there exists some matrix $X \in \RR^{n \times n}$ satisfying $A = X X^T$. We can decompose $X$ into $X = X_1 + \ldots + X_n$ where $X_i$ is the matrix with $x_i$ in the $i$th column and $0$ everywhere else. These matrices satisfy the properties that $X_i X_i^T = x_ix_i^T$ and $X_i X_j^T = 0$ when $i \neq j$. Thus, we have that 
\[
	A = X X^T = \left ( \sum_{i = 1}^n X_i \right ) \left ( \sum_{i = 1}^n X_i^T \right ) = \sum_{i = 1}^n x_ix_i^T.
\]
We have just proved that all positive semi-definite matrices can be written as the sum of rank $1$ matrices of the form $xx^T$. In Lemma~\ref{mixed-discriminant-explicit-formula}, we will use this fact to give an explicit formula for the mixed discriminant on positive semi-definite matrices. 

\begin{lem} [Lemma 5.2.1 in \cite{bapat_raghavan_1997}] \label{mixed-discriminant-explicit-formula}
	Let $A_1, \ldots, A_n$ be positive semi-definite $n \times n$ matrices, and suppose that $A_k = X_k X_k^T$ for each $k$. Then
	\[
		\mathsf{D} (A_1, \ldots, A_n) = \frac{1}{n!} \sum_{\substack{x_j \in X_j \\ 1 \leq j \leq n}} \left [ \Det(x_1, \ldots, x_n) \right ]^2
	\]
	where $x_j \in X_j$ means that $x_j$ is taken over the columns of $X_j$.
\end{lem}

\begin{proof}
	From the multi-linearity of the mixed discriminant, we have that 
	\[
		\mathsf{D}(A_1, \ldots, A_n) = \mathsf{D} \left ( \sum_{x_1 \in \mathsf{col}(X_1)} x_1x_1^T, \ldots, \sum_{x_n \in \mathsf{col}(X_n)} x_nx_n^T \right ) = \sum_{\substack{x_j \in X_j \\ 1 \leq j \leq n}} \mathsf{D} (x_1x_1^T, \ldots, x_nx_n^T).
	\]
	The lemma follows from the computation in Example~\ref{mixed-discriminant-calculation-for-rank-1}. 
\end{proof}

\begin{cor} \label{final-positivity-corollary}
	Let $A_1, \ldots, A_n$ be positive semi-definite $n \times n$ real symmetric matrices. Then 
	\[
		\mathsf{D}(A_1, \ldots, A_n) \geq 0.
	\]	
	If $A_1, \ldots, A_n$ are positive definite, then $\mathsf{D}(A_1, \ldots, A_n) > 0$. 
\end{cor}

\begin{proof}
	This follows from Lemma~\ref{mixed-discriminant-explicit-formula} and Theorem~\ref{cholesky}. 
\end{proof}

In Lemma~\ref{lem:mixed-discriminant-properties}, we compile a few properties of mixed discriminants that will prove useful in future chapters. For example, in the proof of Theorem~\ref{thm:only-inequality-mixed-discriminants}, the property Lemma~\ref{lem:mixed-discriminant-properties}(c) is used as an inductive tool. 

\begin{lem}[Properties of Mixed Discriminants] \label{lem:mixed-discriminant-properties}
	Let $M, M_1, \ldots, M_n$ be $n$-dimensional real symmetric matrices. Then, the following properties are true. 
	\begin{enumerate}[label = (\alph*)]
		\item $\mathsf{D}(M, \ldots, M) = \det (M)$. 
		\item $\mathsf{D}(U M_1 U^T, \ldots, U M_n U^T) = \det (UU^T) \mathsf{D}(M_1, \ldots, M_n)$ for any $n$ dimensional real matrix $U$.
		\item $\mathsf{D}(e_ie_i^T, M_1, \ldots, M_{n-1}) = \frac{1}{n} \mathsf{D}(M_1^{\langle i \rangle}, \ldots, M_{n-1}^{\langle i \rangle})$ where $e_1, \ldots, e_n$ is the standard orthonormal basis in $\RR^n$ and $M^{\langle i \rangle}$ is obtained from $M$ by removing its $i$th row and $i$th column. 
	\end{enumerate}
\end{lem}

\begin{proof}
	See Lemma 2.6 in \cite{bochner}. 
\end{proof}

\section{Convex Bodies} \label{sec:convex-bodies}

In this section, we review the notions of convexity and convex bodies. We use \cite{schneider_2013} as our main reference for the theory of convex bodies and Brunn-Minkowski theory. Recall that a subset $C \subseteq \RR^n$ is convex if for every $x, y \in C$, the line segment $[x, y]$ is contained in $C$. Even though convexity is quite a rigid condition to impose, there still exist topologically wild convex sets. For example, let $S \subseteq \mathbb{S}^{n-1}$ be an arbitrary subset of the unit sphere. Then, the set $S \cup \{x : \norm{x}_2 < 1\}$ is always a convex set. Measure theoretically, this same example gives examples of convex sets which are not even Borel measurable. In this thesis, we will only consider a subclass of convex sets called convex bodies. A \textbf{convex body} is a non-empty, compact, convex subset of $\RR^n$. Let $\mathsf{K}^n$ be the space of convex bodies in $\RR^n$. We can define an associative, commutative binary operation on the set of convex bodies called the Minkowski sum. 

\begin{defn}
	For any convex bodies $K, L \subseteq \RR^n$, we define the \textbf{Minkowski sum} of the two bodies to be the convex body
	\[
		K + L := \{x + y \in \RR^n : x \in K, y \in L\}.
	\]
\end{defn}

For any $x_0, x_1 \in K$ and $y_0, y_1 \in K$ we have that $\lambda x_0 + \tau x_1 \in K$ and $\lambda y_0 + \tau y_1 \in L$ for any $\lambda, \tau \geq 0$ satisfying $\lambda + \tau = 1$. Thus, we have that 
\[
	\lambda (x_0 + y_0) + \tau (x_1 + y_1) = (\lambda x_0 + \tau x_1) + (\lambda y_0 + \tau y_1) \in K + L.
\]
This proves that $K + L$ is a well-defined convex body. In fact, a subset $K \subseteq \RR^n$ is a convex body if and only if $\alpha K + \beta K = (\alpha + \beta) K$ for all $\alpha, \beta \geq 0$. In the next section, we discuss the boundary structure of convex sets.  

\subsection{Support and Facial Structure}

For a convex body $K \subseteq \RR^n$, we define its dimension as $\dim K := \dim \aff K$. The boundary of convex bodies can be characterized in terms of supporting hyperplanes and faces. We call $H$ a \textbf{supporting hyperplane} of $K$ if $K$ lies on one side of the hyperplane and $K \cap H \neq \emptyset$. A \textbf{supporting half-space} is the half-space of a supporting hyperplane which contains $K$. Any convex body will be equal to the intersection of all of its supporting hyperspaces. To explore the boundary structure in greater detail, we define the notion of a face and an exposed face. 
\begin{defn} \label{face-and-exposed-face}
	Let $K \subseteq \RR^n$ be a convex body and let $F \subseteq K$ be a subset. 
	\begin{enumerate}[label = (\alph*)]
		\item If $F$ is a convex subset such that each segment $[x, y] \subseteq K$ with $F \cap \relint [x, y] \neq \emptyset$ is contained in $F$, then we call $F$ a \textbf{face}. If $\dim F = i$, then we call $F$ an $i$-face. 

		\item If there is a supporting hyperplane $H$ such that $K \cap H = F$, we call $F$ an \textbf{exposed face}. The exposed faces of codimension $1$ are called \textbf{facets} and the exposed faces of dimension $0$ are called \textbf{vertices}. 
	\end{enumerate}
	For $i \geq -1$, we define $\mathcal{F}_i(K)$ be the set of $i$-faces of $K$. By convention, we let $\mathcal{F}_{-1}(K) = \{\emptyset\}$ and consider $\emptyset$ to be the unique face of dimension $-1$. The set of all faces $\mathcal{F}(K) := \bigcup_i \mathcal{F}_i(K)$ equipped with set inclusion forms a poset. 
\end{defn}

\begin{remark}
	In general, the notions of faces and exposed faces are not the same. It is not hard to check that an exposed face is a face. On the other hand, a face is not necessarily an exposed face. For example, in Figure~\ref{counter-example} the top semi-circle is a face which is not exposed. In Section~\ref{sec:polytopes}, we study polytopes. Polytopes have the property that their exposed faces and faces are exactly the same.
	\begin{figure}[h] \label{counter-example}
		\begin{center}
			\includegraphics[scale = 1]{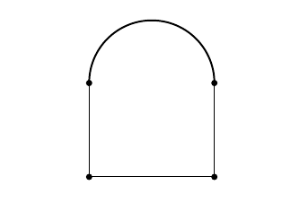}
			\caption{Convex body with a face which is not exposed.}
		\end{center}
	\end{figure}
\end{remark}
For any convex body $K \in \mathsf{K}^n$ and $u \in \RR^n \backslash \{0\}$, we can define the \textbf{support function} $h_K(u)$ of $K$ in the direction $u$ and the \textbf{exposed face} $F_K(u)$ of $K$ in the direction $u$ as 
\begin{align*}
	h_K(u) & := \sup_{x \in K} \langle u, x \rangle \\
	F_K(u) & := K \cap \{x \in \RR^n : \langle u, x \rangle = h_K(u)\}.
\end{align*}
It is not difficult to check that $F_K(u)$ is an exposed face of $K$ for any direction $u$. For any $K, L \subseteq \RR^n$ and $\lambda \geq 0$, the support function and exposed face function satisfy the equation
\[
	h_{\lambda K+L} = \lambda h_K + h_L, \quad \text{and} \quad F_{\lambda K+L} = \lambda F_K + F_L.
\]
Geometrically, for a unit vector $u$, the support value $h_K(u)$ is the (signed) distance of the furthest point on $K$ in the direction $u$. The exposed face $F_K(u)$ consists of the subset of $K$ which achieves this maximum distance in the direction $u$. The support function $h_K$ of a convex body $K$ completely determines the convex body because of the equation 
\[
	K = \bigcap_{u \in \mathbb{S}^{n-1}} H^{-}_{u, h_K(u)} = \bigcap_{u \in \mathbb{S}^{n-1}} \{x \in \RR^n : \langle x, u \rangle \leq h_K(u)\}.
\]
Thus, from the properties of support functions, we have the following cancellation law in Proposition~\ref{prop-cancellation-law}.

\begin{prop} [Cancellation Law for Minkowski Addition] \label{prop-cancellation-law}
	Let $K, L, M \subseteq \RR^n$ be convex bodies such that $K+M = L+M$. Then $K = L$. 
\end{prop}

\begin{proof}
	Since $K+M = L+M$, we have that $h_K + h_M = h_{K+M} = h_{L+M} = h_L + h_M$. So, we have that $h_K = h_L$ and from our remark that the support function determines the convex body we have $K = L$.
\end{proof}

Proposition~\ref{prop-cancellation-law} gives the set of convex bodies the structure of a commutative, associative monoid with a cancellation law.

\subsection{Polytopes and SSI Polytopes} \label{sec:polytopes}

In this section, we define a subclass of convex bodies which are combinatorial in nature and allow us to approximate general convex bodies. We call a convex body $P \subseteq \RR^n$ a \textbf{polytope} if it can be written as the convex hull of a finite number of points. 

\begin{prop}[Properties of Polytopes]
	Let $P \subseteq \RR^n$ be an arbitrary polytope. Then, $P$ satisfies the following properties: 
	\begin{enumerate}[label = (\alph*)]
		\item The exposed faces of $P$ are exactly the faces of $P$. 
		\item $P$ has a finite number of faces. 
		\item Let $F_1, \ldots, F_k$ be the facets of $P$ with normal vectors $u_1, \ldots, u_k$. Then 
		\[
			P = \bigcap_{i = 1}^k H_{u_i, h_P(u_i)}^-
		\]
		In particular, the numbers $h_P(u_1), \ldots, h_P(u_k)$ determine $P$ uniquely. 

		\item The face poset $\mathcal{F}(P)$ of $P$ is a graded lattice where the poset rank function is the dimension of the corresponding face. The face lattice satisfies the \textbf{Jordan-Dedekind chain condition}. In other words, if we have a face $F^j \in \mathcal{F}_j(P)$ and a face $F^k \in \mathcal{F}_k(P)$ satisfying $F^j \subset F^k$, then there are faces $F^i \in \mathcal{F}_i(P)$ for $j+1 \leq i \leq k-1$ such that 
		\[
			F^j \subset F^{j+1} \subset \ldots \subset F^{k-1} \subset F^k.
		\]
	\end{enumerate}
\end{prop}

\begin{proof}
	These properties follow from Corollary 2.4.2, Theorem 2.4.3, Corollary 2.4.4, and Corollary 2.4.8 in \cite{schneider_2013}. 
\end{proof}

We now give a few examples of special polytopes and their combinatorial properties. The calculations performed in these examples will reappear in future chapters. 

\begin{example}[Order Polytope]
	Let $(P, \leq)$ be a finite poset. Let $\RR^P$ be Euclidean space of dimension $|P|$ where the coordinates are indexed by the elements of $P$. Every vector $v \in \RR^P$ can be written in the form $v = \sum_{\omega \in P} v_\omega \cdot e_\omega$ where $\{e_\omega\}_{\omega \in P}$ is a standard orthonormal basis for $\RR^P$. Then, we can define the \textbf{order polytope} of $P$ to be the polytope given by
	\[
		\mcO_P := \{x \in [0, 1]^P : x_{i} \leq x_{j} \text{ whenever } i \leq j \text{ in } P\}.
	\]
	Given any linear extension $\sigma : P \to [n]$, we can define the \textbf{linear extension simplex}
	\[
		\Delta_l := \{x \in [0, 1]^n : 0 \leq x_{l^{-1}(1)} \leq \ldots \leq x_{l^{-1}}(n) \leq 1\}.
	\]
	Let $e(P)$ be the set of linear extensions of $P$. Then, we can triangulate the order polytope with linear extension simplicies as
	\[
		\mcO_P := \bigsqcup_{l \in e(P)} \Delta_l.
	\]
	In the union, the simplices are disjoint except possibly in a set of measure zero. This implies that $\Vol_n(\mcO_P) = e(P) / (n!)$. When $P$ is the poset on $n$ elements with an empty partial order, we recover the familiar fact that the $\Vol_n(\Delta) = 1/n!$ for a simplex $\Delta$ whose vertices lie in $\{0, 1\}^n$ and no two vertices lie in the same affine hyperplane orthogonal to $(1, \ldots, 1)$.  
\end{example}

\begin{example}[Zonotope] \label{zonotope-example}
	For any vectors $v_1, \ldots, v_l \in \RR^n$, we can define the \textbf{zonotope} generated by these vectors by 
	\[
		Z(v_1, \ldots, v_l) := [0, v_1] + \ldots + [0, v_l] = \left \{ \sum_{i = 1}^l \lambda_i v_i : \lambda_i \in [0, 1] \right \}.
	\]
\end{example}
We call a polytope $P \subseteq \RR^n$ \textbf{simple} if $\text{int} (P) \neq \emptyset$ and each of its vertices is contained in exactly $n$ facets. We say two polytopes $P_1, P_2$ are called \textbf{strongly isomorphic} if $\dim F_{P_1}(u) = \dim F_{P_2}(u)$ for all $u \in \RR^n \backslash \{0\}$. Strong isomorphism is an equivalence relation between polytopes. Two strongly isomorphic polytopes have isomorphic face lattices with corresponding faces being parallel to each other. To illustrate the strength of this equivalence relation, it can be shown that given two strongly isomorphic polytopes, the corresponding faces are also strongly isomorphic. This follows from the fact that the faces of a polytope come from the intersections of facets. 

\begin{lem} \label{family-of-strongly-isomorphic}
	If $P_1, P_2$ are strongly isomorphic polytopes, then for each $u \in \RR^n \backslash \{0\}$, the faces $F_{P_1}(u)$ and $F_{P_2}(u)$ are strongly isomorphic. 
\end{lem}
\begin{proof}
	See Lemma 2.4.10 in \cite{schneider_2013}. 
\end{proof}

Given two polytopes, it is not difficult to construct an infinite family of polytopes which are in the same strong isomorphism class. Indeed, any pair of positively-weighted Minkowski sums of two polytopes are strongly isomorphic. 

\begin{prop} \label{families-of-strongly-isomorphic}
	If $P_1, P_2 \subseteq \RR^n$ are polytopes which are strongly isomorphic, then $\lambda_1 P_1 + \lambda_2 P_2$ with $\lambda_1, \lambda_2 > 0$ are strongly isomorphic. If $P_1, P_2$ are strongly isomorphic, then the family of polytopes $\lambda_1 P_1 + \lambda_2 P_2$ with $\lambda_1, \lambda_2 \geq 0$ and $\lambda_1 + \lambda_2 > 0$ is strongly isomorphic.
\end{prop}

\begin{proof}
	See Corollary 2.4.12 in \cite{schneider_2013}.
\end{proof}

Any polytope is determined by its facets. The facets of a polytope are determined by their normal vectors and the polytope's support values in the direction of these normals. Let $\alpha$ be a strong isomorphism class of polytopes. All of the polytopes in this isomorphism class will share the same facet normals $\mathcal{U}$. Hence, all of the polytopes in $\alpha$ are \textit{uniquely} determined by their support values at each of the facet normals. In other words, there is injective map $h : \alpha \to \RR^\mathcal{U}$ from the strong isomorphism class to the finite-dimensional vector space $\RR^{\mathcal{U}}$ defined by 
\[
	h_P := h(P) = \left ( H^-_{u, h_P(u)}\right )_{u \in \mathcal{U}}.
\]
For a polytope $P \in \alpha$, we call $h_P := h(P)$ the \textbf{support vector} of $P$. Clearly, the support vector map is not surjective. For example, if we pick a vector in $\RR^\mathcal{U}$ in which all of the coordinates of $x \in \RR^{\mathcal{U}}$ sufficiently negative, there is no corresponding polytope with those support values. In Lemma~\ref{perturbation-SSI} we prove that simple and strongly isomorphic (SSI) polytopes are robust under small perturbations. That is, after perturbing a polytope in $\alpha$ by a sufficiently small amount, it will remain in $\alpha$. As a corollary, we prove that any vector in $\RR^{\mathcal{U}}$ can be written as a scalar multiple of the difference of two support vectors. 

\begin{lem} \label{perturbation-SSI}
	Let $P$ be a simple $n$-polytope with facet normals $\mathcal{U}$. Then, there is a number $\beta > 0$ such that every polytope of the form 
	\[
		P' := \bigcap_{u \in \mathcal{U}} H_{u, h_P(u) + \alpha_u}^-
	\]
	with $|\alpha_u| \leq \beta$ is simple and strongly isomorphic to $P$. 
\end{lem}

\begin{proof}
	See Lemma 2.4.14 in \cite{schneider_2013}. 
\end{proof}

\begin{cor} [Lemma 5.1 in \cite{bochner}]\label{difference-of-support-vectors}
	Let $\alpha$ be the strong isomorphism class of a simple polytope $P$ with facet normals $\mathcal{U}$. For any $x \in \RR^{\mathcal{U}}$ there are $a > 0$ and $Q \in \alpha$ such that $x = a (h_Q - h_P)$. 
\end{cor}

\begin{proof}
	From Lemma~\ref{perturbation-SSI} there exists $a^{-1} > 0$ sufficiently small and $Q \in \alpha$ such that $h_Q = a^{-1}(x + h_P)$. By rearranging the equation, we get $x = a(h_Q - h_P)$. 
\end{proof}

Given a collection of convex bodies, we will be interested in approximating all of these polytopes simultaneously by SSI polytopes. In order to have a well-defined notion of approximation by polytopes, we must equip the convex bodies with a metric structure. In Section~\ref{sec:hausdorff-metric}, we will define a metric on the space of convex bodies called the Hausdorff metric. 

\subsection{Hausdorff Metric on Convex Bodies} \label{sec:hausdorff-metric}

We define a metric $\delta$ called the \textbf{Hausdorff metric} on $\mathsf{K}^n$ such that for any pair of elements $K, L \in \mathsf{K}^n$, we have 
\begin{align*}
	\delta (K, L) & = \max \left \{ \sup_{x \in K} \inf_{y \in L} |x-y|, \sup_{y \in L} \inf_{x \in K} |x-y|\right \} \\
	& = \inf \{\varepsilon \geq 0 : K \subseteq L + \varepsilon B^n, L \subseteq K + \varepsilon B^n\} \\
	& = \norm{h_K - h_L}_\infty.
\end{align*}
For a proof of the equivalence of these three descriptions, we refer the reader to the proof of Theorem 3.2 in \cite{Hug2020-ue}. In Proposition~\ref{convex-body-is-metric-space}, we prove that $\delta$ is a metric on the space of convex bodies. Theorem~\ref{approximation-SSI} implies that any set of convex bodies can be approximated by SSI polytopes.

\begin{prop} \label{convex-body-is-metric-space}
	The ordered pair $(\mathsf{K}^n, \delta)$ is a metric space. 
\end{prop}

\begin{proof}
	For $K, L, M \in \mathsf{K}^n$, we have 
	\[
		\delta(K, M) = \norm{h_K - h_M}_\infty \leq \norm{h_K - h_L}_\infty + \norm{h_L - h_M}_\infty = \delta(K, L) + \delta(L, M). 
	\]
	It is clear that $\delta (K, L) = \delta (L, K)$. Finally, we have $\delta(K, L) = 0$ if and only if $\norm{h_K - h_L}_\infty = 0$. Since $h_K$ and $h_L$ are continuous functions, we must have $h_K = h_L$. Then $K = L$ since convex bodies are determined by their support functions. This suffices for the proof. 
\end{proof}

\begin{thm} \label{approximation-SSI}
	Let $K_1, \ldots, K_m \subset \RR^n$ be convex bodies. To every $\varepsilon > 0$, there are simple strongly isomorphic polytopes $P_1, \ldots, P_m$ such that $\delta (K_i, P_i) < \varepsilon$ for $i = 1, \ldots, m$.
\end{thm}

\begin{proof}
	See Theorem 2.4.15 in \cite{schneider_2013}. 
\end{proof}

We define a few continuous functions with respect to the Hausdorff distance on $\mathsf{K}^n$. This will allow us to compute the function values of general convex bodies via approximation by SSI polytopes. Since convex bodies are compact, they are Lebesgue measurable. This implies that there is a well-defined function $\Vol_n : \mathsf{K}^n \to \RR_{\geq 0}$ given by
\[
	\Vol_n (K) := \int_{x \in \RR^n} \1_K(x) \, d \lambda(x)
\]
where $\lambda$ is the Lebesgue measure on $\RR^n$. From Theorem 1.8.20 in \cite{schneider_2013}, the volume functional $\Vol_n(\cdot)$ is a continuous function on $(\mathsf{K}^n, \delta)$. For another example of a continuous map, consider the projection map $p : \mathsf{K}^n \times \RR^n \to \RR^n$ which maps $(K, x) \mapsto p(K, x)$ where $p(K, x)$ is the projection of $x$ onto $K$. For the proof that this map is continuous, see Section 1.8 in \cite{schneider_2013}. Finally, the map induced by the Minkowski sum $\mathsf{K}^n \times \mathsf{K}^n \to \mathsf{K}^n$ is continuous.

\subsection{Mixed Volumes}

Recall that in the setting of mixed discriminants, for $n \times n$ matrices $A_1, \ldots, A_m$ we had the identity
\begin{equation} \label{ye-old}
	\det (\lambda_1 A_1 + \ldots + \lambda_m A_m) = \sum_{i_1, \ldots, i_n = 1}^m \lambda_{i_1} \ldots \lambda_{i_n} \cdot \mathsf{D}(A_{i_1}, \ldots, A_{i_n})
\end{equation}
In this section, we study the convex body analog of Equation~\ref{ye-old}. For fixed convex bodies $K_1, \ldots, K_m$, we consider the function

\begin{equation} \label{eqn:mixed-volume-expansion}
	(\lambda_1, \ldots, \lambda_m) \longmapsto \Vol_n \left (\lambda_1 K_1 + \ldots + \lambda_m K_m \right )
\end{equation}

Similar to the situation in mixed discriminants, the function described in Equation~\ref{eqn:mixed-volume-expansion} ends up being a homogeneous polynomial of degree $n$. The coefficients of this polynomial are the convex body analogs of mixed discriminants called mixed volumes. 

\begin{example} \label{example:small-case}
	In this example, we compute Equation~\ref{eqn:mixed-volume-expansion} in the dimension $n = 2$. Let $K \subseteq \RR^2$ be a polygon and $L \subseteq \RR^2$ be the unit disc. Then, by simple geometric reasoning we have 
	\[
		\Vol_2 (\lambda K + \mu L) = \lambda^2 \Vol_2(K) + \lambda \mu \cdot \text{perimeter}(K) + \mu^2 \Vol_2(L). 
	\]
	Note that the coefficients of the polynomial encode geometric information about $K$ and $L$. These coefficients will be the mixed volumes of the convex bodies in the mixed Minkowski sum. 
\end{example}

We first define mixed volumes for polytopes in Definition~\ref{def:mixed-volume-polytope}. Then we prove that the expression for mixed volumes for polytopes admits a continuous extension to all convex bodies. This continuous extension will be the definition for mixed volumes for arbitrary convex bodies.

\begin{defn} \label{def:mixed-volume-polytope}
	For $n \geq 2$, let $P_1, \ldots, P_n \subseteq \RR^n$ be polytopes and let $\mathcal{U}$ be the set of unit facet normals of $P_1 + \ldots + P_{n-1}$. We define their mixed volume $\mathsf{V}_n(P_1, \ldots, P_n)$ inductively by 
	\[
		\mathsf{V}_n(P_1, \ldots, P_n) = \frac{1}{n} \sum_{u \in \mathcal{U}} h_{P_n}(u) \cdot \mathsf{V}_{n-1} (F_{P_1}(u), \ldots, F_{P_{n-1}}(u)).
	\]
	Since $F_{P_1}(u), \ldots, F_{P_{n-1}}(u)$ are contained in parallel hyperplanes in $\RR^n$, the $(n-1)$-dimensional mixed volume of these convex bodies is well-defined. In the case $n = 1$, we define $\mathsf{V}_1 ([a, b]) = b-a$. 
\end{defn}

\begin{thm}[Theorem 3.7 in \cite{Hug2020-ue}] \label{mixed-volume-polynomial-expansion}
	Let $P_1, \ldots, P_m \subseteq \RR^n$ be polytopes. Then, we have 
	\[
		\Vol_n(\lambda_1 P_1 + \ldots + \lambda_m P_m) = \sum_{i_1, \ldots, i_n = 1}^m \lambda_{i_1} \ldots \lambda_{i_n} \cdot \mathsf{V}_n (P_{i_1}, \ldots, P_{i_n}).
	\]
\end{thm}

\begin{proof}
	It suffices to prove the equality when $\lambda_1, \ldots, \lambda_m > 0$. The general expansion will follow from the continuity of polynomials, the volume functional, and Minkowski sums. Since the $\lambda_i$ are all strictly positive, Proposition~\ref{families-of-strongly-isomorphic} implies that $\lambda_1 P_1 + \ldots + \lambda_m P_m$ and $P_1 + \ldots + P_m$ are strongly isomorphic. In particular, the set of facet normals $\mathcal{U}$ are the same for both of these polytopes. We can compute 
	\begin{align*}
		\Vol_n (\lambda_1 P_1 + \ldots + \lambda_m P_m) & = \frac{1}{n} \sum_{u \in \mathcal{U}} h_{\lambda_1 P_1 + \ldots + \lambda_m P_m}(u) \cdot \Vol_{n-1} \left ( F_{\lambda_1 P_1 + \ldots + \lambda_m P_m} (u) \right ) \\
		& = \frac{1}{n} \sum_{u \in \mathcal{U}} \sum_{i = 1}^n \lambda_i h_{P_i}(u) \cdot \Vol_{n-1} \left ( \sum_{i = 1}^m \lambda_i F_{P_i}(u) \right ) \\
		& = \frac{1}{n} \sum_{u \in \mathcal{U}} \sum_{i_1, \ldots, i_n = 1}^m \lambda_{i_1} \ldots \lambda_{i_n} \cdot h_{P_{i_n}}(u) \cdot \mathsf{V}_{n-1} \left (F_{P_{i_1}}(u), \ldots, F_{P_{i_{n-1}}}(u) \right ) \\
		& = \sum_{i_1, \ldots, i_n = 1}^m \lambda_{i_1} \ldots \lambda_{i_n} \sum_{u \in \mathcal{U}} \frac{1}{n} h_{P_{i_n}}(u) \cdot \mathsf{V}_{n-1} \left (F_{P_{i_1}}(u), \ldots, F_{P_{i_{n-1}}}(u) \right ) \\
		& = \sum_{i_1, \ldots, i_n = 1}^m \lambda_{i_1} \ldots \lambda_{i_n} \cdot \mathsf{V}_n \left (F_{P_{i_1}}(u), \ldots, F_{P_{i_{n}}}(u) \right ). 
	\end{align*}
	This suffices for the proof. 
\end{proof}

From Definition~\ref{def:mixed-volume-polytope}, it is not clear that the mixed volume is symmetric. The symmetry for mixed volumes will follow from the inversion formula given in Theorem~\ref{inversion-formula}. This result will not only prove that the mixed volume is symmetric, but it gives a continuous extension of the mixed volume to all convex bodies.

\begin{thm}[Inversion Formula, Lemma 5.1.4 in \cite{schneider_2013}] \label{inversion-formula}
	For polytopes $P_1, \ldots, P_n \subseteq \RR^n$, we have
	\begin{align*}
		\mathsf{V}_n(P_1, \ldots, P_n) & = \frac{1}{n!} \sum_{k = 1}^n (-1)^{n+k} \sum_{1 \leq r_1 < \ldots < r_k \leq n} \Vol_n (P_{r_1} + \ldots + P_{r_k}) \\
		& = \frac{1}{n!} \sum_{I \subseteq [n]} (-1)^{n - |I|} \Vol_n \left ( \sum_{i \in I} P_i \right ).
	\end{align*}
\end{thm}

\begin{proof}
	We present the proof in \cite{schneider_2013} for the sake of completeness. We can define the function $f : \RR^n \to \RR$ given by 
	\[
		f(\lambda_1, \ldots, \lambda_n) := \frac{1}{n!} \sum_{k = 1}^n (-1)^{n+k} \sum_{1 \leq r_1 < \ldots < r_k \leq n} \Vol_n(\lambda_{r_1} P_{r_1} + \ldots+ \lambda_{r_k} P_{r_k}).
	\]
	From Theorem~\ref{mixed-volume-polynomial-expansion}, we get that $f$ is either $0$ or a homogeneous polynomial of degree $n$. If we let $\lambda_1 = 0$, then we have
	\begin{align*}
		n! (-1)^n f(0, \lambda_2, \ldots, \lambda_n) & := S_1 + \sum_{k = 2}^n S_k
	\end{align*}
	where we define $S_1 = - \sum_{2 \leq r_1 \leq n} \Vol_n (\lambda_{r_1} P_{r_1})$ and for $k \geq 2$ we define
	\begin{align*}
		S_k & = \sum_{k = 1}^n (-1)^k \left ( \sum_{2 \leq r_2 < \ldots < r_k \leq n} \Vol_n(\lambda_{r_2} P_{r_2} + \ldots + \lambda_{r_k} P_{r_k}) + \sum_{2 \leq r_1 < \ldots < r_k \leq n} \Vol_n (\lambda_{r_1} P_{r_1} + \ldots + \lambda_{r_k} P_{r_k})\right ).
	\end{align*}
	This sum telescopes and we are left with $f(0, \lambda_2, \ldots, \lambda_n) = 0$. By symmetry, we know that $f$ is a scalar multiple of $\lambda_1 \ldots \lambda_n$. The only term in the sum defining $f$ which contributes the monomial $\lambda_1 \ldots \lambda_n$ is the term $\Vol_n (\sum_{i = 1}^n \lambda_i P_i)$. Thus, we have that 
	\[
		f(\lambda_1, \ldots, \lambda_n) = \lambda_1 \ldots \lambda_n \cdot \mathsf{V}_n (P_1, \ldots, P_n).
	\]
	By substituting $\lambda_1 = \ldots = \lambda_n = 1$, this completes the proof. 
\end{proof}

As an application of the inversion formula, we will compute the mixed volume of a collection of line segments. This calculation will be instrumental for the calculation of the volume of a zonotope. 

\begin{example}\label{mixed-volume-of-line-segments}
	Let $v_1, \ldots, v_n \in \RR^n$ be vectors. From Theorem~\ref{inversion-formula}, we have that 
	\begin{align*}
		\mathsf{V}_n([0, v_1], \ldots, [0, v_n]) = \frac{1}{n!} \Vol_n ([0, v_1] + \ldots + [0, v_n])
	\end{align*}
	where the other summands are zero since they are contained in lower dimensional affine spaces. Let $T : \RR^n \to \RR^n$ be the linear map which sends $Te_i = v_i$ for $1 \leq i \leq n$. Then, $[0, v_1] + \ldots + [0, v_n] = T ([0, 1]^n)$. This gives us the identity 
	\[
		\mathsf{V}_n([0, v_1], \ldots, [0, v_n]) = \frac{1}{n!} \Vol_n(T([0, 1]^n)) = \frac{1}{n!} |\text{Det}(v_1, \ldots, v_n)|.
	\]
\end{example}

Since the volume functional and Minkowski sum is continuous on the space of convex bodies, the function on the right hand side of Theorem~\ref{inversion-formula} is also a continuous function on the space of convex bodies. This allows us to extend the definition of mixed volumes to all convex bodies. 

\begin{defn} \label{def:mixed-volume-general}
	Let $K_1, \ldots, K_n \subseteq \RR^n$ be convex bodies. Then, we define the \textbf{mixed volume} of $K_1, \ldots, K_n$ as  
	\[
		\mathsf{V}_{n}(K_1, \ldots, K_n) := \frac{1}{n!} \sum_{I \subseteq [n]} (-1)^{n- |I|} \Vol_n \left ( \sum_{i \in I} P_i \right ).
	\]
\end{defn}

The function defined in Definition~\ref{inversion-formula} is continuous. Let $K_1, \ldots, K_n \subseteq \RR^n$ be a collection of convex bodies. Given a sequence $P^{(i)}_k$ of polytopes for all $1 \leq i \leq n$ satisfying  $P_i^{(k)} \longrightarrow K_i$ as $k \to \infty$ for all $1 \leq i \leq n$, we have 
\[
	\mathsf{V}_n(K_1, \ldots, K_n) = \lim_{k \to \infty} \mathsf{V}_n (P_1^{(k)}, \ldots, P_n^{(k)})
\] 
from the continuity of the mixed volume. This allows us to extend Theorem~\ref{mixed-volume-polynomial-expansion} to general convex bodies. 

\begin{thm} \label{mixed-volume-polynomial-expansion-FINAL}
	Let $K_1, \ldots, K_m \subseteq \RR^n$ be convex bodies. Then, we have 
	\[
		\Vol_n (\lambda_1 K_1 + \ldots + \lambda_m K_m) = \sum_{i_1 , \ldots , i_n = 1}^m \lambda_{i_1} \ldots \lambda_{i_n} \cdot \mathsf{V}_n (K_{i_1}, \ldots, K_{i_n}).
	\]
\end{thm}

\begin{proof}
	For $1 \leq i \leq n$, let $P_i^{(k)}$ for $k \geq 1$ be a sequence of polytopes converging to $K_i$ in Hausdorff distance. Then, we have 
	\begin{align*}
		\Vol_n \left ( \sum_{i = 1}^n \lambda_i K_i \right ) & = \lim_{k \to \infty} \Vol_n \left ( \sum_{i = 1}^n \lambda_i P_i^{(k)} \right ) \\
		& = \lim_{k \to \infty} \sum_{i_1 , \ldots , i_n = 1}^m \lambda_{i_1} \ldots \lambda_{i_n} \cdot \mathsf{V}_n (P_{i_1}^{(k)}, \ldots, P_{i_n}^{(k)}) \\
		& = \sum_{i_1 , \ldots , i_n = 1}^m \lambda_{i_1} \ldots \lambda_{i_n} \cdot \lim_{k \to \infty} \mathsf{V}_n (P_{i_1}^{(k)}, \ldots, P_{i_n}^{(k)}) \\
		& = \sum_{i_1 , \ldots , i_n = 1}^m \lambda_{i_1} \ldots \lambda_{i_n} \cdot \mathsf{V}_n (K_{i_1}, \ldots, K_{i_n}).
	\end{align*}
	This suffices for the proof. 
\end{proof}

\begin{example}[Volume of a Zonotope] \label{example-volume-of-a-zonotope}
	From Theorem~\ref{mixed-volume-polynomial-expansion-FINAL}, we can compute the volume of a zonotope. Let $v_1, \ldots, v_l \in \RR^n$ be vectors. Then, we have that 
	\begin{align*}
		\Vol_n (Z(v_1, \ldots, v_l)) & = \Vol_n \left ( \sum_{i = 1}^l [0, v_i] \right ) \\
		& = \sum_{i_1, \ldots, i_n = 1}^l \mathsf{V}_n ([0, v_{i_1}], \ldots, [0, v_{i_n}]) \\
		& = \sum_{i_1, \ldots, i_n = 1}^l \frac{1}{n!} |\text{Det}(v_{i_1}, \ldots, v_{i_n}) | \\
		& = \sum_{1 \leq i_1 < \ldots < i_n \leq l} |\text{Det}(v_{i_1}, \ldots, v_{i_n})|
	\end{align*}
	where the equality in the third line follows from the computation in Example~\ref{mixed-volume-of-line-segments}. 
\end{example}

We now define the notion of mixed area measures, a measure-theoretic interpretation of mixed volumes. This notion will become essential when considering the equality cases of the Alexandrov-Fenchel inequality. Our treatment of mixed area measures is inspired by Chapter 4 in \cite{Hug2020-ue}. Recall that for polytopes $P_1, \ldots, P_n$, we define the mixed volume as 
\[
	\mathsf{V}_n (P_1, \ldots, P_n) = \frac{1}{n} \sum_{u \in \mathbb{S}^{n-1}} h_{P_n}(u) \cdot \mathsf{V}_{n-1}(F_{P_1}(u), \ldots, F_{P_{n-1}}(u)).
\]
The sum is well-defined because the number of facet normals for polytopes is finite. Outside of these facet normals, the mixed volume inside the sum vanishes. For any convex body $K \subseteq \RR^n$ which is not necessarily a polytope, we can take a sequence of polytopes convering to $K$ to get the identity
\[
	\mathsf{V}_n(P_1, \ldots, P_{n-1}, K) = \frac{1}{n} \sum_{u \in \mathbb{S}^{n-1}} h_{K}(u) \cdot \mathsf{V}_{n-1}(F_{P_1}(u), \ldots, F_{P_{n-1}}(u)).
\]
From a measure theory perspective, we have a discrete measure given by
\begin{equation} \label{equation:mixed-area-measure}
	S_{P_1, \ldots, P_{n-1}} := \sum_{u \in \mathbb{S}^{n-1}} \mathsf{V}_{n-1}(F_{P_1}(u), \ldots, F_{P_{n-1}}(u)) \cdot \delta_u
\end{equation}
which satisfies the equation 
\[
	\mathsf{V}_n(P_1, \ldots, P_{n-1}, K) = \frac{1}{n} \int_{\mathbb{S}^{n-1}} h_{K}(u) \, S_{P_1, \ldots, P_{n-1}}(du).
\]
In Equation~\ref{equation:mixed-area-measure}, the measure $\delta_u$ is the Dirac measure at $u$. We have defined the measure $S_{P_1, \ldots, P_{n-1}}$ in the case where the first $n-1$ convex bodies in the mixed volume are polytopes. Using the Riesz representation theorem, it can be shown that such a measure exists for general convex bodies (see Theorem 4.1 in \cite{Hug2020-ue} and Theorem 2.14 in \cite{rudin}). The existence of this measure will follow from Theorem~\ref{thm-exists-mixed-area-measure}. 

\begin{thm} \label{thm-exists-mixed-area-measure}
	For convex bodies $K_1, \ldots, K_{n-1} \subseteq \RR^n$, there exists a uniquely determined finite Borel measure $S_{K_1, \ldots, K_{n-1}}$ on $\mathbb{S}^{n-1}$ such that 
	\[
		\mathsf{V}(K_1, \ldots, K_{n-1}, K) = \frac{1}{n} \int_{\mathbb{S}^{n-1}} h_K(u) \, S_{K_1, \ldots, K_{n-1}} (du)
	\]
	for all convex bodies $K \subseteq \RR^n$. 
\end{thm}

\begin{proof}
	For a proof of this result, see Theorem 4.1 in \cite{Hug2020-ue}. 
\end{proof}

\begin{defn} \label{def:mixed-area-measure}
	For convex bodies $K_1, \ldots, K_{n-1} \subseteq \RR^n$, we call the measure $S_{K_1, \ldots, K_{n-1}}$ the \textbf{mixed area measure} associated with $(K_1, \ldots, K_{n-1})$. 
\end{defn}

The mixed area measure is a generalization of the spherical Lebesgue measure on $\mathbb{S}^{n-1}$. For example, when our convex bodies satisfy $K_1 = \ldots = K_{n-1} = \mathbb{B}^n$, we have that $S_{\mathbb{B}^n, \ldots, \mathbb{B}^n}$ is exactly the spherical Lebesgue measure on $\mathbb{S}^{n-1}$. Thus, for any convex body $K \subseteq \RR^n$ have 
\[
	\mathsf{V}_n (K, \mathbb{B}^n, \ldots, \mathbb{B}^n ) = \frac{1}{n} \int_{\mathbb{S}^{n-1}} h_K(u) \, S_{\mathbb{B}^n, \ldots, \mathbb{B}^n} (du) = \frac{1}{n} \int_{\mathbb{S}^{n-1}} h_K(u) \, \sigma (du). 
\]
This implies that the mixed volume $\mathsf{V}_n (K, \mathbb{B}^n, \ldots, \mathbb{B}^n)$ is exactly the surface area of $K$. This calculation explains the appearance of the perimeter in Example~\ref{example:small-case}. In the next section, we discuss the positivity of mixed volumes and the support of mixed area measures. 

\subsection{Positivity and Normal Directions}

We begin this section with necessary and sufficient conditions for a mixed volume to be strictly positive. The conditions are collected in Lemma~\ref{positivity-of-mixed-volumes} for reference.

\begin{lem} \label{positivity-of-mixed-volumes}
	For conex bodies $C_1, \ldots, C_n \subseteq \RR^n$, the following two conditions are equivalent.
	\begin{enumerate}[label = (\alph*)]
		\item $\mathsf{V}_n (C_1, \ldots, C_n) > 0$. 
		\item There are segments $I_i \subseteq C_i$, $i \in [n]$ with linearly independent directions. 
		\item $\dim (C_{i_1} + \ldots + C_{i_k}) \geq k$ for all $k \in [n]$, $1 \leq i_1 < \ldots < i_k \leq n$. 
	\end{enumerate}
\end{lem}
\begin{proof}
	See Lemma 2.2 in \cite{shenfeld2022extremals}. 
\end{proof}
Next, we provide necessary and sufficient conditions for a vector to be in the support of the mixed area measure $S_{B, P_1, \ldots, P_{n-2}}$ where $P_1, \ldots, P_{n-2}$ are polytopes. The terminology used in Definition~\ref{extreme-normal-directions} was introduced in Lemma 2.3 of \cite{shenfeld2022extremals}. 

\begin{defn}[Lemma 2.3 in \cite{shenfeld2022extremals}] \label{extreme-normal-directions}
	For $P_1, \ldots, P_{n-2} \subseteq \RR^n$ convex polytopes and $u \in \mathbb{S}^{n-1}$. We call the vector $u$ a $(B, P_1, \ldots, P_{n-2})$-extreme \textbf{normal direction} if and only if at least one of the following three equivalent conditions hold:
	\begin{enumerate}[label = (\alph*)]
		\item $u \in \supp S_{B, P_1, \ldots, P_{n-2}}$. 
		\item There are segments $I_i \subseteq F (P_i, u)$, $i \in [n-2]$ with linearly independent directions. 
		\item $\dim(F(P_{i_1}, u) + \ldots + F(P_{i_k}, u)) \geq k$ for all $k \geq [n-2]$, $1 \leq i_1 < \ldots < i_k \leq n-2$. 
	\end{enumerate}
\end{defn}

The fact that (a)-(c) are equivalent in Definition~\ref{extreme-normal-directions} is exactly the content of Lemma 2.3 in \cite{shenfeld2022extremals}. Both Lemma~\ref{positivity-of-mixed-volumes} and Definition~\ref{extreme-normal-directions} are important concepts related to the extremals of the Alexandrov-Fenchel inequality. They will reappear in Section~\ref{kahn-saks}.

\chapter{Mechanisms for Log-concavity}
In this chapter, we study a handful of tools which have been used to prove log-concavity conjectures in combinatorics. We begin the chapter with an introduction on log-concavity, ultra-log-concavity, and some of the log-concavity phenomenon present in combinatorics. Then, we prove Cauchy's interlacing theorem. This theorem has been used frequently for example in the theory of Lorentzian polynomials in \cite{lorentzian-polynomials} and Hodge Riemann relations in \cite{MN-gorenstein}. In the next two sections, we introduce Alexandrov's inequality for mixed discriminants as well as the Alexandrov-Fenchel inequality for mixed volumes. We will also consider the equality cases of these inequalities. We end the chapter with an overview of the theory of Lorentzian polynomials. In this chapter, we were only able to cover an $\varepsilon$ percentage of the tools used to prove log-concavity conjectures. For a wider survey of the techniques available, we refer the reader to \cite{Stanley1989LogConcaveAU}. We also chose not to discuss the recent technology of the combinatorial atlas \cite{combinatorial-atlas,logconcave-poset-inequalities} due to space and time constraints. 

\section{Log-concavity and Ultra-log-concavity}

In this thesis we are interested in the log-concavity and ultra-log-concavity of sequences which come from combinatorial structures. Let $a_1, \ldots, a_n$ be a sequence of non-negative numbers. We say that the sequence is log-concave if we have $a_i^2 \geq a_{i-1} a_{i+1}$ for all $1 \leq i \leq n$ where $a_{-1} = a_{n+1} = 0$. We say that the sequence is ultra-log-concave if the sequence $a_i / \binom{n}{i}$ is log-concave. Perhaps the most prototypical example of an ultra-log-concave sequence which arises from combinatorics is the sequence of binomial coefficients $\binom{n}{k}$. In the literature, there are many instances in which a conjecture was made about a sequence being unimodal, and then the conjecture was solved by proving that it is log-concave. This follows because a log-concave sequence with no internal zeroes will automatically be unimodal. This is example of a common phenomenon in mathematics where it is easier to prove a result that is stronger than what was asked. We now give some examples of log-concave sequences in combinatorics.

\begin{example}[Read's Conjecture and Heron-Rota-Welsh Conjecture]
	Let $G$ be a finite graph and let $\chi_G(x)$ be the chromatic polynomial of $G$. In his 1968 paper \cite{Read-conjecture}, Ronald Read conjectured that the absolute values of the coefficients of the chromatic polynomial of a graph $G$ are unimodal. The conjecture was then strengthed by Heron, Rota, and Welsh to the log-concavity of the coefficients of the characteristic polynomial of an arbitrary matroid. Read's conjecture was first proved by June Huh in his paper \cite{milnor-numbers} where he proved the Heron-Rota-Welsh conjecture for representable matroids. The Heron-Rota-Welsh conjecture was then proved in full generality by Karim Adiprasito, June Huh, and Eric Katz in their paper \cite{AHK} by proving that the Chow ring satisfies the hard Lefschetz theorem and Hodge-Riemann relations. 
\end{example}

\begin{example}[The Mason Conjectures]
	Let $M$ be a matroid of rank $r$. For $0 \leq i \leq r$, let $I_i$ denote the number of independent sets of $M$ with $i$ elements. Then there are three conjectures of increasing strength related to the log-concavity of this sequence. 
	\begin{enumerate}
		\item (Mason Conjecture) $I_k^2 \geq I_{k-1} I_{k+1}$. 
		\item (Strong Mason Conjecture) $I_k^2 \geq \left (1 + \frac{1}{k} \right ) I_{k-1} I_{k+1}$. 
		\item (Ultra-Strong Mason Conjecture) $I_k^2 \geq \left (1 + \frac{1}{k} \right ) \left (1 + \frac{1}{n-k} \right ) I_{k-1} I_{k+1}$. 
	\end{enumerate}
	The Mason conjecture says that the sequence $I_i$ is log-concave while the ultra-strong Mason conjecture states that the sequence $I_i$ is ultra-log-concave. All three of these conjectures have been proven. In \cite{AHK}, Karim Adiprasito, June Huh, and Eric Katz proved the Mason conjecture. Later, the strong Mason conjecture was proven by June Huh, Benjamin Schr\"oter, and Botong Wang in \cite{correlation-bounds} and Nima Anari, Kuikui Liu, Shayan Oveis Gharan, and Cynthia Vinzant in \cite{anari2018logconcave} independently.  
\end{example}

\section{Cauchy's Interlacing Theorem} \label{sec:cauchy-interlacing-theorem}

Given an $n \times n$ matrix $A$, we call a submatrix $B$ a \textbf{principal submatrix} if it is obtained from $A$ by deleting some rows and the corresponding columns. When $A$ is a Hermitian matrix and $B$ is a principal submatrix of size $(n-1) \times (n-1)$, Theorem~\ref{cauchy-interlacing} gives a mechanism for controlling the eigenvalues of $B$ based on $A$. We present a short proof from \cite{fisk} of this result. For different proofs using the intermediate value theorem, Sylvester's law of inertia, and the Courant-Fischer minimax theorem, we refer the reader to \cite{suk}, \cite{symmetric-eigenvalue-problem}, and \cite{GoluVanl96}, respectively. 

\begin{thm}[Cauchy Interlace Theorem] \label{cauchy-interlacing}
	Let $A$ be a Hermitian matrix of order $n$, and let $B$ be a principal submatrix of $A$ of order $n-1$. If $\lambda_n \leq \lambda_{n-1} \leq \ldots \leq \lambda_2 \leq \lambda_1$ are the eigenvalues of $A$ and $\mu_n \leq \mu_{n-1} \leq \ldots \leq \mu_3 \leq \mu_2$ are the eigenvalues of $B$. Then, we have
	\[
		\lambda_n \leq \mu_n \leq \lambda_{n-1} \leq \mu_{n-1} \leq \ldots \leq \lambda_2 \leq \mu_2 \leq \lambda_1.
	\]
\end{thm}

Before proving this result, we need Theorem~\ref{interlacing-roots-of-polynomial}. This is a theorem about when roots of polynomials interlace. Suppose that $f, g \in \RR[x]$ are real polynomials with only real roots. We say that the polynomials $f$ and $g$ \textbf{interlace} if their roots $r_1 \leq \ldots \leq r_n$ and $s_1 \leq \ldots \leq s_{n-1}$ satisfy
\[
	r_1 \leq s_1 \leq \ldots \leq r_{n-1} \leq s_{n-1} \leq r_n.
\]

\begin{thm} \label{interlacing-roots-of-polynomial}
	Let $f, g \in \RR[x]$ be polynomials with only only real roots. Suppose that $\deg (f) = n$ and $\deg (g) = n-1$. Then, $f$ and $g$ interlace if and only if the linear combinations $f + \alpha g$ have all real roots for all $\alpha \in \RR$. 
\end{thm} 
\begin{proof}
	See Theorem 6.3.8 in \cite{rahman}.
\end{proof}

\begin{proof}[Proof of Theorem~\ref{cauchy-interlacing}]
	We follow the proof in \cite{fisk}. Without loss of generality, we can decompose 
	\[
		A = \begin{bmatrix}
			B & v \\
			v^T & c
		\end{bmatrix}
	\]
	where $v \in \RR^{(n-1) \times 1}$ and $c \in \RR$. Consider the polynomials $f(x) := \det (A - xI)$ and $g(x) := \det (B - xI)$. Note that the roots of $f$ and $g$ are all real and are exactly the eigenvalues of $A$ and $B$, respectively. For any $\alpha \in \RR$, we have
	\begin{align*}
		f(x) + \alpha g(x) & = \det\begin{bmatrix}
			B-xI & v \\
			v^T & d-x
		\end{bmatrix} + \alpha \det \begin{bmatrix}
			B-xI & v \\
			0 & 1
		\end{bmatrix} \\
		& = \det\begin{bmatrix}
			B-xI & v \\
			v^T & d-x
		\end{bmatrix} + \det \begin{bmatrix}
			B-xI & v \\
			0 & \alpha
		\end{bmatrix} \\
		& = \det \begin{bmatrix}
			B-xI & v \\
			v^T & d+\alpha - x
		\end{bmatrix}.
	\end{align*}
	Thus $f + \alpha g$ is the characteristic polynomial of a Hermitian matrix and has real roots. From Theorem~\ref{interlacing-roots-of-polynomial}, the proof is complete. 
\end{proof}

\begin{cor}
	Let $M$ be a $n \times n$ matrix with exactly one positive eigenvalue. If $A$ is a principal matrix $A$ of size $2 \times 2$ such that there exists $v \in \RR^2$ with $v^T A v > 0$, then $\det (A) \leq 0$. 
\end{cor}

\begin{proof}
	From Theorem~\ref{cauchy-interlacing}, the matrix $A$ has at most one positive eigenvalue. Since $v^T A v > 0$, we know that $A$ has exactly one positive eigenvalue. The other eigenvalue must be at most $0$. Hence, we have $\det (A) \leq 0$. This suffices for the proof. 
\end{proof}

\section{Alexandrov's Inequality for Mixed Discriminants} \label{sec:a-inequality}

In this section we prove Alexandrov's inequality for mixed discriminants. This is an inequality which describes the log-concavity of mixed discriminants. The equality cases of Alexandrov's inequality are fully resolved in \cite{Panov_1987}. In Section~\ref{sec:af-inequality}, we will discuss the mixed volume analog. In Theorem~\ref{A-Inequality-MIXED-DISCRIMINANT}, we give a special case of the mixed discriminant inequality where the equality case is simple. In Theorem~\ref{thm:only-inequality-mixed-discriminants} we give a more general statement of Alexandrov's inequality without mention of equality cases. 

\begin{thm} [Alexandrov's Inequality for Mixed Discriminants] \label{A-Inequality-MIXED-DISCRIMINANT}
	Let $A_1, \ldots, A_{n-2}$ be real symmetric positive definite $n \times n$ matrices. Let $X$ be a real symmetric positive definite $n \times n$ square matrix and let $Y$ be a real symmetric positive semidefinite $n \times n$ square matrix. Then 
	\[
		\mathsf{D}(X, Y, A_1, \ldots, A_{n-2})^2 \geq \mathsf{D}(X, X, A_1, \ldots, A_{n-2}) \cdot \mathsf{D} (Y, Y, A_1, \ldots, A_{n-2})
	\]
	where equality holds if and only if $Y = \lambda X$ for a real number $\lambda \in \RR$. 
\end{thm}

We present a proof given in \cite{bochner} of the inequality Theorem~\ref{thm:only-inequality-mixed-discriminants}. For a proof of Theorem~\ref{A-Inequality-MIXED-DISCRIMINANT} including the equality cases, we refer the reader to \cite{Panov_1987}. 

\begin{thm} \label{thm:only-inequality-mixed-discriminants}
	Let $A, B, M_1, \ldots, M_{n-2}$ be $n$-dimensional symmetric matrices where the last $n-1$ matrices are positive semi-definite. Then, we have 
	\[
		\mathsf{D}(A, B, M_1, \ldots, M_{n-2})^2 \geq \mathsf{D}(A, A, M_1, \ldots, M_{n-2}) \mathsf{D} (B, B, M_1, \ldots, M_{n-2}).
	\]
\end{thm}

Before we begin the proof, we first prove the following result from \cite{hardy1952inequalities} about real-rootedness. Only the case $n = 2$ is needed. Even though the case $n = 2$ can be proven using elementary facts about quadratic polynomials, we prove the general result for the sake of completeness. 

\begin{lem} [\cite{hardy1952inequalities} and Lemma 5.3.2 in \cite{bapat_raghavan_1997}] \label{newton}
	Let $p(x) = \sum_{k = 0}^n \binom{n}{k} a_k x^k$ be a real-rooted univariate polynomial. Then, we have $a_k^2 \geq a_{k-1}a_{k+1}$ for all $1 \leq i \leq n-1$. If $a_0 \neq 0$ then equality occurs for any $i$ if and only if all the roots of $f(x)$ are equal. 
\end{lem}

\begin{proof}
	From Rolle's Theorem, if a polynomial is real-rooted, then so are all of its derivatives. When we reverse the coefficients of a real-rooted polynomial, it remains real-rooted. We call the result of reversing the coefficients of a polynomial the \textbf{reciprocation} of the polynomial. Through a sequence of derivatives and reciprocations, we can reach a polynomial of the form $\alpha_{i-1} + 2\alpha_i x + \alpha_{i+1} x^2$. Since we arrived at this polynomial through a sequence of derivatives and reciprocations, we know that this polynomial is real-rooted. Hence, the discriminant is non-negative and $\alpha_i^2 \geq \alpha_{i-1} \alpha_{i+1}$. The equality case follows from the fact that Rolle's Theorem implies that the roots of the derivative of a real-rooted polynomial interlaces the roots of the original polynomial. 
\end{proof}

We will employ an inductive proof of Theorem~\ref{A-Inequality-MIXED-DISCRIMINANT}. Lemma~\ref{newton} proves the base case immediately. For the base case, we want to prove that if $A, B$ are two-dimensional symmetric matrices with $B$ positive semi-definite, then $\mathsf{D}(A, B)^2 \geq \mathsf{D}(A, A) \mathsf{D}(B, B)$. From continuity, we can assume that $B$ is positive definite. Thus, we have that the polynomial $p(t)$ given by 
\[
	p(t) = \det (A + tB) = \det (B) \cdot \det (B^{-1/2} A B^{-1/2} + t)
\] 
is real-rooted. Since the coefficients of $\det (A + tB)$ are $\mathsf{D}(A, A), 2 \mathsf{D}(A, B)$, and $\mathsf{D}(B, B)$, the base case follows. Now we prove a subcase of Theorem~\ref{thm:only-inequality-mixed-discriminants} given in Lemma 4.2 in \cite{bochner}. The proof will require a result about hyperbolic forms that we introduce later in the thesis (see Lemma~\ref{hyperbolic-quadratic-forms}).

\begin{lem} \label{lem:reduced-inequality}
	Let $n \geq 3$ and let $M_1, \ldots, M_{n-3}$ be $n$-dimensional positive definite matrices. Suppose that Theorem~\ref{thm:only-inequality-mixed-discriminants} holds for mixed discriminants of dimension $n-1$. Then, for any $n$-dimensional diagonal matrix $Z$, we have 
	\[
		\mathsf{D} (Z, I, I, M_1, \ldots, M_{n-3} )^2 \geq \mathsf{D}(Z, Z, I, M_1, \ldots, M_{n-3}) \mathsf{D} (I, I, I, M_1, \ldots, M_{n-3}).
	\]
\end{lem}

\begin{proof}
	For any $x, y \in \RR^n$, we can define the bilinear form 
	\begin{align*}
		D(x, y) & := \mathsf{D} (\operatorname{Diag}(x), \operatorname{Diag}(y), I, M_1, \ldots, M_{n-3} ) \\
		& = \mathsf{D} \left ( \sum_{i = 1}^n x_i e_ie_i^T, \operatorname{Diag}(y), I, M_1, \ldots, M_{n-3} \right ) \\
		& = \frac{1}{n} \sum_{i = 1}^n x_i \mathsf{D} \left ( \operatorname{Diag}(y)^{\langle i \rangle}, I, M_1^{\langle i \rangle}, \ldots, M_{n-3}^{\langle i \rangle} \right ).
	\end{align*}
	The equality follows from linearity and Lemma~\ref{lem:mixed-discriminant-properties}(c). We can defined the linear map $A : \RR^n \to \RR^n$ and constants $p_i$ for $1 \leq i \leq n$ by
	\begin{align*}
		Ay & :=  \sum_{i = 1}^n \frac{\mathsf{D}(\operatorname{Diag}(y)^{\langle i \rangle}, I, M_1^{\langle i \rangle}, \ldots, M_{n-3}^{\langle i \rangle})}{\mathsf{D}(I, I, M_1^{\langle i \rangle}, \ldots, M_{n-3}^{\langle i \rangle})} \\
		p_i & := \frac{1}{n} \mathsf{D} (I, I, M_1^{\langle i \rangle}, \ldots, M_{n-3}^{\langle i \rangle}).
	\end{align*}
	Note that in the definition of $Ay$ we are allowed to divide by the mixed discriminants $\mathsf{D}(I, I, M_1^{\langle i \rangle}, \ldots, M_{n-3}^{\langle i \rangle})$ because all of the arguments are positive definite. In particular, $p_i > 0$ for all $i$. If we define the inner product $\langle x, y \rangle_P := \langle x, P y \rangle$ where $P = \operatorname{Diag}(p_1, \ldots, p_n)$, it follows that 
	\[
		\langle x, Ay \rangle_P = \langle x, PA y \rangle = \mathsf{D}(\operatorname{Diag}(x), \operatorname{Diag}(y), I, M_1, \ldots, M_{n-3} ).
	\]
	Since the mixed discriminant is symmetric, we know $A$ is self-adjoint with respect to the inner product $\langle \cdot, \cdot \rangle_P$. Moreover, we know that $A \sum_{i = 1}^n e_i = \sum_{i = 1}^n e_i$ and $A$ is a positive matrix. From Theorem 1.4.4 in \cite{bapat_raghavan_1997}, the largest eigenvalue of $A$ is $1$ and this is a simple eigenvalue. Using the assumption that Alexandrov's inequality holds for dimension $n-1$, we can bound
	\begin{align*}
		\langle Ay, Ay \rangle_P & = \frac{1}{n} \sum_{i = 1}^n \frac{\mathsf{D}(\operatorname{Diag}(y)^{\langle i \rangle}, I, M_1^{\langle i \rangle}, \ldots, M_{n-3}^{\langle i \rangle})^2}{\mathsf{D} (I, I, M_1^{\langle i \rangle}, \ldots, M_{n-3}^{\langle i \rangle})} \\
		& \geq \frac{1}{n} \sum_{i = 1}^n \mathsf{D}(\operatorname{Diag}(y)^{\langle i \rangle}, \operatorname{Diag}(y)^{\langle i \rangle}, M_1^{\langle i \rangle}, \ldots, M_{n-3}^{\langle i \rangle} ) \\
		& = \frac{1}{n} \mathsf{D}(I, \operatorname{Diag}(y), \operatorname{Diag}(y), M_1, \ldots, M_{n-3}) \\
		& = \langle y, Ay \rangle_P.
	\end{align*}
	For any eigenvalue of $v$ of $A$ with corresponding eigenvalue $\lambda$, we must have 
	\[
		\lambda^2 \norm{v}_P^2 \geq \lambda \norm{v}_P^2 \implies \lambda \geq 1 \text{ or } \lambda \leq 0.
	\]
	This implies that the positive eigenspace of $A$ has dimension $1$. From Lemma~\ref{hyperbolic-quadratic-forms}, this suffices for the proof. 
\end{proof}

\begin{proof}[Proof of Theorem~\ref{thm:only-inequality-mixed-discriminants}.]
	To complete the proof of Theorem~\ref{thm:only-inequality-mixed-discriminants} it suffices to show that we can reduce the inequality to that in Lemma~\ref{lem:reduced-inequality}. We omit this proof and refer the reader to the proof of Corollary 4.3 in \cite{bochner}. 
\end{proof}

Corollary~\ref{mixed-discriminant-log-concave-sequence} illustrates that mixed discriminants involving two positive definite bodies are log-concave. Corollary~\ref{cor-one-implies-all-mixed-discriminants} says that in this situation equality at one point will give equality at all points. 

\begin{cor} \label{mixed-discriminant-log-concave-sequence}
	Let $A, B$ be $n \times n$ positive definite symmetric real matrices. For $0 \leq k \leq n$, define the mixed discriminant
	\[
		D_k := \mathsf{D} (\underbrace{A, \ldots, A}_{k \text{ times}}, \underbrace{B \ldots, B}_{n-k \text{ times}}).
	\]
	Then the sequence $D_0, D_1, \ldots, D_n$ is log-concave. 
\end{cor}

\begin{proof}
	This is an immediate consequence of Theorem~\ref{A-Inequality-MIXED-DISCRIMINANT}. 
\end{proof}

\begin{cor} \label{cor-one-implies-all-mixed-discriminants}
	Let $A, B, D_k$ for $1 \leq k \leq n$ be as in Corollary~\ref{mixed-discriminant-log-concave-sequence}. If $D_k^2 = D_{k-1}D_{k+1}$ for some $1 \leq k \leq n-1$, then $D_k^2 = D_{k-1} D_{k+1}$ holds for all $k = 1, \ldots, n-1$. 
\end{cor}

\begin{proof}
	This follows from the equality case in Theorem~\ref{A-Inequality-MIXED-DISCRIMINANT}. 
\end{proof}
\section{Alexandrov-Fenchel Inequality} \label{sec:af-inequality}

In this section, we prove the Alexandrov-Fenchel inequality. This is a mixed volume analog of Theorem~\ref{A-Inequality-MIXED-DISCRIMINANT}. Formally, these two inequalities look exactly the same. They are very different in terms of their extremals. The extremals of the Alexandrov inequality for mixed discriminants take a rather simple form while the extremals of the Alexandrov-Fenchel inequality are rich in examples. We will discuss more about equality cases in Section~\ref{sec:AF-equality-cases}. 

\begin{thm}[Alexandrov-Fenchel Inequality] \label{AF-inequality}
	Let $K_1, K_2, \ldots, K_{d-2} \subseteq \RR^d$ be convex bodies. For any convex bodies $X, Y \subseteq \RR^d$, we have the inequality
	\[
		\mathsf{V}_d(X, Y, K_1, \ldots, K_{d-2})^2 \geq \mathsf{V}_d(X, X, K_1, \ldots, K_{d-2}) \cdot \mathsf{V}_d(Y, Y, K_1, \ldots, K_{d-2}). 
	\]
\end{thm}

We now present the proof of the inequality given in \cite{bochner}. From Theorem~\ref{approximation-SSI}, we can approximate our convex bodies by SSI polytopes. If we can prove the Alexandrov-Fenchel inequality in the case of SSI polytopes, then the general result will follow from continuity. When restricted to an isomorphism class $\alpha$ of SSI polytopes with facet normals $\mathcal{U}$, the mixed volumes $\mathsf{V}_d (X, Y, P_1, \ldots, P_{d-2})$ can be viewed as a bilinear form applied to the support vectors of $X$ and $Y$. Indeed, for any support vectors $h_P, h_Q$, we can define 
\[
	\mathsf{V}_d (h_P, h_Q, P_1, \ldots, P_{d-2}) := \mathsf{V}_d (P, Q, P_1, \ldots, P_{d-2}).
\]
We can extend this definition to all of $\RR^{\mathcal{U}}$ using Corollary~\ref{perturbation-SSI}. Indeed, for any $x, y \in \RR^\mathcal{U}$ there are polytopes $A_1, A_2, B_1, B_2 \in \alpha$ such that $x = h_{A_2} - h_{A_1}$ and $y = h_{B_2} - h_{B_1}$. We define 
\[
	\mathsf{V}_d (x, y, P_1, \ldots, P_{d-2}) = \sum_{i, j = 1}^2 (-1)^{i+j}\mathsf{V}_d (h_{A_i}, h_{B_j}, P_1, \ldots, P_{d-2}).
\]
To prove that this extension is well-defined, it suffices to prove that if $x = h_K - h_L = h_P - h_Q$ for $K, L, P, Q, P_0 \in \alpha$, then 
\[
	\mathsf{V}_d (h_{K}, \mathcal{P}) - \mathsf{V}_d (h_{L}, \mathcal{P}) = \mathsf{V}_d (h_P,\mathcal{P}) - \mathsf{V}_d(h_Q, \mathcal{P})
\]
where $\mathcal{P} = (P_0, P_1, \ldots, P_{d-2})$. Note that $x = h_K - h_L = h_P - h_Q$ implies that $h_K + h_Q = h_L + h_P$. Thus $h_{K+Q} = h_{L+ P}$ where $K+Q, L+P \in \alpha$ from Proposition~\ref{families-of-strongly-isomorphic} and so $K + Q = L + P$. This allows us to make the computation
\begin{align*}
	\mathsf{V}_d(h_K, \mathcal{P}) + \mathsf{V}_d (h_Q, \mathcal{P}) & = \mathsf{V}_d (K, \mathcal{P}) + \mathsf{V}_d (Q, \mathcal{P}) \\
	& = \mathsf{V}_d (K + Q, \mathcal{P}) \\
	& = \mathsf{V}_d (L + P, \mathcal{P}) \\
	& = \mathsf{V}_d(L, \mathcal{P}) + \mathsf{V}_d (P, \mathcal{P}) \\
	& = \mathsf{V}_d(h_L, \mathcal{P}) + \mathsf{V}_d (h_P,\mathcal{P}).
\end{align*}
This proves that the extension of $\mathsf{V}_d (\cdot, \cdot, P_1, \ldots, P_{d-2})$ to $\RR^{\mathcal{U}}$ is well-defined. It suffices to prove the inequality written in Theorem~\ref{AF-SSI-version}. We can also extend the mixed volumes of support vectors to the mixed volumes of the faces of the support vectors. For all $u \in \mathcal{U}$ and $x \in \RR^\mathcal{U}$, there is a well-defined extension 
\[
	\mathsf{V}_{d-1} (F(x, u), F(\mathcal{P}, u)) := \mathsf{V}_{d-1}(F(Q, u), F(\mathcal{P}, u)) - \mathsf{V}_{d-1}(F(Q', u), F(\mathcal{P}, u))
\]
where $F(\mathcal{P}, u) := (F(P_3, u), \ldots, F(P_d, u))$ and $x = h_Q - h_{Q'}$. We can define the matrix $\widetilde{A} : \RR^\mathcal{U} \to \RR^{\mathcal{U}}$ given by 
\[
	\widetilde{A} x := \frac{1}{d} \sum_{u \in \mathcal{U}} \mathsf{V}_{d-1} (F(x, u), F(P_3, u), \ldots, F(P_d, u)) \cdot e_u.
\]
This matrix satisfies the property that 
\begin{align*}
	\langle h_Q, \widetilde{A} h_P \rangle & = \frac{1}{d} \sum_{u \in \mathcal{U}} h_Q(u) \cdot \mathsf{V}_{d-1} (F(P, u), F(P_3, u), \ldots, F(P_d, u)) \\
	& = \mathsf{V}_d (P, Q, P_3, \ldots, P_d).
\end{align*}
By linearity, we have the equality $\langle x, \widetilde{A} y \rangle = \mathsf{V}_d (x, y, P_3, \ldots, P_d)$ and $\widetilde{A}$ is a symmetric matrix. We now prove that $\widetilde{A}$ is irreducible by proving that the non-zero entries correspond exactly to the graph on facets where two facets are adjacent if and only if their intersection is a face of dimension $d-2$. 

\begin{lem} \label{lemma-structure-of-matrix-widetilde-A}
	Let $d \geq 3$. Then the matrix $\widetilde{A}$ is a symmetric irreducible matrix with non-negative off-diagonal entries. 	
\end{lem}

\begin{proof}
	Let $\mathcal{U} = \{u_1, \ldots, u_m\}$ be the facet normals of the strong isomorphism class of our polytopes where the facet corresponding to $u_i$ is $F_i$. For $i, j \in [m]$ we write $i \sim j$ if and only if $F_i \cap F_j$ is a face of dimension $d-2$. We write $F_{ij} = F_i \cap F_j$ when we consider $F_i \cap F_j$ as a facet of $F_i$. Let $u_{ij}$ be the facet normal of $F_{ij}$ in $F_i$. For $i \sim j$, let $\theta_{ij}$ be the angle satisfying $\langle u_i, u_j \rangle = \cos \theta_{ij}$. Note that $\codim F_i = \codim F_j = 1$, $\codim F_{ij} = 2$. Since no two of $u_i, u_j, u_{ij}$ are linearly independent, there are coefficients $a_i, a_{ij} \in [-1, 1]$ such that $u_j = a_{ij} u_{ij} + a_i u_i$ where $a_i^2 + a_{ij}^2 = 1$. By taking inner products with $u_i$, we get that $a_i = \cos \theta_{ij}$. This implies that $a_{ij} = \pm \cos \theta_{ij}$. By negating $\theta_{ij}$ is necessary, we have 
	\[
		u_j = (\cos \theta_{ij}) u_i + (\sin \theta_{ij}) u_{ij} \implies u_{ij} = (\csc \theta_{ij}) u_j - (\cot \theta_{ij}) u_i.
	\]
	We can then compute the support value of $F_{ij}$ as a face of $F_i$:
	\begin{align*}
		h_{F(P, u_i)} (u_{ij}) & = \sup_{x \in F(P, u_i)} \langle u_{ij}, x \rangle \\
		& = \sup_{x \in F_i(P)} \langle (\csc \theta_{ij}) u_j - (\cot \theta_{ij}) u_i, x \rangle \\
		& = (\csc \theta_{ij}) \sup_{x \in F(P, u_i)} \langle u_j, x \rangle - (\cot \theta_{ij}) h_P(u_i) \\
		& = (\csc \theta_{ij}) h_P(u_j) - (\cot \theta_{ij}) h_P(u_i).
	\end{align*}
	For $i \sim j$, we can define the constants
	\[
		A_{ij} := \frac{\mathsf{V}_{d-2}(F(F(P_3, u_i), u_{ij}), \ldots, F(F(P_d, u_i), u_{ij}))}{d(d-1)}.
	\]
	Then, for $x = h_P = \sum_{i = 1}^m x_i e_i$ where $x_i = h_P(u_i)$, we have that 
	\begin{align*}
		\widetilde{A} x & = \frac{1}{d} \sum_{i \in [m]} \mathsf{V}_{d-1} (F(P, u_i), F(P_3, u_i), \ldots, F(P_d, u_i)) \cdot e_i \\
		& = \sum_{i \in [m]}  \left ( \sum_{j \sim i} A_{ij} h_{F(P, u_i)}(u_{ij}) \right ) \cdot e_i \\
		& = \sum_{i \in [m]} \left (  \sum_{j \sim i} A_{ij} (\csc \theta_{ij}) x_j - A_{ij} (\cot \theta_{ij}) x_i) \right ) \cdot e_i \\
		& = \sum_{i \in [m]} \left ( \sum_{j \sim i} A_{ij} (\csc \theta_{ij}) x_j \right ) e_i - \sum_{i \in [m]} \left ( \sum_{j \sim i} A_{ij} \cot \theta_{ij} \right ) x_i \cdot e_i.
	\end{align*}
	For $i \in [m]$, we have that 
	\[
		(\widetilde{A})_{ii} = \langle e_i, \widetilde{A}e_i \rangle = - \sum_{j \sim i} A_{ij} \cot \theta_{ij}.
	\]
	For $i, j \in [m]$ distinct, we have 
	\[
		(\widetilde{A})_{ij} = \langle e_i, \widetilde{A}e_j \rangle = \1_{i \sim j} \cdot (A_{ij} \csc \theta_{ij}).
	\]
	When $i \sim j$, then $(\widetilde{A})_{ij} > 0$. This implies that the non-zero entries of $\widetilde{A}$ except the diagonals have the same non-zero positions as the non-zero entries in the adjacency matrix of the graph on facets where two facets are adjacent if and only if $i \sim j$. This graph is clearly strongly-connected, which proves our matrix $\widetilde{A}$ is irreducible. 
\end{proof}

\begin{thm}[Alexandrov-Fenchel Inequality for SSI Polytopes] \label{AF-SSI-version}
	Let $\alpha$ be a strong isomorphism class with facet normals $\mathcal{U} = \{u_1, \ldots, u_m\}$ of simple strongly isomorphic polytopes $P_2, \ldots, P_d$. Then, for all $x, y \in \RR^m$ we have the inequality
	\[
		\mathsf{V}_d (x, P_2, \mathcal{P})^2 \geq \mathsf{V}_d (x, x, \mathcal{P}) \cdot \mathsf{V}_d (P_2, P_2, \mathcal{P})
	\]
	where $\mathcal{P} := (P_3, \ldots, P_{d})$.
\end{thm}

The inequality of Theorem~\ref{AF-SSI-version} with respect to a bilinear form implies that the bilinear form $\mathsf{V}_d (x, y, \mathcal{P})$ has similar properties to a bilinear form with respect to a \textbf{hyperbolic} matrix. We call a symmetric matrix $M \in \RR^{d \times d}$ \textbf{hyperbolic} if for all $v, w \in \RR^d$ satisfying $\langle w, Mw \rangle \geq 0$, we have 
\[
		\langle v, Mw \rangle^2 \geq \langle v, Mv \rangle \langle w, M w \rangle. 
\]
From Lemma 1.4 in \cite{bochner}, we find necessary and sufficient conditions for a matrix to be hyperbolic. 

\begin{lem}[Lemma 1.4 in \cite{bochner}] \label{hyperbolic-quadratic-forms}
	Let $M$ be a symmetric matrix. Then, the following conditions are equivalent:
	\begin{enumerate}[label = (\alph*)]
		\item $M$ is hyperbolic. 

		\item The positive eigenspace of $M$ has dimension at most one. 
	\end{enumerate}
\end{lem}

\begin{proof}[Proof of Theorem~\ref{AF-SSI-version}]
	This proof follows that of \cite{bochner}. We induct on the dimension $d$. For the base case $d = 2$, see Lemma~\ref{base-case-for-AF-inequality} in the appendix. Now suppose that the claim is true for dimensions less than $d$. Currently, it is not clear that the matrix $\widetilde{A}$ is hyperbolic. We will alter it to become a hyperbolic matrix which is self-adjoint with respect to a different bilinear form. For $u \in \mathcal{U}$, we can define the matrix $A \in \RR^{\mathcal{U} \times \mathcal{U}}$ and diagonal matrix $P = \text{Diag}(p_u : u \in \mathcal{U}) \in \RR^{\mathcal{U} \times \mathcal{U}}$ such that 
	\begin{align*}
		Ax & := \sum_{u \in \mathcal{U}} \frac{h_{P_3}(u) \mathsf{V}_{d-1} (F(x, u), F(P_3, u), \ldots, F(P_n, u))}{\mathsf{V}_{d-1}(F(P_3, u), F(P_3, u), \ldots, F(P_n, u))} \cdot e_u\\
		p_u & := \frac{1}{d} \frac{\mathsf{V}_{d-1} (F(P_3, u), F(P_3, u), \ldots, F(P_d, u))}{h_{P_3}(u)}.
	\end{align*}
	These definitions are well-defined because we can always translate our polytopes so that $0 \in \text{int}(P_3)$ or equivalently $h_{P_3} > 0$. If we define the inner product $\langle x, y \rangle_{P} := \langle x, P y \rangle$, then we have that 
	\[
		\langle x, Ay \rangle_P = \langle x, \widetilde{A} y \rangle = \mathsf{V}_d (x, y, P_3, \ldots, P_d)
	\]
	since $\widetilde{A} = PA$. In particular, the matrix $A$ is self-adjoint with respect to the inner product $\langle \cdot, \cdot \rangle_P$. From Lemma~\ref{lemma-structure-of-matrix-widetilde-A}, we know that $A$ is irreducible with non-negative entries in the off-diagonal entries. Moreover, we have that 
	\[
		A(h_{P_3}) = \sum_{u \in \mathcal{U}}\frac{h_{P_3}(u) \mathsf{V}_{d-1} (F(P_3, u), F(P_3, u), \ldots, F(P_n, u))}{\mathsf{V}_{d-1}(F(P_3, u), F(P_3, u), \ldots, F(P_n, u))} \cdot e_u = \sum_{u \in \mathcal{U}} h_{P_3}(u) \cdot e_u = h_{P_3}.
	\]
	Thus $A$ has an eigenvector of eigenvalue $1$. Let $c > 0$ be sufficiently large so that $A + c I$ has non-negative entries. From the Perron-Frobenius Theorem (see Theorem 1.4.4 in \cite{bapat_raghavan_1997}), the vector $h_{P_3}$ is an eigenvector of eigenvalue $1 + c$ of $A + cI$. Since it has strictly positive entries this is the largest eigenvector and it happens to be simple. Thus, the largest eigenvalue of $A$ is $1$ and this eigenvalue has multiplicity $1$. From the inductive hypothesis, we have that
	\begin{align*}
		\langle Ax, Ax \rangle_P & = \sum_{u \in \mathcal{U}} (Ax)_u^2 p_u \\
		& = \sum_{u \in \mathcal{U}} \frac{1}{d} \cdot \frac{h_{P_3}(u) \mathsf{V}_{d-1}(F(x, u), F(P_3, u), \ldots, F(P_d, u))^2}{\mathsf{V}_{d-1}(F(P_3, u), F(P_3, u), \ldots, F(P_d, u))} \\
		& \geq \sum_{u \in \mathcal{U}} \frac{1}{d} h_{P_3}(u) \cdot \mathsf{V}_{d-1} (F(x, u), F(x, u), F(P_4, u), \ldots, F(P_d, u))^2 \\
		& = \mathsf{V}_d (x, x, P_3, \ldots, P_d) \\
		& = \langle x, Ax \rangle_P.
	\end{align*}
	Let $\lambda$ be an arbitrary eigenvalue of $A$. For the corresponding eigenvector $v$, we have that 
	\[
		\langle Av, Av \rangle_P \geq \langle v, Av \rangle_P \implies \lambda^2 \geq \lambda. 
	\]
	Thus, the eigenvalues satisfy $\lambda \geq 1$ or $\lambda \leq 0$. From Lemma~\ref{hyperbolic-quadratic-forms}, we know that $A$ is hyperbolic. Since $\langle h_{P_2}, A h_{P_2} \rangle_P = \mathsf{V}_d (P_2, P_2, \mathcal{P}) > 0$, we know from hyperbolicity that 
	\begin{align*}
		\mathsf{V}_d (x, P_2, \mathcal{P})^2 & = \langle x, A h_{P_2} \rangle_P^2 \\
		& \geq \langle x, Ax \rangle_P \cdot \langle h_{P_2}, Ah_{P_2} \rangle_P \\
		& = \mathsf{V}_d(x, x, \mathcal{P}) \cdot \mathsf{V}_d (P_2, P_2, \mathcal{P}).
	\end{align*}
	This completes the induction and suffices for the proof. 
\end{proof}

\begin{cor} \label{AF-log-concavity}
	Let $K, L \subseteq \RR^n$ be convex bodies. For $0 \leq k \leq n$, we can define the mixed volumes
	\[
		V_k := \mathsf{V}_n (\underbrace{K, \ldots, K}_{k \text{ times}}, \underbrace{L, \dots, L}_{n-k \text{ times}}).
	\]
	Then, the sequence $V_0, V_1, \ldots, V_n$ is log-concave. 
\end{cor}

\begin{proof}
	From Theorem~\ref{AF-inequality}, we immediately get $V_k^2 \geq V_{k-1}V_{k+1}$. This suffices for the proof. 
\end{proof}

\subsection{Equality Cases} \label{sec:AF-equality-cases}

In both Alexandrov's inequality for mixed discriminants and the Alexandrov-Fenchel inequality, the equality cases are split into two types: supercritical and critical. Unlike the Alexandrov's inequality for mixed discriminants, the equality cases of the Alexandrov-Fenchel inequality are rich and highly non-trivial even in the supercritical case. In this thesis, we do not provide any deep explanations for the mechanisms involved in the equality of the Alexandrov-Fenchel inequality as this is not our main focus. If interested in these topics, we refer the reader to \cite{minkowski-quadratic-inequality,shenfeld2022extremals}. We provide an overview of the equality cases of the Alexandrov-Fenchel inequality. We mainly list the results from \cite{shenfeld2022extremals} that we will need in Chapter~\ref{log-concavity-results}.

The Alexandrov-Fenchel inequality can be viewed as a generalization of Minkowski's Inequality and various isoperimetric inequalities. We recall the Brunn-Minkowski inequality and its equality cases in Theorem~\ref{brunn-minkowski}.
\begin{thm}[Brunn-Minkowski Inequality, Theorem 3.13 in \cite{Hug2020-ue}] \label{brunn-minkowski}
	Let $K, L \subseteq \RR^n$ be convex bodies and $\alpha \in (0, 1)$. Then
	\[
		\sqrt[n]{V(\alpha K + (1-\alpha) L)} \geq \alpha \sqrt[n]{V(K)} + (1-\alpha) \sqrt[n]{V(L)}
	\]
	with equality if and only if $K$ and $L$ lie in parallel hyperplanes or $K$ and $L$ are homothetic. 
\end{thm}

The inequality in Theorem~\ref{brunn-minkowski} will explain the equality cases of Minkowski's Inequality (Theorem~\ref{minkowski-inequality}) which follows from the Alexandrov-Fenchel inequality.

\begin{thm}[Minkowski's Inequality, Theorem 3.14 in \cite{Hug2020-ue}] \label{minkowski-inequality}
	Let $K, L \subseteq \RR^n$ be convex bodies. Then, 
	\[
		\mathsf{V}_n(\underbrace{K, \ldots, K}_{n-1 \text{ times}}, L)^n \geq \Vol_n(K)^{n-1} \cdot \Vol_n(L).
	\]
	Equality occurs if and only if $\dim K \leq n-2$ or $K$ and $L$ lie in parallel hyperplanes or $K$ and $L$ are homothetic. When $\mathsf{V}_n(K[k], L[n-k]) > 0$ for all $k$, then equality occurs if and only if 
	\[
		\mathsf{V}_n(\underbrace{K, \ldots, K}_{k \text{ times}}, \underbrace{L, \ldots, L}_{n-k \text{ times}})^2 = \mathsf{V}_n(\underbrace{K, \ldots, K}_{k-1 \text{ times}}, \underbrace{L, \ldots, L}_{n-k+1 \text{ times}}) \cdot \mathsf{V}_n(\underbrace{K, \ldots, K}_{k+1 \text{ times}}, \underbrace{L, \ldots, L}_{n-k-1 \text{ times}})
	\]
	for all $k$. 
\end{thm}

\begin{proof}
	For $k$, $0 \leq k \leq n$, define the mixed volume $V_k = \mathsf{V}_n (K[k], L[n-k])$. Then, we want to prove that $V_{n-1}^n \geq V_n^{n-1}V_0$. Assuming that $V_k > 0$ for all $k$, from Theorem~\ref{AF-inequality}, we have that 
	\[
		\frac{V_{n-1}}{V_n} \geq \frac{V_{n-2}}{V_{n-1}} \geq \ldots \geq \frac{V_{0}}{V_1}.
	\]
	Thus, we have that 
	\[	
		\left ( \frac{V_{n-1}}{V_n} \right )^n \geq \prod_{k = 1}^{n-1} \frac{V_k}{V_{k-1}} = \frac{V_0}{V_n}.
	\]
	This proves that $V_{n-1}^n \geq V_n^{n-1} \cdot V_0$. When $V_k > 0$, equality occurs if and only if all of the ratios $V_k / V_{k-1}$ are equal, which is equivalent to the Alexandrov-Fenchel inequality holding at each $k$: $V_k^2 = V_{k-1} V_{k+1}$. Now, suppose we remove the assumption that $V_k > 0$ for all $k$. Then we can still prove the inequality via approximation by SSI polytopes. This method would not allow us to recover equality cases. From the proof of Theorem 3.14 in \cite{Hug2020-ue}, using the concavity of $f(t) := \Vol_n (K + t L)^{1/n}$ we can prove that 
	\[
		\Vol_n(K)^{\frac{1}{n} - 1} \mathsf{V}_n (K[n-1], L) \geq \Vol_n(K+L)^{\frac{1}{n}} - \Vol_n(K)^{\frac{1}{n}} \geq \Vol_n(L)^{\frac{1}{n}}
	\]
	where the second inequality is exactly Theorem~\ref{brunn-minkowski}. The equality conditions in Theorem~\ref{brunn-minkowski} then give us the equality conditions in the theorem. 
\end{proof}

In general, the equality cases of the Alexandrov-Fenchel inequality are complicated. Indeed, the previous pattern of having only degenerate equality cases or homothetic equality cases ends when one considers even the simplest examples of Minkowski's Quadratic inequality. Minkowski's quadratic ienquality is version of the Alexandrov-Fenchel inequality where $n = 3$ and we have convex bodies $K, L, M \subseteq \RR^3$ satisfying 
\begin{equation} \label{minkowski-quadratic-inequality-model}
	\mathsf{V}_3(K, L, M)^2 \geq \mathsf{V}_3 (K, K, M) \mathsf{V}_3 (L, L, M). 
\end{equation}
In \cite{minkowski-quadratic-inequality}, Yair Shenfeld and Ramon van Handel completely characterize the extremals of Equation~\ref{minkowski-quadratic-inequality-model}. Even in the case where $M = K, L = B$ the equality cases are non-trivial. The inequality becomes
\begin{equation} \label{isoperimetric-inequality-in-degree-3}
	\mathsf{V}_3 (B, K, K)^2 \geq \mathsf{V}_3(B, B, K) \mathsf{V}_3(K, K, K). 
\end{equation}
Each of these mixed volumes has a geometric interpretation: $\mathsf{V}_3 (B, K, K)$ is the surface area of $K$, $\mathsf{V}_3 (B, B, K)$ is the mean width of $K$, and $\mathsf{V}_3 (K, K, K)$ is the volume of $K$. The problem of finding equality cases to Equation~\ref{isoperimetric-inequality-in-degree-3} can be rephrased as an isoperimetric inequality: given a fixed mean width and fixed volume, what convex body $K$ will achieve the minimum surface area? This is a generalization of the classical isoperimetric inequality:
\[
	\mathsf{V}_2(B, K)^2 \geq \mathsf{V}_2(B, B) \mathsf{V}_2(K, K)
\]
which the equality cases is trivial. Unlike the classical case, the equality cases of Equation~\ref{isoperimetric-inequality-in-degree-3} involve convex bodies called cap bodies which are the convex hull of the ball with some collection of points satisfying some disjointness conditions. From \cite{Bol}, the cap bodies end up being the only equality cases of the quadratic isoperimetric inequality. 

\subsection{Equality Cases for Polytopes}

In \cite{minkowski-quadratic-inequality}, Yair Shenfeld and Ramon van Handel completely characterize the equality cases of the Alexandrov-Fenchel inequality for polytopes. We record some of their results in this section for future use in this thesis. In \cite{shenfeld2022extremals}, they define the notion of a supercritical collection of convex bodies. 

\begin{defn}[Definition 2.14 in \cite{shenfeld2022extremals}]
	A collection of convex bodies $\mathcal{C} = (C_1, \ldots, C_{n-2})$ is \textbf{supercritical} if $\dim (C_{i_1} + \ldots + C_{i_k}) \geq k+2$ for all $k \in [n-2]$, $1 \leq i_1 < \ldots < i_k \leq n-2$. 
\end{defn} 

The paper \cite{shenfeld2022extremals} gives a characterization of the equality cases of 
\[
	\mathsf{V}_n(K, L, P_1, \ldots, P_{n-2})^2 \geq \mathsf{V}_n (K, K, P_1, \ldots, P_{n-2}) \mathsf{V}_n (L, L, P_1, \ldots, P_{n-2})
\]
in both the case where $(P_1, \ldots, P_{n-2})$ is critical and the case where $(P_1, \ldots, P_{n-2})$ is supercritical. In the latter case, the equality cases are more simple. Since all of our applications will involve supercritical collections of convex bodies, we will only be concerned with the supercritical characterization. For an application of the critical case, see the paper \cite{ma2022extremals} in which the equality cases of the Stanley's poset inequality are studied.
	
\begin{thm}[Corollary 2.16 in \cite{shenfeld2022extremals}] \label{AF-equality-polytopes-theorem}
	Let $\mathcal{P} = (P_1, \ldots, P_{n-2})$ be a supercritical collection of polytopes in $\RR^n$, and let $K, L$ be convex bodies such that $\mathsf{V}_n (K, L, P_1, \ldots, P_{n-2}) > 0$. Then 
	\[
		\mathsf{V}_n (K, L, P_1, \ldots, P_{n-2})^2 = \mathsf{V}_n (K, K, P_1, \ldots, P_{n-2}) \mathsf{V}_n (L, L, P_1, \ldots, P_{n-2})
	\]
	if and only if there exist $a > 0$ and $v \in \RR^n$ so that $K$ and $a L + v$ have the same supporting hyperplanes in all $(B, \mathcal{P})$-extreme normal directions.
\end{thm}

In Theorem~\ref{AF-equality-polytopes-theorem}, the constant $a > 0$ determines the ratio between the mixed volumes. To see this, define $\mathsf{V} (A, B) := \mathsf{V}_n(A, B, P_1, \ldots, P_{n-2})$ for any convex bodies $A, B \subseteq \RR^n$. From the mixed area measure interpretation of mixed volumes, the value of $\mathsf{V}(A, B)$ only depends on the support function values of $A$ and $B$ in the directions of the support of the mixed area measure. From the result in Theorem~\ref{AF-equality-polytopes-theorem}, the polytopes $K$ and $aL+v$ have the same support values in all relevant directions. Thus, we can interchange the two polytopes when calculating the mixed volume. This implies that
\begin{align*}
	\mathsf{V}(K, K) & = \mathsf{V}(K, aL + v) = a \mathsf{V}(K, L) \\
	\mathsf{V}(K, L) & = \mathsf{V}(aL+v, L) = a \mathsf{V} (L, L).
\end{align*}
Hence, the constant $a > 0$ indicates the ratio between the mixed volumes.

\section{Basic Theory of Lorentzian polynomials} \label{sec:lorentzian-polynomials}

In this section, we give an overview of the theory of Lorentzian polynomials. Lorentzian polynomials are polynomials which encode log-concavity in various ways. According to \cite{baker_2022}, the notion of Lorentzian polynomials was independently discovered by Brändén-Huh \cite{lorentzian-polynomials} and Anari-Oveis-Garan-Vinzant \cite{anari2018logconcave}. For our main reference on the basic theory of Lorentzian polynomials, we primarily use the paper \cite{lorentzian-polynomials} by Brändén-Huh. Following their notation, we let $\opH_n^d$ be the degree $d$ homogeneous subring defined by 
\[
	\opH_n^d := \{f \in \RR[x_1, \ldots, x_n] : \text{$f$ is homogeneous,} \deg (f) = d\}.
\] 
We can turn $\opH_n^d$ into a topological space by equipping it with the Euclidean topology on its coefficients. There are many different equivalent definitions of Lorentzian polynomials. These different definitions make it easier to prove that certain polynomials are Lorentzian. We first define the notion of a strictly Lorentzian polynomial in Definition~\ref{def:strictly-lorentzian-polynomials}. 

\begin{defn} \label{def:strictly-lorentzian-polynomials}
	Let $\underbar{L}_n^2 \subseteq \opH_n^2$ be the set of quadratic forms in $\RR[x_1, \ldots, x_n]$ with positive coefficients such that the Hessian has signature $(+, -, \ldots, -)$. For $d \geq 3$, we can define $\underbar{L}_n^d$ inducively as
	\[
		\underbar{L}_n^d := \{f \in H_n^d : \partial_i f \in \underbar{L}_n^{d-1} \text{ for all $i$}\}.
	\] 
	We call polynomials in $\underbar{L}_n^d$ \textbf{strictly Lorentzian polynomials}. 
\end{defn}

The space of strictly Lorentzian polynomials $\underbar{L}_n^d$ forms an open subset of $\opH_n^d$. One defintion of Lorentzian polynomials is that it is any polynomial in the closure of $\underbar{L}_n^d$ as a subset in the topological space $\opH_n^d$. The next definition of Lorentzian polynomials combines analytic and combinatorial properties of the polynomial. Specifically, the properties of the polynomial as a function on $\mathbb{C}^n$ and the support structure of its monomials. 

\begin{defn} \label{def:stable-polynomial}
	A polynomial $f \in \RR[w_1, \ldots, w_n]$ is \textbf{stable} if $f$ is non-vanishing on $\mathbb{H}^n$ where $\mathbb{H}$ is the open upper half plane in $\mathbb{C}$. We denote by $S_n^d$ the set of degree $d$ homogeneous stable polynomials in $n$ variables. 
\end{defn}

In this thesis, we will not concern ourselves with the notion of stable polynomials. We have only written Definition~\ref{def:stable-polynomial} because it is a part of an alternative definition of Lorentzian polynomials. For background on stable polynomials, we refer the reader to \cite{wagner2009multivariate}. 

\begin{defn}
	We define $J \subseteq \NN^n$ to be \textbf{$M$-convex} if for any $\alpha, \beta \in J$ and any index $i$ satisfying $\alpha_i > \beta_i$, there is an index $j$ satisfying $\alpha_j < \beta_j$ and $\alpha - e_i + e_j \in J$. Equivalently, we have that for any $\alpha, \beta \in J$ and any index $i$ satisfying $\alpha_i > \beta_i$, there is an index $j$ satisfying $\alpha_j < \beta_j$ and $\alpha - e_i + e_j \in J$ and $\beta - e_j + e_i \in J$. 
\end{defn}

We will also not concern ourselves too much with $M$-convex sets. In this thesis, the only fact that we use related to $M$-convex sets is that the set of bases of a matroid form a $M$-convex set. Indeed, any M-convex set restricted to the hypercube is a matroid. For more information about $M$-convex sets, we refer the reader to \cite{discete-convex-analysis}. We will consider the $M$-convexity of the support of a homogeneous polynomial. For a polynomial $f = \sum_{\alpha \in \Delta_n^d} c_\alpha x^\alpha$ in $H_n^d$, we define the support of $f$ to be the exponents for which the corresponding coefficient is non-zero. Specifically, we have
\[
	\supp (f) := \{\alpha \in \Delta_n^d : c_\alpha \neq 0\}.
\]
For example, the Newton polygon of any polynomial will depend only on its support. We let $M_n^d$ be the set of degree $d$ homogeneous polynomials in $\RR_{\geq 0} [x_1, \ldots, x_n]$ whose support are $M$-convex. With these notions on hands, we can now define Lorentzian polynomials equivalently as the homogeneous polynomials with non-negative coefficients and $M$-convex support such that all codegree $2$ partial derivatives are stable polynomials. We write the full description of this set in Definition~\ref{def:lorentzian-m-convex-defn}.

\begin{defn}[Definition 2.6 in \cite{lorentzian-polynomials}] \label{def:lorentzian-m-convex-defn}
	We set $L_n^0 = S_n^0$, $L_n^1 = S_n^1$, and $L_n^2 = S_n^2$. For $d \geq 3$, we define 
	\begin{align*}
		L_n^d & = \{f \in M_n^d : \partial_i f \in L_n^{d-1} \text{ for all } i \in [n]\} \\
		& = \{f \in M_n^d : \partial^\alpha f \in S_n^2 \text{ for all } \alpha \in \Delta_n^{d-2}\}.
	\end{align*}
\end{defn}

\begin{defn} \label{defn-lorentzian}
	We call a degree $d$ homogeneous polynomial $f \in \RR_{\geq 0}[x_1, \ldots, x_n]$ a \textbf{Lorentzian polynomial} if it satisfies any of the following equivalent conditions:
	\begin{enumerate}[label = (\alph*)]
		\item $f \in L_n^d$. 

		\item There exists a sequence of polynomials $f_k \in \underbar{L}_n^d$ such that $f_k \to f$ as $k \to \infty$.  

		\item For any $\alpha \in \NN^n$ with $|\alpha| \leq d-2$, the polynomial $\partial^\alpha f$ is identically zero or log-concave at any $a \in \RR_{>0}^n$. 
	\end{enumerate}
\end{defn}

Polynomials which satisfy condition (c) in Definition~\ref{defn-lorentzian} are also known as \textbf{strongly log-concave polynomials}. These polynomials were studies in \cite{anari2018logconcave}. For proofs of the equivalences of the conditions in Definition~\ref{defn-lorentzian}, we refer the reader to \cite{lorentzian-polynomials}. As an application of this definition, we will now prove that the basis generating polynomial of a matroid is Lorentzian.

\begin{thm} \label{basis-generating-polynomial-is-lorentzian-thm}
	Let $M$ be a matroid. Then the basis generating polynomial $f_M$ is Lorentzian. 
\end{thm}
\begin{proof}
	We prove the result by induction on the size of our matroid. For matroids of size $2$, the only possible basis generating polynomials are $x_1x_2, x_1 + x_2, x_1, x_2$, and $0$. All of these polynomials are Lorentzian because they are stable. Now, suppose that the claim holds for all matroids of smaller size. Since the support of $f_M$ is the collection of bases of a matroid, the support is $M$-convex. Hence, it suffices to prove that all of its partial derivatives are Lorentzian. For all $i \in E$, we know that $\partial_i f_M = f_{M / i}$ for all $i \in E$. From the inductive hypothesis, the partial derivative is Lorentzian because it the basis generating polynomial of a smaller matroid. This completes the induction and suffices for the proof.  
\end{proof}

In Proposition~\ref{proposition-properties-of-Lorentzian-polynomials}, we have compiled some properties of Lorentzian polynomials. These properties will allow us to prove that certain polynomials are Lorentzian by relating them to other Lorentzian polynomials. 

\begin{prop} \label{proposition-properties-of-Lorentzian-polynomials}
	Let $f \in \RR[x_1, \ldots, x_n]$ be a Lorentzian polynomial of degree $d$. Then, $f$ satisfies the following properties: 
	\begin{enumerate}[label = (\alph*)]
		\item Let $A$ be any $n \times m$ matrix with nonnegative entries. Then $f : \RR^m \to \RR$ defined by $f(y) := f(Ay)$ is Lorentzian. 

		\item For any $a_1, \ldots, a_n \geq 0$, the polynomial $\sum_{i = 1}^n a_i \partial_i f$ is Lorentzian. 

		\item If $f \neq 0$, then $\Hess_f(a)$ has exactly one positive eigenvalue for all $a \in \RR_{> 0}^n$. 
	\end{enumerate}
\end{prop}
\begin{proof}
	For a proof of these properties, see Theorem 2.10, Corollary 2.11, and Proposition 2.14 in \cite{lorentzian-polynomials}. 
\end{proof}

Lorentzian polynomials satisfy log-concavity inequalities similar to that of Alexandrov's inequality on mixed discriminants and the Alexandrov-Fenchel inequality. We present the proofs of two analogs of these log-concavity inequalities given in \cite{lorentzian-polynomials}. 

\begin{prop}[Proposition 4.4 in \cite{lorentzian-polynomials}] \label{proposition-log-concavity-property-of-lorentzian-polynomial}
	If $f = \sum_{\alpha \in \Delta_n^d} \frac{c_\alpha}{\alpha!} x^\alpha$ is a Lorentzian polynomial, then $c_\alpha^2 \geq c_{\alpha + e_i - e_j} c_{\alpha - e_i + e_j}$ for any $i, j \in [n]$ and any $\alpha \in \Delta_n^d$. 
\end{prop}

\begin{proof}
	Consider the differential operator $\partial^{\alpha - e_i - e_j}$. Applying $D$ to the polynomial $f$, we have 
	\begin{align*}
		D f & = \sum_{\beta \in \Delta_n^d} \frac{c_\beta}{\beta!} \partial^{\alpha - e_i - e_j} x^\beta \\
		& = \frac{c_{\alpha + e_i - e_j}}{(\alpha + e_i - e_j)!} \cdot \frac{(\alpha + e_i - e_j)!}{2} x_i^2 + \frac{c_{\alpha - e_i + e_j}}{(\alpha - e_i + e_j)!} \cdot \frac{(\alpha - e_i + e_j)!}{2} x_j^2 + \frac{c_\alpha}{\alpha !} \cdot \alpha! \cdot x_ix_j \\
		& = \frac{1}{2} \left ( c_{\alpha + e_i - e_j} x_i^2 + 2 c_\alpha x_i x_j + c_{\alpha - e_i + e_j} x_j^2 \right ).
	\end{align*}
	From Proposition~\ref{proposition-properties-of-Lorentzian-polynomials}, the polynomial $Df$ is Lorentzian. This suffices for the proof. 
\end{proof}

Recall that the mixed discriminant $\mathsf{D}_n(A_1, \ldots, A_n)$ can be defined as the polarization form of the polynomial $\det (X)$ where $X \in \RR^{n \times n}$. It happens to be the case that the polarization of this polynomial satisfies a log-concavity inequality given by the Alexandrov inequality for mixed discriminants. An analog of this phenomenon occurs for general Lorentzian polynomials. 

\begin{prop} \label{prop:analog-of-alexandrov-fenchel}
	Let $f \in \RR[x_1, \ldots, x_n]$ be a Lorentzian polynomial of degree $d$. Then, for any $v_1, v_2, \ldots, v_d \in \RR_{\geq 0}^n$, we have that 
	\[
		F_f(v_1, v_2, v_3, \ldots, v_d)^2 \geq F_f(v_1, v_1, v_3, \ldots, v_d) \cdot F_f(v_2, v_2, v_3, \ldots, v_d).
	\] 
\end{prop}

\begin{proof}
	The polarization identity in Theorem~\ref{polariziation-identity} gives us the equation 
	\begin{equation} \label{eqn:yesyesyes}
		f(x_1 v_1 + \ldots + x_m v_m) = d! \sum_{\alpha \in \Delta_m^d} \frac{F_f(v_1[\alpha_1], \ldots, v_m[\alpha_m])}{\alpha!} x^\alpha
	\end{equation}
	The result follows from Proposition~\ref{proposition-log-concavity-property-of-lorentzian-polynomial}. 
\end{proof}

\begin{remark}
	In Proposition 4.5 of \cite{lorentzian-polynomials}, Br\"aden-Huh prove a stronger version of Proposition~\ref{prop:analog-of-alexandrov-fenchel} where we allow the first vector $v_1$ to be any vector in $\RR^n$. The proof of this stronger fact uses the Cauchy interlacing theorem from Section~\ref{sec:cauchy-interlacing-theorem}. This stronger fact also follows from the version we proved in Proposition~\ref{prop:analog-of-alexandrov-fenchel}. Indeed, suppose that $v_1 \in \RR^n$ and $v_2 \in \RR_{> 0}^n$. Then, we can pick $\alpha > 0$ sufficiently large so that $v_1 + \alpha v_2 \in \RR_{> 0}^n$. To simplify notation, we define 
	\[
		\operatorname{F}(x, y) := F_f(x, y, v_3, \ldots, v_d).
	\]
	From Proposition~\ref{prop:analog-of-alexandrov-fenchel}, we have that
	\[
		\operatorname{F}(v_1 + \alpha v_2, v_2)^2 \geq \operatorname{F}(v_1 + \alpha v_2, v_1 + \alpha v_2) \operatorname{F}(v_2, v_2).
	\]
	This is equivalent to $\operatorname{F}(v_1, v_2) \geq \operatorname{F}(v_1, v_1) \operatorname{F}(v_2, v_2)$. From continuity, the inequality holds for all $v_2 \in \RR_{\geq 0}^n$. 
\end{remark}
\begin{remark}
	Through private discussions with Ramon van Handel, it may be possible to use the ideas in \cite{shenfeld2022extremals} to give a characterization of the equality cases of Proposition~\ref{prop:analog-of-alexandrov-fenchel}. In the expansion of Equation~\ref{eqn:yesyesyes} for the case $m = n$, by substituting $v_i = e_i$ for all $1 \leq i \leq n$ we get the equation
	\[
		f(x_1, \ldots, x_n) = d! \sum_{\alpha \in \Delta_n^d} \frac{F_f(e_1 [\alpha_1], \ldots, e_n [\alpha_n])}{\alpha!} x^\alpha.
	\]
	Thus a characterization for the Lorentzian analog of a characterization for the extremals of the Alexandrov-Fenchel inequality (as given in \cite{shenfeld2022extremals}) would be useful in characterizing equality in the inequalities given in Proposition~\ref{proposition-log-concavity-property-of-lorentzian-polynomial}. The characterization would involve checking when the vectors $e_1, \ldots, e_n$ satisfy the equaliy conditions. We plan to pursue this direction in future work.
\end{remark}

\chapter{Log-concavity Results for Posets and Matroids} \label{log-concavity-results}

In this chapter, we discuss proofs of several log-concavity results using mixed discriminants, mixed volumes, and Lorentzian polynomials. In the first section, we write about Stanley's poset inequality where he proved that a sequence which enumerates linear extensions based on where it sends a fixed element is log-concave. We also sketch the proof of the combinatorial characterization of Stanley's inequality given in \cite{shenfeld2022extremals}. We then use the same proof strategy to completely characterize the extremals of the Kahn-Saks inequality. This was done in joint work with Ramon van Handel, Xinmeng Zeng, and the author. In the final section, we study another inequality of Stanley about the log-concavity of a basis counting sequence for regular matroids. By combining perspectives from both mixed discriminants and mixed volumes, we characterize the equality cases of a simplified version of Stanley's matroid inequality. We then generalize Stanley's matroid inequality to all matroids using the technology of Lorentzian polynomials. 

\section{Stanley's Poset Inequality} \label{stanley-poset-inequality}

Let $(P, \leq)$ be a finite poset on $n$ elements and let $x \in P$ be a distinguished element in the poset. For every $k \in [n]$, let $N_k$ be the number of linear extensions of $P$ which send $x$ to the index $k$. Then we can consider properties of the sequence $N_1, \ldots, N_n$. Intuitively, one would expect that as a function in $[n]$, the sequence should be roughly unimodal. The fact that the sequence is unimodal is originally a conjecture by Ronald Rivest in his paper \cite{stanley-poset-inequality-origin}. Rivest proves his conjecture in the case where $P$ can be covered by two linear orders. In \cite{STANLEY}, Richard Stanley proves the full conjecture by proving the stronger claim that the sequence is log-concave.

\begin{thm}[Stanley's Poset Inequality, Theorem 3.1 in \cite{STANLEY}] \label{stanley-poset-inequality-general-case}
	Let $x_1 < \ldots < x_k$ be a fixed chain in $P$. If $1 \leq i_1 < \ldots < i_k \leq n$, then define $N(i_1, \ldots, i_k)$ to be the number of linear extensions $\sigma : P \to [n]$ satisfying $\sigma (x_j) = i_j$ for $1 \leq j \leq k$. Suppose that $1 \leq j \leq k$ and $i_{j-1} + 1 < i_j < i_{j+1} - 1$, where $i_0 = 0$ and $i_{k+1} = n+1$. Then 
	\[
		N(i_1, \ldots, i_k)^2 \geq N(i_1, \ldots, i_{j-1}, i_j - 1, i_{j+1}, \ldots, i_k) N(i_1, \ldots, i_{j-1}, i_j + 1, i_{j+1}, \ldots, i_k).
	\]
\end{thm}

Rivest's conjecture follows from Theorem~\ref{stanley-poset-inequality-general-case} in the case $k = 1$. We write the statement of this simplified inequality in Corollary~\ref{stanley-inequality-poset-simple} and call it Stanley's Simple Poset Inequality.

\begin{cor}[Stanley's Simple Poset Inequality] \label{stanley-inequality-poset-simple}
	The sequence $N_1, \ldots, N_n$ is log-concave. That is, we have $N_k^2 \geq N_{k-1}N_{k+1}$ for all $k \in \{2, \ldots, n-1\}.$
\end{cor}

We will present the proof to Corollary~\ref{stanley-inequality-poset-simple} rather than Theorem~\ref{stanley-poset-inequality-general-case} since they are essentially the same up to technical computations.

\begin{proof}[Proof of Corollary~\ref{stanley-inequality-poset-simple}]
	Define the polytopes $K, L \subseteq \RR^{n}$ by 
	\begin{align*}
		K & := \{t \in \mcO_P : t_\omega = 1 \text{ if } \omega \geq x\} \\
		L & := \{t \in \mcO_P : t_\omega = 0 \text{ if } \omega \leq x\}.
	\end{align*}
	For any $\lambda \in [0, 1]$, we have that 
	\[
		\Vol_{n-1} \left ( (1 - \lambda) K + \lambda L \right ) = \sum_{l \in e(P)}  \Vol_{n-1} \left ( \Delta_l \cap \left((1-\lambda) K + \lambda L\right) \right ).
	\]
	Even though $K, L$ are in $\RR^n$, they lie in parallel hyperplanes. This allows us to take $(n-1)$ dimensional volumes and mixed volumes. Now suppose that $l$ satisfies $l(x) = k$. Then, we have that $\Delta_l \cap \left((1-\lambda) K + \lambda L \right)$ is equal to 
	\[ 
		\{0 \leq t_{\pi^{-1}(1)} \leq \ldots \leq t_{\pi^{-1}(k-1)} \leq 1-\lambda\} \cap \{t_{\pi^{-1}(k)} = 1-\lambda\} \cap \{1 - \lambda \leq t_{\pi^{-1}(k+1)} \leq \ldots \leq t_{\pi^{-1}(n)} \leq 1\}.
	\]
	This is essentially the product of two simplices and has volume given by
	\[
		\Vol_{n-1} (\Delta_l \cap \left((1 - \lambda) K + \lambda L\right)) = \frac{(1-\lambda)^{k-1} \lambda^{n-k}}{(k-1)! (n-k)!}.
	\]
	Thus, we have the equation 
	\begin{align*}
		\Vol_{n-1}\left((1-\lambda) K + \lambda L\right) & = \sum_{k = 1}^n \binom{n-1}{k-1} \frac{N_k}{(n-1)!} (1-\lambda)^{k-1} \lambda^{n-k}.
	\end{align*}
	From Theorem~\ref{mixed-volume-polynomial-expansion-FINAL}, we have for $1 \leq i \leq n-1$ the equality
	\[
		\mathsf{V}_{n-1} (K[i-1], L[n-i]) = \frac{N_i}{(n-1)!} \implies N_i = (n-1)! \cdot \mathsf{V}_{n-1} (K[i-1], L[n-i]).
	\]
	The log-concavity of the sequence $N_1, N_2, \ldots, N_n$ follows from Theorem~\ref{AF-log-concavity}. 
\end{proof}

The proof of Corollary~\ref{stanley-inequality-poset-simple} illustrates one strategy for proving log-concavity results. To prove that a sequence is log-concave, it is enough to associate each element in the sequence to some mixed volume. Log-concavity then follows immediately from the Alexandrov-Fenchel inequality. In the sequel, we will discuss how to use the results about the extremals of the Alexandrov Fenchel inequality from \cite{shenfeld2022extremals} to give a combinatorial characterization of the equality cases of $N_k^2 = N_{k-1} N_{k+1}$ in Stanley's simple poset inequality. 

\subsection{Equality Case for the Simple Stanley Poset Inequality}

The equality cases for the Stanley poset inequality was computed in \cite{shenfeld2022extremals} using their results on the equality cases of the Alexandrov-Fenchel inequality for polytopes. From the proof of Corollary~\ref{stanley-inequality-poset-simple}, we have the equalities $N_k = (n-1)! \mathsf{V}_{n-1} (K[k], L[n-k-1])$ where the polytopes $K$ and $L$ were defined in the proof. Thus, to find the equality cases of $N_k^2 = N_{k-1} N_{k+1}$, it is enough to find the equality cases of 
\begin{equation} \label{stanley-equivalent-mixed-volume-inequality-interpretation-eqn}
	\mathsf{V}_{n-1}(K[k-1], L[n-k])^2 \geq \mathsf{V}_{n-1}(K[k], L[n-k-1]) \cdot \mathsf{V}_{n-1}(K[k-2], L[n-k+1])
\end{equation}
Our goal is to give combinatorial conditions to the poset $P$ so that the corresponding polytopes $K$ and $L$ satisfy equality in Equation~\ref{stanley-equivalent-mixed-volume-inequality-interpretation-eqn}. We sketch the proof given in Section 15 in \cite{shenfeld2022extremals} to give an idea on how the argument works. The same type of argument will be used to charaterize the extremals of the Kahn-Saks inequality. The combinatorial characterization given in \cite{shenfeld2022extremals} is compiled in Theorem~\ref{combinatorially-characterization-of-stanley-poset-simple}. 

\begin{thm}[Theorem 15.3 in \cite{shenfeld2022extremals}] \label{combinatorially-characterization-of-stanley-poset-simple}
	Let $i \in \{2, \ldots, n-1\}$ be such that $N_i > 0$. Then the following are equivalent:
	\begin{enumerate}[label = (\alph*)]
		\item $N_i^2 = N_{i-1} N_{i+1}$. 
		\item $N_i = N_{i-1} = N_{i+1}$. 
		\item Every linear extension $\sigma : P \to [n]$ with $\sigma(x) = i$ assigns ranks $i-1$ and $i+1$ to elements of $P$ that are incomparable to $x$. 
		\item $|P_{< y}| > i$ for all $y \in P_{> x}$, and $|P_{> y}| > n-i+1$ for all $y \in P_{< x}$. 
	\end{enumerate}
\end{thm}

In Theorem~\ref{combinatorially-characterization-of-stanley-poset-simple}, it is not difficult to prove that $(d)\implies(c)\implies(b)\implies(a)$. The main difficulty lies in proving that equality and positivity in $N_i^2 = N_{i-1} N_{i+1}$ implies the combinatorial conditions in (d). From here, let us assume that $N_i^2 = N_{i-1} N_{i+1}$ and $N_i > 0$. From Lemma~\ref{positivity-of-mixed-volumes}, we get necessary and sufficient conditions to guarantee $N_i > 0$. 

\begin{lem} [Lemma 15.2 in \cite{shenfeld2022extremals}] \label{lem:positivity-of-stanley-sequence-posets}
	For any $i \in [n]$, we have that $N_i = 0$ if and only if $|P_{< x}| > i-1$ or $|P_{> x}| > n-i$. 
\end{lem}

\begin{proof}
	Recall that $N_i = \mathsf{V}_{n-1} (\underbrace{K, \ldots, K}_{i-1 \text{ times}}, \underbrace{L, \ldots, L}_{n-i \text{ times}})$. From Lemma~\ref{positivity-of-mixed-volumes}, we know that $N_i > 0$ if and only if $\dim K \geq i-1$, $\dim L \geq n-i$, and $\dim (K+L) \geq n-1$. From the definition of $K$ and $L$, we can compute
	\begin{align*}
		\dim K & = n-1-|P_{> x}| \\
		\dim L & = n-1 - |P_{< x}| \\
		\dim (K+L) & = n-1. 
	\end{align*} 
	To see the details behind this computation, see Lemma 15.7 in \cite{shenfeld2022extremals}. Thus, $N_i = 0$ if and only if $|P_{> x}| > n-i$ and $|P_{<x}| > i-1$. 
\end{proof}

We get inequalities on the dimensions of $\dim K$ and $\dim L$ from the assumption $N_i > 0$. However, since $N_i^2 = N_{i-1}N_{i+1}$, we also have $N_{i-1}, N_{i+1} > 0$. This gives slightly stronger inequalities from Lemma~\ref{lem:positivity-of-stanley-sequence-posets} applied to $N_{i-1}$ and $N_{i+1}$. Specifically, we get the inequalities
\begin{align*}
	|P_{< x}| \leq i-2, \quad \text{and } \quad |P_{>x}| \leq n-i-1. 
\end{align*}
With these inequalities and Theorem~\ref{AF-equality-polytopes-theorem} we get a specialized version of Corollary 2.16 in \cite{shenfeld2022extremals}.

\begin{lem} [Lemma 15.8 in \cite{shenfeld2022extremals}] \label{stanley-lemma-to-extract-combinatorial-information}
	Let $i \in \{2, \ldots, n-1\}$ be such that $N_i > 0$ and $N_i^2 = N_{i+1} N_{i-1}$. Then $|P_{<x}| \leq i-2$, $|P_{>x}| \leq n-i-1$, and there exist $a > 0$ and $v \in \RR^{n-1}$ so that 
	\[
		h_K(x) = h_{aL + v}(x) \text{ for all } x \in \supp S_{B, K[i-2], L[n-i-1]}.
	\]

\end{lem}

In the beginning of our analysis, the convex bodies $K$ and $L$ were contained in parallel hyperplanes of $\RR^{n-1}$. When studying equality via Lemma~\ref{stanley-lemma-to-extract-combinatorial-information}, we project these bodies to the same copy of $\RR^{n-1}$. The vector $v \in \RR^{n-1}$ will also lie in the same hyperplane as our convex bodies $K$ and $L$. Lemma~\ref{stanley-lemma-to-extract-combinatorial-information} will be the main tool to extract combinatorial information about $P$ from the equality in the Alexandrov-Fenchel inequality. Recall in Definition~\ref{extreme-normal-directions}, we are giving a characterization for vectors in the support of the mixed area measure. By specializing this result in the Stanley case, we get the result in Lemma~\ref{support-vector-stanley-characterization}. 

\begin{lem} \label{support-vector-stanley-characterization}
	Let $u \in \RR^{n-1}$ be a unit vector. Then the following are equivalent. 
	\begin{enumerate}[label = (\alph*)]
		\item $u \in \supp S_{B, K[i-2], L[n-i-1]}$.

		\item $\dim F_K(u) \geq i-2$, $\dim F_L(u) \geq n-i-1$, and $\dim F_{K+L}(u) \geq n-3$. 
	\end{enumerate}
\end{lem}

From here, the argument will involve finding suitable vectors $u \in \mathbb{S}^{n-2} \subseteq \RR^{n-1}$ and applying Lemma~\ref{stanley-lemma-to-extract-combinatorial-information}. The suitability of the vector $u \in \mathbb{S}^{n-2}$ will depend on the corresponding inequalities in Lemma~\ref{support-vector-stanley-characterization}. When we find a ``suitable" vector $u \in \mathbb{S}^{n-2}$, one of two things generally occur. In the first case, the vector $u$ will unequivocally satisfy the dimension inequalities in Lemma~\ref{support-vector-stanley-characterization} and from Lemma~\ref{stanley-lemma-to-extract-combinatorial-information} we get the relation $h_K(u) = a h_L(u) + \langle v, u \rangle$. In this case, the identity we get usually tells us something about $v$ or $a$. In the second case, it is unclear whether or not the vector $u$ will satisfy the dimension inequalities in Lemma~\ref{support-vector-stanley-characterization}. The dimension inequalities will usually involve some statistics about the combinatorial structure of our object. For example, it might involve the number of elements greater than another, or the number of elements between two elements in a poset. In this case, we know that if the dimension inequalities in Lemma~\ref{support-vector-stanley-characterization} are satisfied then we get the identity in Lemma~\ref{stanley-lemma-to-extract-combinatorial-information}. If we are lucky, the resulting identity we get will contradict some information that we know is true from the vectors in the first case. This will imply that at least one of the dimension inequalities are incorrect, giving a combinatorial condition that must be satisfied for equality. 

In the Stanley case, vectors in the first case are vectors of the form $-e_{\omega}$ where $\omega$ is a minimal element, $e_\omega$ where $\omega$ is a maximal element, and some vectors of the form $e_{\omega_1,\omega_2} := \frac{e_{\omega_1} - e_{\omega_2}}{\sqrt{2}}$ where $\omega_1, \omega_2 \in P$ are elements in the poset such that $\omega_1 \lessdot \omega_2$. From the first two vectors, we get that $v_\omega = 1-a$ for any maximal element $\omega$, the second vector gives $v_\omega = 0$ for any minimal element $\omega$. The third type of vector tells us that in some situations where $\omega_1 \lessdot \omega_2$ we have $v_{\omega_1} = v_{\omega_2}$. From the proof of Corollary 15.11 of \cite{shenfeld2022extremals}, it is possible to create a chain from a minimal element to a maximal element such as the adjacent relations are covering relations where the corresponding vector $e_{\omega_1, \omega_2}$ is in the support of the mixed area measure. This implies that all of the coordinates of $v$ with respect to this chain are equal to each other. In particular, the coordinate at the minimal element will be equal to the coordinate at the maximal element (that is, $a = 1$). This result is interesting in its own right. We have just proved that if $N_i^2 = N_{i-1} N_{i+1}$ and $N_i > 0$, then $N_{i+1} = N_i = N_{i-1}$.

Vectors in the second case consist of vectors of the form $-e_\omega$ where $\omega$ is a minimal element in $P_{> x}$ and also vectors of the form $e_\omega$ where $\omega$ is a maximal element in $P_{< x}$. From the information that $a = 1$, it is not difficult to show that the equality that we would get as a result of these vectors being in the support of the mixed area measure give contradictions. Thus, some of the dimension inequalities corresponding to these vectors must be incorrect. This is enough to give the combinatorial characterization in Theorem~\ref{combinatorially-characterization-of-stanley-poset-simple}. During the present discussion, we did not show any of the computations for the sake of time and space. To see the computations in full detail and rigour, we refer the reader to the original source \cite{shenfeld2022extremals}. 

\section{Extremals of the Kahn-Saks Inequality} \label{kahn-saks-inequality}

This section covers joint work by Ramon van Handel, Xinmeng Zeng, and the author. We consider a slight generalization of the simple version of the Stanley poset inequality called the \textbf{Kahn-Saks Inequality}. We also cover aspects of the proof of our characterization of the equality cases. In the paper \cite{balancing-poset-extensions}, Jeff Kahn and Michael Saks prove that any finite poset contains a pair of elements $x$ and $y$ such that the proportion of linear extensions of $P$ in which $x$ lies below $y$ is between $\frac{3}{11}$ and $\frac{8}{11}$. This fact has consequences in theoretic computer science and sorting algorithms. The interested reader should consult the original source \cite{balancing-poset-extensions} for these applications as they are tangential to our thesis. One ingredient in their proof is a log-concavity inequality for a linear extension counting sequence. Let $P$ be a finite poset and fix elements $x, y \in P$ with $x \leq y$. Let $N_k$ be the number of linear extensions $f$ satisfying $f(y) - f(x) = k$ for all $1 \leq k \leq n-1$. The Kahn-Saks inequality written in Theorem~\ref{kahn-saks} states that the sequence $N_1, \ldots, N_{n-1}$ is log-concave. 

\begin{thm}[Kahn-Saks Inequality, Theorem 2.5 in \cite{balancing-poset-extensions}] \label{kahn-saks}
	For all $k \in \{2, \ldots, n-2\}$, we have $N_k^2 \geq N_{k-1} N_{k+1}$. 
\end{thm}

In the case where $x$ is an 0 element in $P$, Theorem~\ref{kahn-saks} reduces to Corollary~\ref{stanley-inequality-poset-simple} since $x$ will be forced to be the lowest element in any linear extension. The proof of Theorem~\ref{kahn-saks} is similar to the proof of Stanley's inquality. To prove log-concavity, we will associate the sequence $N_k$ to a sequence of mixed volumes of suitable polyopes. The polytopes that we use will be cross-sections of the order polytope $\mcO_P$. For all $\lambda \in \RR$, we define the polytope $K_\lambda$ as 
\[
	K_\lambda := \{t \in \mcO_P : t_y - t_x = \lambda\}.
\]
Each of these cross-sections can be written as a convex combination of $K_0$ and $K_1$. We defer the proof of $(1-\lambda)K_0 + \lambda K_1 = K_\lambda$ to Lemma~\ref{cross-section-mixed-lemma-computation-that-wasnt-in-kahn-saks}. From \cite{balancing-poset-extensions}, we know that $N_k =(n-1)! \mathsf{V}_{n-1}(K_0 [n-k], K_1[k-1])$. The mixed volume is well-defined since $K_0$ and $K_1$ lie in parallel hyperplanes. We give the full computation of this fact in Lemma~\ref{computation-of-mixed-volume-in-kahn-saks-case} in the appendix. The Kahn-Saks inequality then follows immediately from the Alexandrov-Fenchel inequality. In the next section, we begin our combinatorial characterization of the extremals of the Kahn-Saks inequality. Our main results are Theorem~\ref{kahn-saks-thm-1}, Theorem~\ref{kahn-saks-thm-2}, and Theorem~\ref{kahn-saks-thm-3}. 

\subsection{Simplifications and Special Regions}

Let $(P, \leq)$ be our poset with distinguished elements $x, y \in P$ satisfying $x \not \geq y$. We can always add a $0$ element and $1$ element to our poset. In any linear extension, these two elements will be forced to be placed in the beginning and end of the linear extension. Hence, adding these two elements to the poset will not change the values of the sequence $\{N_k\}$. If $x || y$, Lemma~\ref{add-new-relation} and Lemma~\ref{doesn't-change} in the appendix prove that for the poset generated by our original relations and the new relation $x \leq y$, our sequence $N_k$ will remain the same. Thus, we will assume that our poset $P$ has a $0$ element, has a $1$ elements, and satisfies $x \leq y$. In our poset, we will define the several special regions which each need to be handled separately in our analysis.
\begin{align*}
		\END_x & := \{\omega \in P : \omega < x\} = P_{< x} \\
		\END_y & := \{\omega \in P : \omega > y\} = P_{> y} \\
		\MID & := \{\omega \in P : x < \omega < y\} = P_{x < \cdot < y} \\
		\MID_x & := \{\omega \in P : x < \omega \text{ and } \omega || y\} = P_{\cdot > x, \cdot || y} \\
		\MID_y & := \{\omega \in P : \omega < y \text{ and } \omega || x\} = P_{\cdot < y, \cdot || x} \\
\end{align*}

These regions are disjoint, but they do not account for all elements in $P$. For a full partition of our poset $P$ we have 
\[
	P = \END_x \sqcup \END_y \sqcup \MID_x \sqcup \MID_y \sqcup \MID \sqcup \{x, y\} \sqcup P_{\not\sim x, y}
\]
where the subset $P_{\sim x, y}$ consists of the elements in $P$ which are not comparable to either $x$ or $y$. 

\subsection{Combinatorial characterization of the extremals}

Some of our main results about the extremals of the Kahn-Saks inequality were conjectured by Swee Hong Chan, Igor Pak, and Greta Panova in their paper \cite{chan2022extensions}. In their paper, they define the notions of the midway and dual-midway properties recorded in Definition~\ref{midway-defn}.
\begin{defn} \label{midway-defn}
	We say that $(x, y)$ satisfies the $k$-\textbf{midway property} if 
	\begin{itemize}
		\item $|P_{< z}| + |P_{>y}| > n-k$ for every $z \in P$ such that $x < z$ and $z \not > y$, 
		\item $|P_{z < \cdot < y}| > k$ for every $z < x$. 
	\end{itemize}
	Similarly, we say that $(x, y)$ satisfies the \textbf{dual $k$-midway property} if:
	\begin{itemize}
		\item $|P_{> z}| + |P_{<x}| > n-k$ for every $z \in P$ such that $z < y$ and $z \not < x$, and 
		\item $|P_{x < \cdot < z}| > k$ for every $z > y$. 
	\end{itemize}
	For any $z \in P$, we say that $z$ satisfies the $k$-midway property or satisfies the dual $k$-midway property if the relevant inequality for $z$ in the corresponding property is satisfied. For example, if $z < x$, we would say that $z$ satisfies the $k$-midway property if $|P_{z < \cdot < y}| > k$. 
\end{defn}
In Conjecture 8.7 in \cite{chan2022extensions}, Chan-Pak-Panova conjecture that when $N_k > 0$, the equality case $N_{k+1} = N_k = N_{k-1}$ occurs if and only if the midway property holds or the dual midway property holds. Using purely combinatorial methods, they prove this conjecture for width two posets where $x$ and $y$ are elements in the same chain in the two chain decomposition. By appealing to the extremals of the Alexandrov-Fenchel inequality, we are able to prove this conjecture in its full generality. We are also able to find all possible values of $a$ for which $N_{i+1} = aN_i = a^2 N_{i-1}$, and give combinatorial characterizations for the equality cases in all possible values of $a$. We now state our main results. 

\begin{thm} \label{kahn-saks-thm-1}
	If $N_k > 0$ and $N_k^2 = N_{k-1} N_{k+1}$, then we either have $N_{k+1} = N_k = N_{k-1}$ or $N_{k+1} = 2N_k = 4N_{k-1}$. 
\end{thm}

\begin{thm} \label{kahn-saks-thm-2}
	If $N_k > 0$, then the following are equivalent:
	\begin{enumerate}[label = (\alph*)]
		\item $N_{k-1} = N_k = N_{k+1}$.
		\item $(x, y)$ satisfies the $k$-midway property or the dual $k$-midway property. 
	\end{enumerate}
\end{thm}

\begin{thm} \label{kahn-saks-thm-3}
	If $N_k > 0$, then $N_{k+1} = 2N_k = 4N_{k-1}$ if and only if the following properties hold: 
	\begin{enumerate}[label = (\roman*)]
		\item $k$-midway holds for $\END_x$ and dual $k$-midway holds for $\END_y$.  
		\item Every $z \in P$ is comparable to $x$ and $y$. 
		\item $\MID$ is empty. 
		\item For every $z \in \MID_y$ and $z' \in \MID_x$ with $z < z'$, we have that $|P_{z < \cdot < y}| + |P_{x < \cdot < z'}| \geq k-1$. 
	\end{enumerate}
\end{thm}

\begin{remark}
	Theorem~\ref{kahn-saks-thm-2} resolves Conjecture 8.7 in \cite{chan2022extensions}. In their paper, Chan-Pak-Panova prove the important implication $(b) \implies (a)$. Since this result is central for the full characterization, we record it in Proposition~\ref{almost-theorem}.
\end{remark}

\begin{prop} [Theorem 8.9, Proposition 8.8 in \cite{chan2022extensions}]\label{almost-theorem}
	If $(x, y)$ satisfies the $k$-midway property or the dual $k$-midway property, then $N_{k-1} = N_k = N_{k+1}$. 
\end{prop}

The proof strategy of these results will be similar to characterization of the equality cases of the Stanley inequality in \cite{shenfeld2022extremals}. In this discussion, let us assume that $N_k^2 = N_{k-1} N_{k+1}$ and $N_k > 0$. Recall in the proof of Theorem~\ref{kahn-saks}, we showed the identity 
\[
	N_k = (n-1)! \mathsf{V}_{n-1} \left ( \underbrace{K_0, \ldots, K_0}_{n-k \text{ times}}, \underbrace{K_1, \ldots, K_1}_{k-1 \text{ times}} \right ).
\]
Then, we have equality in $N_k^2 = N_{k-1} N_{k+1}$ if and only if we have equality in the corresponding Alexadrov-Fenchel inequality:
\[
	\mathsf{V}_{n-1}(K_0[n-k], K_1[k-1])^2 \geq \mathsf{V}_{n-1} (K_0[n-k-1], K_1[k]) \cdot \mathsf{V}_{n-1} (K_0[n-k+1], K_1[k-2]).
\]
Using Lemma~\ref{positivity-of-mixed-volumes}, we have the characterization of positivity given in Proposition~\ref{prop-positivity-for-kahn-saks}. We omit the proof since we have prove a similar positivity result in the Stanley case. 

\begin{prop} \label{prop-positivity-for-kahn-saks}
	The following are equivalent. 
	\begin{enumerate}[label = (\alph*)]
		\item $N_k = 0$, 
		\item $\dim K_0 < n-k$ or $\dim K_1 < k-1$, 
		\item $|P_{< x}| + |P_{> y}| > n-k-1$ or $|P_{x < \cdot < y}| > k-1$. 
	\end{enumerate}
\end{prop}

Since $N_k^2 = N_{k-1} N_{k+1}$ and $N_k > 0$, we also have that $N_{k-1} > 0$ and $N_{k+1} > 0$. From Proposition~\ref{prop-positivity-for-kahn-saks} this gives us stronger inequalities on $\dim K_0$ and $\dim K_1$. From Proposition~\ref{affine-hulls}, we know that
\begin{align*}
	\dim K_0 & = n - |P_{x < \cdot < y}| - 1 \\
	\dim K_1 & = n - |P_{<x}| - |P_{> y}| - 2 \\
	\dim (K_0 + K_1) & = n-1. 
\end{align*}
Let $V = (e_y - e_x)^\perp \cong \RR^{n-1}$ be the orthogonal complement of the vector $e_y - e_x$. The polytopes $K_0$ and $K_1$ are contained in parallel copies of $V$. We can specialize Theorem~\ref{AF-equality-polytopes-theorem} to the Kahn-Saks case and prove Theorem~\ref{kahn-saks-main-workhorse-theorem}. This will be the main workhorse theorem to prove the combinatorial characterization for the equality case of the Kahn-Saks inequality.

\begin{thm} \label{kahn-saks-main-workhorse-theorem}
	Let $k \in \{2, \ldots, n-2\}$ such that $N_k > 0$ and $N_k^2 = N_{k-1} N_{k+1}$. Then 
	\[
		1 + |P_{x < \cdot < y}| < k < n - |P_{< x}| - |P_{> y}| - 1
	\]
	and there exists a positive scalar $a > 0$ and $v \in V$ such that $h_{K_0}(u) = h_{aK_1 + v}(u)$ for all $u \in \supp S_{B, K_0[n-k-1], K_1[k-2]}$.
\end{thm}

\begin{proof}
	Since $N_k > 0$, we know that $N_{k-1}, N_{k+1} > 0$. Proposition~\ref{prop-positivity-for-kahn-saks} applied to $N_{k-1}$ and $N_{k+1}$ implies that $\dim K_0 \geq n-k+1$ and $\dim K_1 \geq k$. Thus, $\mathcal{P} := (K_0[n-k-1], K_1[k-2])$ is supercritical. The theorem then follows from Theorem~\ref{AF-equality-polytopes-theorem}. 
\end{proof}

To slightly simplify our notation, we let $\mu := S_{B, K_0[n-k-1], K_1[k-2]}$ be the mixed area measure on the copy of $\mathbb{S}^{n-2}$ in $V$. The next lemma is a simple criterion for when a vector lies in the support of the mixed area measure. 

\begin{lem} \label{kahn-saks-conditions-for-being-in-support-area-measure-lemma}
	For any vector $u \in V$, we have that $u \in \supp \mu$ if and only if 
	\begin{align*}
		\dim F(K_0, u) & \geq n-k-1 \\
		\dim F(K_1, u) & \geq k-2 \\
		\dim F(K_0 + F_1, u) & \geq n-3.
	\end{align*}
\end{lem}

\begin{proof}
	This is an application of Lemma 2.3 in \cite{shenfeld2022extremals}.
\end{proof}

\subsection{Transition Vectors}

In this section, we assume that the hypothesis of Theorem~\ref{kahn-saks-main-workhorse-theorem} is true. That is, we have $N_k > 0$ and $N_k^2 = N_{k-1} N_{k+1}$. We can label our poset $(P, \leq)$ as $P = \{z_1, \ldots, z_n\}$ where $x = z_{n-1}$ and $y = z_n$. Let $v = (v_1, \ldots, v_n)$ where $v_i$ is the coordinate of $v$ with respect to the coordinate indexed by $z_i$. When it is more convenient to label the coordinate by the actual poset element itself, we use the notation $v_{z_i}$. In Definition~\ref{transition-vectors}, we define the notion of transition vectors. These are vectors in $\supp \mu$ which, from Corollary~\ref{transition-vector-implies-equality}, allows us to equate the two coordinates of $v$ of the corresponding pair of poset elements.

\begin{defn} \label{transition-vectors}
	For $z_i, z_j \in P \backslash \{x, y\}$, we define the vector 
	\[
		e_{ij} := \frac{1}{\sqrt{2}} (e_i - e_j) \in V.
	\]
	We call $e_{ij}$ a \textbf{transition vector} if $z_i \lessdot z_j$ and $e_{ij} \in \supp \mu$. We also use the notation $e_{z_i, z_j}$ in some cases. 
\end{defn}
Transition vectors will imply that the coordinates of the corresponding pair of poset elements are equal because the support function values of transition vectors will always be $0$. Indeed, for $z_i \lessdot z_j$, we can compute 
\begin{align*}
    h_{K_0}(e_{ij}) & = \frac{1}{\sqrt{2}} \sup_{t \in K_0} \langle t, e_i - e_j \rangle = \frac{1}{\sqrt{2}} \sup_{t \in K_0} t_i - t_j \leq 0 \\
    h_{K_1} (e_{ij}) & = \frac{1}{\sqrt{2}} \sup_{t \in K_1} \langle t, e_i - e_j \rangle  = \frac{1}{\sqrt{2}} \sup_{t \in K_1} t_i - t_j \leq 0
\end{align*}
where the last inequalities follow from the fact that $z_i\lessdot z_j$ and the parameter $t$ lies in the order polytope $\mathcal{O}_P$. Since we have $(0, \ldots, 0) \in K_0$ we have $h_{K_0}(e_{ij}) = 0$. To prove that $h_{K_1}(e_{ij}) = 0$, note that $z_i \lessdot z_j$ implies that it cannot be the case that $z_i < x$ and $z_j > y$. Since $\sum_{z_i \geq y} e_i , \sum_{z_i \not \leq x} e_i \in K_1$, we have $h_{K_1}(e_{ij}) = 0$. The equality $v_i = v_j$ then follows from Theorem~\ref{kahn-saks-main-workhorse-theorem}.

\begin{cor} \label{transition-vector-implies-equality}
    Suppose that the hypothesis of Theorem~\ref{kahn-saks-main-workhorse-theorem} is true and let $v \in V$ be the vector in the theorem. If $e_{ij}$ is a support vector, then $v_i = v_j$.  
\end{cor}

\begin{proof}
    From Theorem~\ref{kahn-saks-main-workhorse-theorem}, we have that $h_{K_0}(e_{ij}) = h_{aK_1 + v}(e_{ij}) = a h_{K_1} (e_{ij}) + \langle v, e_{ij} \rangle$. Since $h_{K_0}(e_{ij}) = h_{K_1}(e_{ij}) = 0$ and $\langle v, e_{ij} \rangle = v_i - v_j$, we have $v_i = v_j$. This proves the lemma. 
\end{proof}
In Lemma~\ref{some-transition-vectors}, we compile a list of transition vectors. 

\begin{lem} \label{some-transition-vectors}
    Suppose that the hypothesis in Theorem~\ref{kahn-saks-main-workhorse-theorem} is true. Then, we have the following transition vectors: 
    \begin{enumerate}[label = (\alph*)]
        \item Let $R = \END_x$ or $R = \END_y$. If $z_i, z_j \in R$ satisfy $z_i \lessdot z_j$, then $e_{ij}$ is a transition vector.

        \item Let $R = \MID_x$ or $R = \MID_y$. If $z_i, z_j \in R$ satisfy $z_i \lessdot z_j$, then $e_{ij}$ is a transition vector. 

        \item If $z_i \in \MID_x$ and $z_j \in \END_y$ such that $z_i \lessdot z_j$ and there does not exist $z_l \in \MID_x$ with $z_i \lessdot z_l$, then $e_{ij}$ is a transition vector. 

        \item If $z_j \in \MID_y$ and $z_i \in \END_x$ such that $z_i \lessdot z_j$ and there does not exist $z_l \in \MID_y$ with $z_l \lessdot z_j$, then $e_{ij}$ is a transition vector. 

        \item If $z_i, z_j \in \MID$ such that $z_i \lessdot z_j$, then $e_{ij}$ is a transition vector. 
    \end{enumerate}
\end{lem}

\begin{proof}
	 We only prove the result when $R = \END_x$ since the proofs of the other regions are relatively similar. A full proof of the result will be given in an upcoming paper with Ramon van Handel and Xinmeng Zeng. Before computing the dimensions of these faces, we first prove that there exists a linear extension $f : P \to [n]$ with $f(z_j) - f(z_i) = 1$ and $f(y) - f(x) = 1 + |P_{x < \cdot < y}|$. From Proposition~\ref{extension-narrow-exists}, there is a linear extension $\tilde{f} : P \to [n]$ with $\tilde{f}(y) - \tilde{f}(x) = |P_{x < \cdot < y}| + 1$. Both $z_i$ and $z_j$ will be located to the left of $x$ in the linear extension $\tilde{f}$. From Proposition~\ref{covering-modification}, we can modify $\tilde{f}$ by only changing the elements between $z_i$ and $z_j$ in the linear extension $\tilde{f}$ to a linear extension $f$ satisfying $f(z_j) - f(z_i) = 1$. Clearly, we have the following set inclusion
    \[
        F_{K_0}(e_{ij}) \supseteq \Delta_f \cap \{t_i = t_j, t_x = t_y\}.
    \]
    Taking the affine span of both sets, we have 
    \begin{align*}
        \aff F_{K_0}(e_{ij}) & \supseteq \aff \Delta_f \cap \{t_i = t_j, t_x = t_y\} \\
        & = \RR \left [ \sum_{\omega \in P_{x \leq \cdot \leq y}} e_\omega \right ] \oplus \RR [ e_i + e_j] \oplus \bigoplus_{\omega \in P \backslash (P_{x \leq \cdot \leq y} \cup \{z_i, z_j\})} \RR[e_\omega]. 
    \end{align*}
    This gives us the following bounds for the dimensions of $F_{K_0}(e_{ij})$ and $F_{K_1}(e_{ij})$:
    \begin{align*}
        \dim F_{K_0}(e_{ij}) & \geq n - |P_{x < \cdot < y}| - 2 = \dim K_0 - 1\\
        \dim F_{K_1}(e_{ij}) & = \dim K_1. 
    \end{align*}
    From Lemma~\ref{linear-algebraic-lemma}, we have $\dim F_{K_0 + K_1}(e_{ij}) \geq \dim (K_1 + K_0) - 1 = n-2$. From the bounds in Theorem~\ref{kahn-saks-main-workhorse-theorem}, we have 
    \begin{align*}
        \dim F_{K_0}(e_{ij}) & = n - |P_{x < \cdot < y}| - 2 \geq n - k \\
        \dim F_{K_1}(e_{ij}) & = n - |P_{< x}| - |P_{> y}| - 2 \geq k. 
    \end{align*}
    From Lemma~\ref{kahn-saks-conditions-for-being-in-support-area-measure-lemma}, we know that $e_{ij} \in \supp \mu$. This proves that $e_{ij}$ is a transition vector.
\end{proof}

For any result about support vectors, the calculations needed to determine whether or not certain vectors are support vectors look very similar to the calculation performed in our proof of Lemma~\ref{some-transition-vectors}. The strategy is to find ``extremal linear extensions'' where the intersection between the corresponding simplex and the faces of our polytopes has large dimension. To avoid reptitiveness and bogging the thesis with technical details, we omit a large portion of the support vector calculations. 

\begin{lem} \label{value-of-minima-maxima}
    Suppose that the hypothesis of Theorem~\ref{kahn-saks-main-workhorse-theorem} holds. Let $z_i \in P \backslash \{x, y\}$ be an element of the poset. If $z_i$ is minimal, then $v_i = 0$. If $z_i$ is maximal, then $v_i = 1-a$. 
\end{lem}

\begin{proof}
	If $z_i \in P$ is a minimal element, then we can show that $-e_i$ is in the support of the mixed area measure. If $z_i \in P$ is a maximal element, then we can show that $e_i$ is in the support of the mixed area measure. These facts will imply via Theorem~\ref{kahn-saks-main-workhorse-theorem} that $v_i = 0$ when $z_i$ is minimal and $v_i = 1-a$ when $z_i$ is maximal. 
\end{proof}

Using Lemma~\ref{some-transition-vectors} and Lemma~\ref{value-of-minima-maxima}, we can compute some of the coordinates of the vector $v$ depending on the region containing the element poset corresponding to the coordinate.  

\begin{cor} \label{cor-coordinates-in-the-regions}
    Let $\omega \in P \backslash \{x, y\}$ be some element in our poset. 
    \begin{enumerate}[label = (\alph*)]
        \item If $\omega \in \END_x \cup \MID_y$, then $v_\omega = 0$;
        \item If $\omega \in \END_y \cup \MID_x$, then $v_\omega = 1-a$;
        \item If $\omega_1, \omega_2 \in \MID$ are comparable, then $v_{\omega_1} = v_{\omega_2}$.
    \end{enumerate}
\end{cor}

\begin{proof}
    Suppose that $\omega \in \END_x$. If $\omega$ is a minimal element, then we already know that $v_\omega = 0$ from Lemma~\ref{value-of-minima-maxima}. Suppose that $\omega$ is not a minimal element. Then there is a sequence of poset elements $\omega_1, \omega_2, \ldots, \omega_l$ for $l \geq 1$ such that $\omega_1 \lessdot \omega$, $\omega_{i+1} \lessdot \omega_i$ for all $i$, and $\omega_l$ is a minimal element. From Lemma~\ref{some-transition-vectors}(a), we know that $e_{\omega_1, \omega}$ and $e_{\omega_{i+1}, \omega_i}$ are transition vectors for all $i$. From Corollary~\ref{transition-vector-implies-equality} and Corollary~\ref{value-of-minima-maxima} applied to $\omega_l$, we have $v_\omega = v_{\omega_1} = \ldots = v_{\omega_l} = 0$. A similar proof will work when $\omega \in \END_y$, but instead we built a maximal chain to a maximal element. \\

    Now, suppose that $\omega \in \MID_x$. If $\omega$ is maximal, then we automatically know that $v_\omega = 1-a$. If $\omega$ is not maximal, we can build a chain $\omega \lessdot \omega_1 \lessdot \omega_2 \lessdot \ldots \lessdot \omega_l$ where $\omega_l$ is a maximal element by using a special procedure to guarentee that all adjacent elements correspond to transition vectors. To illustrate this procedure, suppose that we have picked $\omega \lessdot \omega_1 \lessdot \ldots \lessdot \omega_i$ up to some $i$. Then $\omega_i \in \END_y$ or $\omega_i \in \MID_x$. If $\omega_i \in \END_y$, then we can continue to a maximal element arbitrarily. Otherwise, if $\omega_i \in \MID_x$, then we pick $\omega_{i+1}$ to be an element of $\MID_x$ which covers $\omega_i$. If none exists, we then pick $\omega_{i+1}$ to be an element of $\END_y$ which covers $\omega_i$. If this doesn't exist again, then we know that $\omega_i$ is a maximal element. By our construction and Lemma~\ref{some-transition-vectors}, we know that the adjacent poset elements in the chain form transition vectors. This implies that $v_\omega = v_{\omega_1} = \ldots = v_{\omega_l} = 1-a$. A similar proof works for $\MID_y$ but instead we create the same type of chain to a minimal element. \\

    For the result on $\MID$, suppose that $\omega_1, \omega_2 \in \MID$ are comparable. Since any two comparable elements in $\MID$ can be connected by a chain completely contained in $\MID$, it suffices to prove the result when $\omega_1 \lessdot \omega_2$. In this case, from Lemma~\ref{some-transition-vectors}(e), we know that $v_{\omega_1, \omega_2}$ is a transition vector. Hence $v_{\omega_1} = v_{\omega_2}$. This suffices for the proof.  
\end{proof}

To conclude the section on transition vectors, we give conditions for when a vector of the form $e_{ij}$ is a transition vector from $\MID$ to $\MID_x$ or $\MID_y$. 

\begin{prop} \label{prop-actually-important-transition-vector}
	Let $z_i \in P \backslash \{x, y\}$ be a poset element satisfying $x < z_i < y$. 
	\begin{enumerate}[label = (\alph*)]
		\item If $z_j \in P$ satisfies $z_j \in \MID_y$ and $z_j \lessdot z_i$, then $e_{ji}$ is a transition vector if $|\MID| + |S| + 1 \leq k$ where we define the set $S$ to be 
		\[
			S = P_{z_i < \ldots < y} \backslash \MID = \{z : z_i < z < y, z \notin \MID\}.
		\]
		\item If $z_j \in P$ satisfies $z_j \in \MID_x$ and $z_j \gtrdot z_i$, then $e_{ij}$ is a transition vector if $|\MID| + |S| + 1 \leq k$ where we define the set $S$ to be
		\[
			S = P_{x < \cdot < z_i} \backslash \MID = \{z : x < z < z_i, z \notin \MID \}.
		\]
	\end{enumerate}
\end{prop}

\begin{proof}
	We only prove (a) as the proof of (b) is similar. Note that 
	\begin{align*}
		\aff F_K(e_{ji}) & = \bigoplus_{\omega \notin \MID \sqcup S \sqcup \{x, y, z_i\}} \RR[e_\omega] \oplus \RR \left [ \sum_{\omega \in \MID \sqcup S \sqcup \{x, y, z_i\}} e_\omega \right ] \\
		\aff F_L(e_{ji}) & = \bigoplus_{\omega \notin \{z_i, z_j\} \sqcup P_{\leq x} \sqcup P_{\geq y}} \RR[e_\omega] \oplus \RR[e_i + e_j] + \sum_{\omega \in P_{\geq y}} e_\omega \\
		\aff F_{K+L}(e_{ji}) & = \bigoplus_{\omega \notin \{z_i, z_j, x, y\}} \RR[e_\omega] \oplus \RR \left [ \sum_{\omega \in \MID \sqcup S \sqcup \{x, y, z_i\}} e_\omega \right ] \oplus \RR[e_i + e_j] \oplus \sum_{\omega \in P_{\geq y}} e_\omega.
	\end{align*}
	The only obstruction to the vector $e_{ji}$ being in the support is the dimension inequality corresponding to $\dim \aff F_K(e_{ji})$. Specifically, we need the inequality 
	\[
		n - |\MID| - |S| - 2 = \dim F_K(e_{ij}) \geq n - k - 1 \iff |\MID| + |S| + 1 \leq k.
	\]
	This suffices for the proof. 
\end{proof}

\subsection{More Vectors}

In this section, we introduce a new type of vector not found in the Stanley case that will allow us to extract the midway and dual-midway inequalities from our poset.

\begin{defn} \label{special-vectors}
    For $j \in \{1, \ldots, n-2\}$, define 
    \begin{align*}
        u_j^+ := \sqrt{\frac{2}{3}} \left ( e_j - \frac{1}{2}(e_x + e_y) \right ) \\
        u_j^- := \sqrt{\frac{2}{3}} \left ( \frac{1}{2} (e_x + e_y) - e_j \right )
    \end{align*}
    where $e_i := e_{z_i}$ is the basis vector corresponding to the poset element $z_i \in P$.
\end{defn}

We provide the following table which shows which vector we use in our analysis based on the location of each element in the poset. We also display the corresponding support values. 

\begin{center}
    \begin{tabular}{||c c c c c||} 
         \hline
         Region & Relation & Vector & $h_{K_0}(\cdot)$ & $h_{K_1}(\cdot)$ \\ [0.5ex] 
         \hline\hline
         $z_j \in \MID$, & $z_j \lessdot y$ & $u_j^+$ & $0$ & $\sqrt{\frac{1}{6}}$ \\
         \hline
         $z_j \in \MID$, & $z_j \gtrdot x$ & $u_j^-$ & $0$ & $\sqrt{\frac{1}{6}}$ \\
         \hline
         $z_j \in \MID_x$, & $z_j \gtrdot x$ & $u_j^-$ & $0$ & $\sqrt{\frac{1}{6}}$\\
         \hline 
         $z_j \in \MID_y$, & $z_j \lessdot y$ & $u_j^+$ & $0$ & $\sqrt{\frac{1}{6}}$ \\
         \hline
         $z_j \in \END_x$, & $z_j \lessdot x$ & $u_j^+$ & $0$ & $-\sqrt{\frac{1}{6}}$ \\
         \hline 
         $z_j \in \END_y$, & $z_j \gtrdot y$ & $u_j^-$ & $0$ & $-\sqrt{\frac{1}{6}}$ \\
         \hline 
    \end{tabular}
\end{center}

After finding the conditions for the vectors in the table to be in the support of the mixed area measure, we can prove Proposition~\ref{prop-conditions-for-special-vectors}. As discussed before, we will omit the proofs.

\begin{prop} \label{prop-conditions-for-special-vectors}
    Let $z_i \in P \backslash \{x, y\}$ be an element. Then 
    \begin{enumerate}[label = (\alph*)]
        \item If $z_i \in \MID$, $z_i \lessdot y$, then $u_i^+ \in \supp \mu$ if and only if $|P_{> z_i}| \leq n-k-|P_{< x}|$. 
        \item If $z_i \in \MID$, $z_i \gtrdot x$, then $u^-_i \in \supp \mu$ if and only if $|P_{<z_i}| \leq n-k - |P_{>y}|$.
        \item If $z_i \in \MID_x^{\text{min}}$, then $u^-_i \in \supp \mu$ if and only if $|P_{<z_i}| \leq n-k - |P_{>y}|$.
        \item If $z_i \in \END_x$, $z_i \lessdot x$, then $u^+_i \in \supp \mu$ if and only if $|P_{z_i < \cdot < y}| \leq k$. 
        \item If $z_i \in \MID_y^{\text{max}}$, then $u_i^+ \in \supp \mu$ if and only if $|P_{>z_i}| \leq n-k-|P_{< x}|$.
        \item If $z_i \in \END_y$, $z_i \gtrdot y$, then $u^-_i \in \supp \mu$ if and only if $|P_{x < \cdot < z_i}| \leq k$. 
    \end{enumerate}
\end{prop}

Using Theorem~\ref{kahn-saks-main-workhorse-theorem} and the table of support value computations, we can translate Proposition~\ref{prop-conditions-for-special-vectors} into information about $v$ and $a$. The result of this application of Theorem~\ref{kahn-saks-main-workhorse-theorem} is given in Proposition~\ref{prop-conditions-for-special-vectors-on-a-v}.

\begin{prop} \label{prop-conditions-for-special-vectors-on-a-v}
	Let $z_i \in P \backslash \{x, y\}$ be an element of our poset. Let $v_{xy} := \frac{1}{2} (v_x + v_y)$. 
	\begin{enumerate}[label = (\alph*)]
		\item If $z_i \in \MID$ and $z_i \lessdot y$, then $v_i \neq v_{xy} - a/2 \implies |P_{> z_i}| > n-k - |P_{< x}|$. 
		\item If $z_i \in \MID$ and $z_i \gtrdot x$, then $v_i \neq v_{xy} + a/2 \implies |P_{< z_i}| > n-k-|P_{>y}|$. 
		\item If $z_i \in \MID_x^{\text{min}}$, then $v_i \neq v_{xy} + a/2 \implies |P_{<z_i}| > n-k-|P_{>y}|$.
		\item If $z_i \in \MID_y^{\text{max}}$, then $v_i \neq v_{xy} - a/2 \implies |P_{> z_i}| > n-k - |P_{< x}|$. 
		\item If $z_i \in \END_x$ and $z_i \lessdot x$, then $v_i \neq v_{xy} + a/2 \implies |P_{z_i < \ldots < y}| > k$. 
		\item If $z_i \in \END_y$ and $z_i \gtrdot y$, then $v_i \neq v_{xy} - a/2 \implies |P_{x < \cdot < z_i}| > k$. 
	\end{enumerate}
\end{prop}

In Proposition~\ref{prop-conditions-for-special-vectors-on-a-v}, the result was only stated for poset elements satisfying certain minimal/maximal relations. But, the result is still true for all poset elements in the corresponding region. This immediately gives Corollary~\ref{cor-key-result}. 

\begin{cor} \label{cor-key-result}
	Let $z_i \in P \backslash \{x, y\}$ be an element of our poset. Let $v_{xy} := \frac{1}{2} (v_x + v_y)$. Then, the following statements are true. 
	\begin{enumerate}[label = (\alph*)]
		\item If $z_i \in \MID$, then $v_i \neq v_{xy} - a/2 \implies |P_{> z_i}| > n-k - |P_{< x}|$. 
		\item If $z_i \in \MID$, then $v_i \neq v_{xy} + a/2 \implies |P_{< z_i}| > n-k-|P_{>y}|$. 
		\item If $z_i \in \MID_x$, then $v_i \neq v_{xy} + a/2 \implies |P_{<z_i}| > n-k-|P_{>y}|$.
		\item If $z_i \in \MID_y$, then $v_i \neq v_{xy} - a/2 \implies |P_{> z_i}| > n-k - |P_{< x}|$. 
		\item If $z_i \in \END_x$ and $z_i \lessdot x$, then $v_i \neq v_{xy} + a/2 \implies |P_{z_i < \ldots < y}| > k$. 
		\item If $z_i \in \END_y$ and $z_i \gtrdot y$, then $v_i \neq v_{xy} - a/2 \implies |P_{x < \cdot < z_i}| > k$.  
	\end{enumerate}
\end{cor}

\begin{proof}
	We only prove (a) since the rest are similar. Let $z_i \in \MID$. Then there is some $z_0 \in \MID$ so that $z_i \leq z_0 \lessdot y$. Then, since $v_{z_i} = v_{z_0}$, we have that 
	\[
		|P_{>z_i}| \geq |P_{>z_0}| > n- k - |P_{< x}|.
	\]
	This suffices for the proof. 
\end{proof}

From Corollary~\ref{cor-coordinates-in-the-regions}, we already know some of the values of $v_i$ in Proposition~\ref{prop-conditions-for-special-vectors-on-a-v}. This immediately implies Corollary~\ref{cor-more-explicit-conditions}. 

\begin{cor} \label{cor-more-explicit-conditions}
	Let $z_i \in P \backslash \{x, y\}$ be an element of our poset. Let $v_{xy} := \frac{1}{2} (v_x + v_y)$. Then, the following statements are true. 
	\begin{enumerate}[label = (\alph*)]
		\item Suppose that $z_i \in \MID_x$. Then $a \neq \frac{2}{3} (1 - v_{xy}) \implies |P_{<z_i}| > n-k-|P_{>y}|$. 
		\item Suppose that $z_i \in \MID_y$. Then $a \neq 2v_{xy} \implies |P_{>z_i}| > n-k-|P_{< x}|$. 
		\item Suppose that $z_i \in \END_x$. Then $a \neq -2v_{xy} \implies |P_{z_i < \cdot < y}| > k$. 
		\item Suppose that $z_i \in \END_y$. Then $a \neq 2 (1 - v_{xy}) \implies |P_{x < \cdot < z_i}| > k$. 
	\end{enumerate}
\end{cor}
Based on Corollary~\ref{cor-more-explicit-conditions}, we define four statements which are stand-ins for midway and dual-midway in the regions $\END_x, \END_y, \MID_x$, and $\MID_y$. Consider the four statements 
\begin{align*}
	M_x & := \left \{ a \neq \frac{2}{3} (1 - v_{xy}) \right \}, \\
	M_y & := \left \{ a \neq 2v_{xy} \right \}, \\
	E_x & := \{ a \neq -2v_{xy} \}, \\
	E_y & := \{ a \neq 2 ( 1 - v_{xy}) \}.
\end{align*}
From Corollary~\ref{cor-more-explicit-conditions}, we know that $M_x$ implies that midway holds for $\MID_x$, $M_y$ implies that dual-midway holds for $\MID_y$, $E_x$ implies that midway holds for $\END_x$, and $E_y$ implies that dual-midway holds for $\END_y$.

\begin{lem} \label{lem-when-a-not-half-then-either-mid-or-dual-mid}
	If $a \neq 1/2$, then either $M_x$ and $E_x$ hold simultaneously or $M_y$ and $E_y$ hold simultaneously. As a consequence, dual-midway holds for $\END_y$ and $\MID_y$ or midway holds for $\END_x$ and $\MID_x$.
\end{lem}

\begin{proof}
	Suppose that $M_x$ and $M_y$ do not hold simultaneously. Then, we have 
	\[
		2v_{xy} = \frac{2}{3} (1 - v_{xy}) \implies v_{xy} = \frac{1}{4}.
	\]
	This would imply that $a = 1/2$, which contradicts the hypothesis. Thus $M_x$ and $M_y$ cannot be false simultaneously. Suppose that $M_x$ and $E_y$ do not hold simultaneously. Then, 
	\[
		2(1 -v_{xy}) = \frac{2}{3} (1 - v_{xy}) \implies v_{xy} = 1.
	\]
	But then $a = 0$ which contradicts $N_i > 0$. Thus $M_x$ and $E_y$ cannot be false simultaneously. Suppose that $E_x$ and $M_y$ do not hold simultaneously. Then we run into the same contradiction $a = 0$. Now, suppose that $E_x$ and $E_y$ do not hold simultaneously. Then we get an obvious contradiction. This suffices for the proof.  
\end{proof}

\subsection{Understanding MID}

We now introduce some results that will give us a better understand of the region $\MID$.

\begin{lem} \label{lem-simultaneously-violate-mid-and-dual-mid}
	Suppose that the hypothesis in Theorem~\ref{kahn-saks-main-workhorse-theorem} holds. Let $z, w \in P$ be two comparable elements in $\MID$. Then, the following are true. 
	\begin{enumerate}[label = (\alph*)]
		\item If $z$ violates midway, then $w$ cannot violate dual-midway. In other words, the inequality $|P_{> y}| + |P_{< z}| \leq n-k$ implies $|P_{<x}| + |P_{>w}| > n-k$. 
		\item If $z$ violates dual-midway, then $w$ cannot violate midway. In other words, the inequality $|P_{< x}| + |P_{> z}| \leq n-k$ implies $|P_{> y}| + |P_{< w}| > n-k$. 
	\end{enumerate}
\end{lem}

\begin{proof}
	We only prove (a) as (b) is symmetric. Since $|P_{> y}| + |P_{< z}| \leq n-k$, we have that $v_z = v_{xy} + a/2$ from Proposition~\ref{prop-conditions-for-special-vectors-on-a-v}. Suppose for the sake of contradiction that $|P_{< x}| + |P_{> w}| \leq n-k$ as well. Then Proposition~\ref{prop-conditions-for-special-vectors-on-a-v} would imply that $v_w = v_{xy} - a/2$. But Corollary~\ref{cor-coordinates-in-the-regions} gives $v_z = v_w$. This would give $a = 0$ which contradicts $N_k > 0$. This suffices for the proof. 
\end{proof}

\begin{prop} \label{prop-connective-alan-kahn-saks}
	Let $z \in P$ be an element satisfying $x < z < y$. Then, the following are true. 
	\begin{enumerate}[label = (\alph*)]
		\item If $|P_{< x}| + |P_{> z}| \leq n-k$, then there is a chain from $z$ to a minimal element of $P_{x < \cdot < y}$ where the minimal element covers an element outside of $P_{x < \cdot < y}$. 
		\item If $|P_{> y}| + |P_{< z}| \leq n-k$, then there is a chain from $z$ to a maximal element of $P_{x < \cdot < y}$ which is covered by an element outside of $P_{x < \cdot < y}$. 
		\item We cannot have both $|P_{< x}| + |P_{> z}| \leq n-k$ and $|P_{> y}| + |P_{< z}| \leq n-k$.  
	\end{enumerate}
\end{prop}

\begin{proof}
 	Part (c) follows from Lemma~\ref{lem-simultaneously-violate-mid-and-dual-mid}. We only prove (a) since (b) is symmetric. There is some $z' \in \MID$ so that $x \lessdot z' \leq z < y$. Suppose that we cannot extend $z'$ to a minimal element outside of $P_{x< \cdot < y}$. Then, we would have $|P_{< z'}| = 1 + |P_{< x}|$. From (b) of Lemma~\ref{lem-simultaneously-violate-mid-and-dual-mid}, we have the inequality 
 	\[
 		n-k < |P_{> y}| + |P_{< z'}| = |P_{> y}| + |P_{< x}| + 1 < n-k
 	\]
 	where the last inequality follows from Theorem~\ref{kahn-saks-main-workhorse-theorem}. This is a contradiction and suffices for the proof of the proposition. 
\end{proof}

\begin{lem} \label{lem-midway-then-mx-false-and-other-version}
	Let $z_i \in \MID$ be an element in the poset. If $z_i$ violates midway then $M_x$ is false. If $z_i$ violates dual midway, then $M_y$ is false. In other words, we have the following implications:
	\begin{align*}
		|P_{> y}| + |P_{< z_i}| \leq n-k & \implies a = \frac{2}{3} (1 - v_{xy}) \\
		|P_{< x}| + |P_{> z_i}| \leq n-k & \implies a = 2v_{xy}. 
	\end{align*}
\end{lem}

\begin{proof}
	Suppose that $z_j \in \MID$ violates midway. That is, we have that $|P_{> y}| + |P_{<z_j}| \leq n-k$. Then, from Proposition~\ref{prop-connective-alan-kahn-saks}, we know that there is an element $z'$ satisfying $x < z_j \leq z' \lessdot y$ so that there exists $z'' \in \MID_x$ where $z' \lessdot z''$. From Proposition~\ref{lem-simultaneously-violate-mid-and-dual-mid}, we have the inequality $|P_{< x}| + |P_{>z'}| \geq n-k+1$. Let 
	\[
		S = P_{x < \cdot < z''} \backslash \MID = \{z : x < z < z'', z \notin \MID\}.
	\]
	Then, we have that 
	\[
		|S| + |\MID| + |P_{< x}| + |P_{>z'}| \leq n
	\]
	since these sets do not over lap. We can conclude that 
	\[
		|S| + |\MID| + 1 \leq n - (|P_{< x}| - |P_{> z'}|) + 1 \leq k.
	\]
	From Proposition~\ref{prop-actually-important-transition-vector}, this implies that $e_{z', z''}$ is a transition vector. Hence $v_{z_j} = 1-a$. Now consider $z_0 \in P$ satisfying $x \lessdot z_0 \leq z_j < y$. We have that 
	\[
		|P_{< z_0}| + |P_{> y}| \leq |P_{< z_j}| + |P_{> y}| \leq n-k.
	\]
	This implies that $v_{z_0} = v_{xy} + a/2$ from Proposition~\ref{prop-conditions-for-special-vectors-on-a-v}. But then this implies that $1 - a = v_{xy} + a/2$ or $a = \frac{2}{3} (1 - v_{xy})$. This means that $M_x$ is false. Similarly, if $z_j$ violates dual-midway, then we would know that $a = 2v_{xy}$ or $M_y$ is false.
\end{proof}

\begin{lem}
	If $a \neq 1/2$, then the hypothesis in Theorem~\ref{kahn-saks-main-workhorse-theorem} implies that $(x, y)$ satisfies $k$-midway or dual $k$-midway. 
\end{lem}

\begin{proof}
	From Lemma~\ref{lem-simultaneously-violate-mid-and-dual-mid}, we know that every element either satisfies dual-midway or satisfies midway (or possibly both). Suppose that there is an element which violates midway and another element which violates dual-midway. Then we must have $M_x$ and $M_y$ are both false, which cannot happen since $a \neq 1/2$. Thus every element in $\MID$ either all satisfies midway or all satisfies dual-midway or all satisfies both. If it all satisfies both, then we are done. If there is at least one element in $\MID$ which violates midway, then all elements in MID satisfies dual-midway. Moreover, it also means that $M_x$ is false from Lemma~\ref{lem-midway-then-mx-false-and-other-version} and $M_y, E_y$ are true. This means that $\MID$, and $\MID_y$ and $\END_y$ satisfies dual-midway, i.e. $P$ satisfies dual $k$-midway. By similar reasoning, if there is at least one element which violates dual-midway, then $P$ satisfies $k$-midway. This suffices for the proof. 
\end{proof}

We can now prove Theorem~\ref{kahn-saks-thm-1} and Theorem~\ref{kahn-saks-thm-2}. 

\begin{proof}[Proof of Theorem~\ref{kahn-saks-thm-2}]
	If $N_{k-1} = N_k = N_{k+1}$, that means that $a = 1$. From the previous theorem, this implies that $P$ satisfies $k$-midway or dual $k$-midway. The reverse direction was proved in \cite{chan2022extensions}. 
\end{proof}

\begin{proof}[Proof of Theorem~\ref{kahn-saks-thm-1}]
	If $a \neq 1/2$, then we have that $P$ satisfies $k$-midway or dual $k$-midway. But this implies that $N_k = N_{k-1} = N_{k+1}$, or $a = 1$. Hence the only two choices of $a$ are $a = 1, 1/2$. This corresponds to $N_{k+1} = N_k = N_{k-1}$ and $N_{k+1} = 2N_k = 4 N_{k-1}$. 
\end{proof}

The only case that is remaining is the case $N_{k+1} = 2N_k = 4 N_{k-1}$. To show that this case actually occurs, we give an example in Figure~\ref{fig:example-of-a-is-two}. In this poset, we can compute explicitly compute $N_1 = 1, N_2 = 2$ and $N_3 = 4$. This satisfies $N_3 = 2N_2 = 4N_1$. 

\begin{figure}[ht] 
	\begin{center}
		\includegraphics[scale = 0.8]{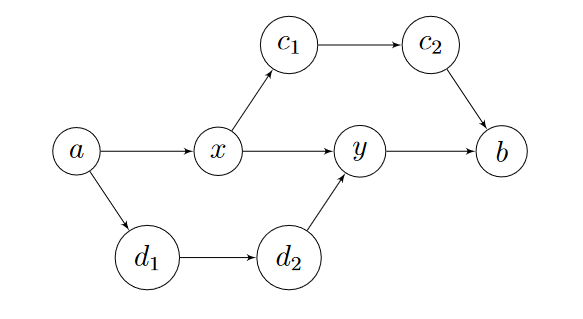}
		\caption{Example with $N_{i+1} = 2N_i = 4 N_{i-1}$}
		\label{fig:example-of-a-is-two}
	\end{center}
\end{figure}

\begin{proof}[Proof of Theorem~\ref{kahn-saks-thm-3}.]
	We first prove that if $N_{k+1} = 2N_k = 4 N_{k-1}$ and $N_k > 0$, then conditions (i)-(iv) are true in Theorem~\ref{kahn-saks-thm-3}. Since $a = 1/2$, we either have that both $M_x$ and $M_y$ are true or both $E_x$ and $E_y$ are true. Suppose that both $M_x$ and $M_y$ are true. Then every element in $\MID$ does not violate midway or dual midway. Moreover, at least one of $E_x$ and $E_y$ must be true as well. This proves that $P$ satisfies midway and dual midway. But this would imply that $a = 1$, which contradicts $a = 1/2$. Hence, it must be the case that $E_x$ and $E_y$ are true. In other words, $\END_x$ and $\END_y$ both satisfy the midway and dual-midway property, respectively.  This is exactly the property that (i) describes. The fact that $\MID$ is empty follows from the fact that $\END_x$ satisfies midway and $\END_y$ satisfies dual midway. Consider an arbitrary $z \in \MID$. Then, we can pick $z_1, z_2 \in P$ so that $x \lessdot z_1 \leq z \leq z_2 \lessdot y$. Since $v_{z_1} = v_{z_2} = v_z$, if both of the inequalities
	\begin{align*}
		|P_{> y}| + |P_{< z_1}| & \leq n-k \\
		|P_{< x}| + |P_{> z_2}| & \leq n-k 
	\end{align*}
	are true, then Corollary~\ref{cor-key-result} implies that $a = 0$, which cannot happen. This means that one of the inequalities must be wrong. Without loss of generality, suppose that $|P_{> y}| + |P_{<z_1}| > n-k$. Note that $\MID_x \sqcup \MID \sqcup \{y\} \sqcup P_{< z_1} \sqcup P_{> y} \subseteq P$ is a disjoint union of sets. Thus, we have that 
	\[
		|\MID_x| + |\MID| + |P_{< z_1}| + |P_{> y}| + 1 \leq n \implies |\MID_x| + |\MID| \leq k-2.
	\]
	Let $z'' \gtrdot y$. This element exists because $P$ has a $1$ element. Then, we have that 
	\[
		|P_{x < \cdot < z''}| \leq |\MID_x| + |\MID| \leq k-2.  
	\]
	But this is a contradiction since $|P_{x < \cdot < z''}| > k$ from the fact that $\END_y$ satisfies dual-midway. This proves (iii). To prove (ii), suppose that $z \in P$ is incomparable to $x$ and $y$. Then, we can build a chain from $z$ to a maximal element by picking only elements incomparable to $x$ and $y$ until we are forced to pick elements in $\MID_x$ or $\END_y$. We can do the same to reach a minimal element. It is not hard to prove that all of these vectors are transition vectors. Hence, we get that $0 = 1 - a$. But this cannot be true since $a = 1/2$. This proves (ii). To prove (iv), suppose that $z \in \MID_y$ and $z' \in \MID_x$ so that $z < z'$. Then, we can find $z_1 \in \MID_y$ and $z_2 \in \MID_x$ such that $z \leq z_1 \lessdot z_2 \leq z'$ since $\MID = \emptyset$. We analyze the conditions needed for $e_{z_1, z_2}$ to be a transition vector. Note that 
	\begin{align*}
		\aff F_{K_0}(e_{z_1, z_2}) & = \bigoplus_{\omega \notin P_{z_1 \leq \cdot \leq y} \cup P_{x \leq \cdot \leq z_2}} \RR[e_\omega] \oplus \RR \left [ \sum_{\omega \in P_{z_1 \leq \cdot \leq y} \cup P_{x \leq \cdot \leq z_2}} e_\omega \right ] \\
		\aff F_{K_1} (e_{z_1, z_2}) & = \sum_{\omega \geq y} e_\omega + \bigoplus_{\omega \notin P_{\leq x} \cup P_{\geq y} \cup \{z_1, z_2\}} \RR[e_\omega]\\
		\aff F_{K_0 + K_1}(e_{z_1, z_2}) & = \sum_{\omega \geq y} e_\omega + \bigoplus_{\omega \notin \{z_1, z_2, x, y\}} \RR[\omega] \oplus \RR \left [ \sum_{\omega \in P_{z_1 \leq \cdot \leq y} \cup P_{x \leq \cdot \leq z_2}} e_\omega \right ].
	\end{align*} 
	The only obstruction to $e_{z_1, z_2} \in \supp \mu$ is the dimension of $F(K_0, e_{z_1, z_2})$. We have that $e_{z_1, z_2} \in \supp \mu$ if and only if 
	\[
		n - |P_{z_1 < \cdot < y}| - |P_{x < \cdot < z_2}| - 3 \geq n - k - 1
	\]
	or $|P_{z_1 < \cdot < y}| + |P_{x < \cdot < z_2}| \leq k-2$. But we know that $e_{z_1, z_2}$ cannot be in the support because that would imply that $0 = 1 - a$. Hence, we have that 
	\[
		|P_{z < \cdot < y}| + |P_{x < \cdot < z'}| \geq |P_{z_1 < \cdot < y}| + |P_{x < \cdot < z_2}| \geq k-1.
	\]
	To finish the proof, we must prove the reverse implication. Specifically, given the conditions that (0) $N_k > 0$, (1) $\END_x$ and $\END_y$ satisfy midway and dual-midway, (2) $\MID$ is empty, (3) every element is comparable to either $x$ or $y$, and (4) for every $z \in \MID_y$ and $w \in \MID_x$ satifsying $z < w$, we have $|P_{z < \cdot < y}| + |P_{x < \cdot < w}| \geq k-1$, we want to prove that $N_{k+1} = 2N_k = 4N_{k-1}$. To prove this, we prove Claim~\ref{auxiliary-claim}.
	\begin{claim} \label{auxiliary-claim}
		Let $L_x$ be the number of linear extensions of the poset $\MID_y \sqcup \END_x$ and let $L_y$ be the number of linear extensions of the poset $\MID_x \sqcup \END_y$. Then, we have that $N_m = 2^{m-1} L_x L_y$ for $m \in \{k-1, k, k+1\}$.
	\end{claim}

	\begin{proof}
		Any linear extension $f$ satisfying $f(y) - f(x) = k$ induces a total ordering on $\MID_x \sqcup \END_y$ and $\MID_y \sqcup \END_x$ as well as a binary string in $\{0, 1\}^{m-1}$ where a $0$ in the $i$th coordinate means that the $i$th element between $x$ and $y$ is in $\MID_x$ and a $1$ in the $i$th coordinate menas that the $i$th element between $x$ and $y$ is in $\MID_y$. Given this binary string and the two total orderings, we can reconstruct the linear extension $f$. This is because after specifying the positions between $x$ and $y$ that will be in $\MID_x$ and $\MID_y$, the exact placement of the elements will be completely determined by the total orderings. In order to prove the claim, it suffices to prove that if we try to do this reconstruction process with any two linear extensions on $\MID_x \sqcup \END_y$ and $\MID_y \sqcup \END_x$ and any binary string in $\{0, 1\}^{m-1}$, we will get a linear extension of $P$ with $f(y) - f(x) = k$. There are three possibilities that may stop $f$ from being a linear extension: 
		\begin{enumerate}[label = (\arabic*)]
			\item $\MID_y$ may be too small. For example, we may not be able to accommodate the string of all zeros $0 \ldots 0$. To overcome this problem we must prove that $|\MID_y| \geq m-1$. 
			\item $\MID_x$ may be too small for the same reason. Similarly, we must prove that $|\MID_x| \geq m-1$. 
			\item Assuming that (1) and (2) are non-issues, the final issue is that the final ordering may not be a linear extension. This will happen when there is a $z \in \MID_x$ and $w \in \MID_y$ satisfying $w < z$ where $z$ lies between $x$ and $w$ in the final ordering, and $w$ lies between $z$ and $y$ in the final ordering.
		\end{enumerate}
		Note that (1) and (2) are not issues after applying the midway property and dual midway property of $\END_x$ and $\END_y$ to $z_x \lessdot x$ and $z_y \gtrdot y$. These elements exists because of the assumption of a $0$ element and $1$ element. To prove that (3) is not an issue, suppose for the sake of contradiction that we had elements $z \in \MID_x$ and $w \in \MID_y$ that satisfy the problem in (3). Then, we would have 
		\begin{align*}
			m-1 & \geq \left | P_{x < \cdot < z} \sqcup \{z\} \sqcup |P_{w < \cdot < y}| \sqcup \{w\} \right | = 2 + |P_{x < \cdot < z}| + |P_{w < \cdot < y}| \geq k+1.
		\end{align*}
		But we know that $m \leq k+1$, which is a contradiction. This suffices for the proof of the claim. 
	\end{proof}
	From Claim~\ref{auxiliary-claim}, we have that 
	\[
		\frac{N_{k+1}}{4} = \frac{N_{k}}{2} = N_{k-1} = 2^{k-2} L_x L_y. 
	\]
	This completes the proof.
\end{proof}
\section{Stanley's Matroid Inequality} \label{sec:stanley-matroid-inequality}

Stanley's poset inequality was proven by Richard Stanley in \cite{STANLEY}. We have already written about this inequality extensively in Section~\ref{stanley-poset-inequality}. In the same paper, Stanley proves an inequality associated to the bases of a matroid. Let $M = (E, \mcB)$ be a matroid of rank $r$ with ground set $E$ and bases $\mcB$. For any subsets $T_1, \ldots, T_r \subseteq E$, we can let $B(T_1, \ldots, T_r)$ denote the number of sequences $(y_1, \ldots, y_r)$ where $y_i \in T_i$ for $i \in [r]$ such that $\{y_1, \ldots, y_r\} \in \mathcal{B}$. Explicitly, we have that 
\[
	B(T_1, \ldots, T_r) := \# \{ (y_1, \ldots, y_r) \in T_1 \times \ldots \times T_r : \{y_1, \ldots, y_r\} \in \mcB(M) \}.
\]
Stanley's inequality associated to this sequence is written in Theorem~\ref{stanley-matroid-inequality-full-generality}. The inequality describes the log-concavity of the basis counting sequence $B(T_1, \ldots, T_r)$. 

\begin{thm}[Theorem 2.1 in \cite{STANLEY}] \label{stanley-matroid-inequality-full-generality}
	Let $M = (E, \mcB)$ be a regular matroid of rank $r$. Let $\mathcal{T} = (T_1, \ldots, T_{r-m})$ be a collection $r-m$ subsets of $E$ and let $X, Y \subseteq E$. Then, 
	\[
		B(\mathcal{T}, \underbrace{X, \ldots, X}_{k \text{ times}}, \underbrace{Y, \ldots, Y}_{m-k \text{ times}})^2 \geq B(\mathcal{T}, \underbrace{X, \ldots, X}_{k-1 \text{ times}}, \underbrace{Y, \ldots, Y}_{m-k+1 \text{ times}}) B(\mathcal{T}, \underbrace{X, \ldots, X}_{k+1 \text{ times}}, \underbrace{Y, \ldots, Y}_{m-k-1 \text{ times}})
	\]
	for all $1 \leq k \leq m-1$. 
\end{thm}

Before we present the proof of Theorem~\ref{stanley-matroid-inequality-full-generality}, we will first go over two different ways to view the basis counting numbers $B(T_1, \ldots, T_r)$. The first way to view these numbers is to associate them with a mixed volume of suitable polytopes. The second way to view these numbers is to associate them with a mixed discriminant of suitable matrices. Both the mixed volume and mixed discriminant perspective will lead to immediate proofs of the theorem. When we study the equality case of the inequality, it is also useful to have both perspectives. If we wish to prove Theorem~\ref{stanley-matroid-inequality-full-generality} with the regularity assumption removed, the mixed volume and mixed discriminant perspectives will prove to be insufficient. At the core of the mixed discriminant and mixed volume perspective is that the matroid has a unimodular coordinatization. In Section~\ref{subsection-matroid-lorentzian-stanley-general} we will remove the regularity assumption and prove Stanley's matroid inequality in its full generality using the technology of Lorentzian polynomials.

\subsection{Mixed Volume Perspective of Basis Counting Number} \label{section-mixed-volume-persepctive-of-basis-counting-number}

Since our matroid $M = (E, \mcB)$ is unimodular, there is a unimodular coordinatization $v : E \to \RR^r$. This implies that to every subset $T \subseteq E$, we can associate the zonotope
\[
	Z(T) := Z(v(e) : e \in T) = \sum_{e \in T}[0, v(e)].
\]
Since $v : E \to \RR^r$ is a unimodular coordination, any determinant of $r$ of these vectors will lie in the set $\{0, \pm 1\}$. From Example~\ref{example-volume-of-a-zonotope}, this gives us the equation 
\begin{equation} \label{volume-of-zonotope-unimodular-case-equation}
	\Vol_r(Z(T)) = \sum_{I \subseteq T : |I| = r} |\text{Det}(v(e) : e \in I)| = \sum_{I \subseteq T : |I| = r} \1_{I \text{ is a basis}} = \# \mcB \left (M|_T \right ).
\end{equation}
Given subsets $T_1, \ldots, T_r \subseteq E$, consider the Minkowski sum $\sum_{i = 1}^r \lambda_i Z(T_i)$. From Equation~\ref{volume-of-zonotope-unimodular-case-equation}, we have that 
\begin{align*}
	\Vol_r \left ( \sum_{i = 1}^r \lambda_i Z(T_i) \right ) & = \Vol_n \left (\sum_{i = 1}^r \sum_{e \in T_i} [0, \lambda_i v(e)] \right ) = \sum_{a_1 + \ldots + a_r = r} C(a_1, \ldots, a_r) \lambda_1^{a_1} \ldots \lambda_r^{a_r}
\end{align*}
where $C(a_1, \ldots, a_r)$ is the number of ways to pick subsets $Q_i \subseteq T_i$ and $|Q_i| = a_i$ for $i \in [r]$ where $Q_1 \cup \ldots \cup Q_r$ is a basis of $M$. According to Theorem~\ref{mixed-volume-polynomial-expansion-FINAL}, we have that 
\[
	C(a_1, \ldots, a_r) = \binom{r}{a_1, \ldots, a_r} \mathsf{V}_r (Z(T_1)[a_1], \ldots, Z(T_r)[a_r])
\]
and combinatorially we have $B(T_1, \ldots, T_r) = C(1, \ldots, 1)$. Thus, we have the equation
\begin{equation} \label{eqn-mixed-vol-perspective-matroid-inequality}
	B(T_1, \ldots, T_r) = \binom{r}{1, \ldots, 1} \mathsf{V}_r(Z(T_1), \ldots, Z(T_r)) = r! \mathsf{V}_r(Z(T_1), \ldots, Z(T_r)). 
\end{equation}
Equation~\ref{eqn-mixed-vol-perspective-matroid-inequality} immediately implies Theorem~\ref{stanley-matroid-inequality-full-generality} via the Alexandrov-Fenchel inequality (Theorem~\ref{AF-inequality}).

\subsection{Mixed Discriminant Perspective of the Basis Counting Number} \label{sec:mixed-discriminant-stanley-basis}

From Section~\ref{section-mixed-volume-persepctive-of-basis-counting-number}, we saw that the regularity condition on our matroid $M$ implies that there exists a unimodular coordinatization $v : E \to \RR^r$. From \cite{bapat_raghavan_1997}, this implies that we can view the basis counting sequence in terms of mixed discriminants. To see this, consider $T_1, \ldots, T_s \subseteq E$. Recall that for $r_1, \ldots, r_s \geq 0$ satisfying $r_1 + \ldots + r_s = r$, we defined $B(T_1[r_1], \ldots, T_s[r_s])$ as the number of sequences in $T_1^{r_1} \times \ldots \times T_s^{r_s}$ which forms a basis of $M$. We define a different sequence $N(r_1, \ldots, r_s) := N^{T_1, \ldots, T_s}(r_1, \ldots, r_s)$ which is the number of ways to pick $Q_i \subseteq T_i$ with $|Q_i| = r_i$ such that $Q_1 \cup \ldots \cup Q_s$ is a basis of $M$. Then, we have that
\begin{equation} \label{eqn-nice-equality-between-combinatorial-sequences-matroid-mixed-discriminant}
	B(T_1[r_1], \ldots, T_s[r_s]) = r_1! \ldots r_s! N^{T_1, \ldots, T_s}(r_1, \ldots, r_s).
\end{equation}
For $1 \leq i \leq s$, let $X_i$ be the matrix with colums $v(e)$ for $e \in T_i$. Let $A_i = X_i X_i^T$. Then, we know that 
\begin{equation} \label{nice-discriminant-perspective-matroid-inequality-stanley}
	\mathsf{D}_r (A_1[r_1], \ldots, A_s[r_s]) = \frac{1}{r!} B(T_1[r_1], \ldots, T_s[r_s]) = \frac{N^{T_1, \ldots, T_s}(r_1, \ldots, r_s)}{\binom{r}{r_1, \ldots, r_s}}
\end{equation}
from Lemma~\ref{mixed-discriminant-explicit-formula} and Equation~\ref{eqn-nice-equality-between-combinatorial-sequences-matroid-mixed-discriminant}. Equation~\ref{nice-discriminant-perspective-matroid-inequality-stanley} immediately implies Theorem~\ref{stanley-matroid-inequality-full-generality} via Alexandrov's Inequality on mixed discriminants (Theorem~\ref{A-Inequality-MIXED-DISCRIMINANT}).

\subsection{Equality cases of Stanley's Matroid Inequality for Graphic Matroids}

At the moment, we have represented Stanley's basis counting sequence in terms of mixed volumes and in terms of mixed discriminants. Based on our knowledge of the equality cases of the Alexandrov inequality for mixed discrminants and the Alexandrov-Fenchel inequality, we can hope to understand the combinatorial conditions on regular matroids $M$ which give equality in the log-concavity inequality. To simplify the problem, we only consider the case where we have two subsets $T_1, T_2 \subseteq E$ which partition our ground set. In other words, we have some set $R \subseteq E$. We define $N_i$ to be the number of bases of $M$ which intersect $R$ at $k$ points. From our analysis in previous sections, we know that 
\[
	N_i = C^{R, E \backslash R} (i, r-i) = \binom{r}{i} \mathsf{V}_r (Z(R)[k], Z(E \backslash R)[r-k]).
\] 
Thus, the sequence $N_i$ is ultra-log-concave. Let $\widetilde{N}_i = N_i / \binom{r}{i}$ be the log-concave normalization of $N_i$. We study the equality cases of $\widetilde{N}_i^2 = \widetilde{N}_{i-1} \widetilde{N}_{i+1}$. From the mixed volume perspective, Stanley in \cite{STANLEY} is able to give the equality cases for the Minkowski inequality analog for his basis counting sequence. However, we will later see that in the strictly positive case, the extremals of the Minkowski inequality analog are exactly the extremals of $\widetilde{N}_k^2 = \widetilde{N}_{k-1} \widetilde{N}_{k+1}$ for some $k$. This fact comes from viewing the sequence as mixed discriminants. It is unknown to the author whether or not Stanley knew about this fact at the time, but it is an interesting fact nonetheless.

\begin{thm} [Theorem 2.8 in \cite{STANLEY}] \label{stanley-equality-cases-matroid-thm}
	Let $M$ be a loopless regular matroid of rank $n$ on the finite set $E$, and let $R \subseteq S$. Then, the following two conditions are equivalent:
	\begin{enumerate}[label = (\alph*)]
		\item $\widetilde{N}_1^n = \widetilde{N}_0^{n-1} \widetilde{N}_n$. 
		\item One of the following two conditions hold:
		\begin{enumerate}[label = (\roman*)]
			\item $\widetilde{N}_1 = 0$.
			\item There is some rational number $q \in (0, 1)$ such that for every point $x \in E$, we have that $|\overline{x} \cap R| = q |\overline{x}|$ for all $x \in E$. Here $\overline{x}$ is the closure of the point $x$.
		\end{enumerate}
	\end{enumerate}
\end{thm}

Before we present Stanley's proof of Theorem~\ref{stanley-equality-cases-matroid-thm} for regular matroids, we first present the proof for a combinatorial characterization of the equality cases of Stanley's matroid inequality for graphic matroids. This proof will use the mixed discriminant perspective to find the extremals of $\widetilde{N}_k^2 \geq \widetilde{N}_{k-1} \widetilde{N}_{k+1}$ in the strictly positive case. In general, the combinatorial meaning of the matrices which show up in the mixed discriminant perspective are difficult to decipher. However, in the case of graphic matroids, the matrices in the mixed discriminant are related to the Laplacian of the underlying graph. Suppose that $M$ is a graphic matroid. Proposition~\ref{prop:graphic-matroids-connected} implies that there is a connected graph $G = (V, E)$ such that $M \cong M(G)$. Then, in the case of graphic matroids, we have Theorem~\ref{stanley-matroid-thm-for-graphic-matroids-equality}.

\begin{thm} \label{stanley-matroid-thm-for-graphic-matroids-equality}
	Suppose that $M$ is a graphic matroid corresponded to a connected graph $G = (V, E)$. Let $n = |V|$ and let $E = R \sqcup Q$ be a partition of the edge set so that both $R$ and $Q$ contain at least one spanning tree of $G$. Then, the following conditions are equivalent:
	\begin{enumerate}[label = (\alph*)]
		\item $\widetilde{N}_k^2 = \widetilde{N}_{k-1} \widetilde{N}_{k+1}$ for some $k \in \{1, \ldots, n-1\}$.

		\item $\widetilde{N}_k^2 = \widetilde{N}_{k-1} \widetilde{N}_{k+1}$ for all $k \in \{1, \ldots, n-1\}$. 

		\item For any two distinct vertices $v, w \in V(G)$, the ratio between the number of edges between $v$ and $w$ that are in $R$ to the number of edges between $v$ and $w$ that are in $Q$ is some positive number which does not depend on our choice of $v, w \in V(G)$. 
	\end{enumerate}
\end{thm}

\begin{proof}
	The proof of $(c) \implies (b)$ follows from Theorem~\ref{one-direction-of-the-conjecture}. It suffices to prove $(a) \implies (c)$. Suppose that $n \geq 4$. Consider the edges of a $R$ spanning tree. Without loss of generality, we get let the vertices $n-1$ and $n$ be two leaves of this spanning tree. Since $n \geq 4$, there exists a $R$-edge connecting two vertices each of which neither $n-1$ nor $n$. . \\

	Let $B \in \RR^{n \times |E|}$ be the incidence matrix of $G$. We can partition the columns of $B = [B_R | B_Q]$ where $B_R$ are the columns corresponding to the edges in $R$ and $B_Q$ are the columns corresponding to the edges in $Q$. We can further partition $B_R$ and $B_Q$ into
	\begin{align*}
		B_R & = \begin{bmatrix} X_R \\ v_R^T \end{bmatrix}, \quad \text{ and } \quad B_Q = \begin{bmatrix} X_Q \\ v_Q^T \end{bmatrix}
	\end{align*}
	where $X_R$ and $X_Q$ are $B_R$ and $B_Q$ with the bottom row removed. We know that the reduced incidence matrix $C = [X_R | X_Q]$ is a unimodular coordinatization of our graphic matroid. From the equality cases of Alexandrov's inequality for mixed discriminants, the equality in (a) implies that $X_RX_R^T = \alpha X_Q X_Q^T$ for some $\alpha \in \RR$. We can apply the mixed discriminant equality characterization because $X_R, X_Q$ are both full rank and $X_RX_R^T$, $X_QX_Q^T$ are positive definite. Let $L_R$ be the Laplacian of the subgraph on $V$ with $R$-edges and $L_Q$ be the Laplacian of the subgraph on $V$ with $Q$-edges. Proposition~\ref{incidence-and-laplacian-relation-prop} gives us
	\begin{align*}
		L_{G|_R} & = B_R B_R^T = \begin{bmatrix}
			X_RX_R^T & \bullet \\ \bullet & \bullet
		\end{bmatrix}, \\
		L_{G|_Q} & = B_Q B_Q^T = \begin{bmatrix}
			X_Q X_Q^T & \bullet \\ \bullet & \bullet
		\end{bmatrix}.
	\end{align*}
	where we only care about the top left $(n-1) \times (n-1)$ principal submatrix. This tells us the off-diagonal entries of $X_R X_R^T$ and $X_QX_Q^T$ encode exactly the number of $R$-edges and $Q$-edges between vertices in $V \backslash n$. This implies that 
	\[
		e_R(v, w) = \alpha \cdot e_Q(v, w) \text{ for all distinct } v, w \in [n] \backslash \{n\}.
	\]
	From the assumption that $R$ and $Q$ contain spanning trees, it must be the case that $\alpha > 0$. If we repeat the same argument except we delete the row corresponding to $n-1$, we get a similar result:
	\[
		e_R(v, w) = \beta \cdot e_Q(v, w) \text{ for all distinct } v, w \in [n] \backslash \{n-1\}.
	\]
	From the same reasoning, we must have $\beta > 0$. From our construction, there is a $R$-edge connecting two vertices in $[n] \backslash \{n-1, n\}$. This implies that $\alpha = \beta$. The only thing we must check is that $e_R(n-1, n) = \alpha \cdot e_Q(n-1, n)$. But this follows $\deg_R(v) = \alpha \deg_Q(v)$ for all $v \in [n]$. This proves (c) for $n \geq 4$. We defer the case $n = 3$ to Lemma~\ref{lem:base-case} in the appendix. The case $n < 3$ is vacously true. This suffices for the proof.
\end{proof}

\subsection{Equality cases of Stanley's Matroid Inequality for Regular Matroids}

Before we prove Theorem~\ref{stanley-equality-cases-matroid-thm} for regular matroids, we give necessary and sufficient conditions for when two zonotopes are homothetic to each other. For any $v_1, \ldots, v_m \in \RR^n$, we define the \textbf{centrally symmetric zonotope} generated by $v_1, \ldots, v_m$ as the convex body given by
\[
	Z_0 (v_1, \ldots, v_m) = \sum_{i = 1}^m [-v_i, v_i] = \left \{ \sum_{i = 1}^m \lambda_i v_i : \lambda_i \in [-1, 1] \text{ for all } 1 \leq i \leq m \right \}.
\]
This is related to our definition of a zonotope by the relation 
\[
	Z_0(v_1, \ldots, v_m) = 2 Z(v_1, \ldots, v_m) - \sum_{i = 1}^m v_i.
\]
In the literature, a zonotope is generally defined to be simply the Minkowski sum of a collection of segments. However, it is common to see the study of zonotopes restricted to either centrally symmetric zonotopes or zonotopes in the sense we have defined them. For examples of such conventions, we refer the reader to \cite{schneider_2013,Shephard_Zonotopes,STANLEY}. We call a convex body a (centrally symmetric) \textbf{zonoid} if it is the limit, in the sense of Hausdorff distance, of a sequence of (centrally symmetric) zonotopes. From Theorem~\ref{thm:representation-of-zonoid}, to every centrally symmetric zonoid we can associate an even measure on $\mathbb{S}^{n-1}$. From Theorem~\ref{thm:uniqueness-of-even-measure}, this even measure is unique. Both of these facts are found in \cite{schneider_2013}. 

\begin{thm} \label{thm:representation-of-zonoid}
	A convex body $K \subseteq \RR^n$ is a centrally symmetric zonoid if and only if its support function can be represented in the form 
	\[
		h_K(x) = \int_{\mathbb{S}^{n-1}} |\langle x, v \rangle | \, \rho (dv) \text{ for all $x \in \RR^n$}
	\]
	where $\rho$ is some even measure on $\mathbb{S}^{n-1}$.
\end{thm}
\begin{proof}
	See Theorem 3.5.3 in \cite{schneider_2013}.
\end{proof}

\begin{thm} \label{thm:uniqueness-of-even-measure}
	If $\rho$ is an even signed measure on $\mathbb{S}^{n-1}$ with 
	\[
		\int_{\mathbb{S}^{n-1}} |\langle u, v \rangle | \, \rho(dv) = 0 \text{ for } u \in \mathbb{S}^{n-1}, 
	\]
	then $\rho = 0$. 
\end{thm}

\begin{proof}
	See Theorem 3.5.4 in \cite{schneider_2013}. 
\end{proof}

For a centrally symmetric zonotope $Z_0 = Z_0(v_1, \ldots, v_m)$, the support function is given by the equation $h_{Z_0}(x) = \sum_{i = 1}^k \norm{v_i} \cdot  | \langle x, u_i \rangle |$ where $u_i := v_i / \norm{v_i}$. Thus, the corresponding even signed measure will be given by 
\begin{equation} \label{eqn:even-signed-measure-of-zonotope}
	\rho_{Z_0} = \sum_{i = 1}^k \frac{\norm{v_i}}{2} ( \delta_{u_i} + \delta_{-u_i}).
\end{equation}
With the description and the uniqueness of the even measure corresponding to centrally symmetric zonoids, we immediately get the result in Lemma~\ref{lem-condition-zonotope-homothety}. 

\begin{lem} \label{lem-condition-zonotope-homothety}
	Let $v_1, \ldots, v_r, w_1, \ldots, w_s \in \RR^n$ be non-zero vectors such that none of the $v_i$'s are parallel and none of the $w_i$'s are parallel. Then, there exists $\lambda \in \RR$ and $x \in \RR^n$ such that
	\[
		Z(v_1, \ldots, v_r) = \lambda Z(w_1, \ldots, w_s) + x
	\]
	if and only if $r = s$, and up to a permutation of the vectors $v_i = \lambda w_i$ for all $i$.
\end{lem}

\begin{proof}
	It is clear that the conditions in the lemma imply that $Z(v_1, \ldots, v_r)$ and $Z(w_1, \ldots, w_s)$ are homothetic. For the other direction, note that $Z(v_1, \ldots, v_r)$ and $Z(w_1, \ldots, w_s)$ being homothetic implies that $Z_0(v_1, \ldots, v_r)$ and $Z_0(w_1, \ldots, w_s)$ are also homothetic. Hence, we must have
	\[
		Z_0 (v_1, \ldots, v_r) = \lambda Z_0(w_1, \ldots, w_s)
	\]
	where $\lambda$ is the same dilation factor in the homothety for $Z(v_1, \ldots, v_r)$ and $Z(w_1, \ldots, w_s)$. In the homothety between centrally symmetric zonotopes, there is only a dilation factor since they both have center of mass at $0$. From the uniqueness and description in Equation~\ref{eqn:even-signed-measure-of-zonotope} of the even measure corresponding to centrally symmetric zonoids, we can conclude that $r = s$ and $v_i = \lambda s_i w_i$ for all $i$ up to a reordering of the vectors and some choice of $s_i \in \{\pm 1\}$. This suffices for the proof of the lemma. 
\end{proof}

We can now begin the proof of Theorem~\ref{stanley-equality-cases-matroid-thm}. 
\begin{proof}[Proof of Theorem~\ref{stanley-equality-cases-matroid-thm}]
	The fact that (b) implies (a) is the subject of Theorem~\ref{one-direction-of-the-conjecture}. For the other direction, let $Q = E \backslash R$ and let $v : E \to \RR^n$ be a unimodular coordinatization of $M$. Since negating vectors in our coordinatization will not break the unimodularity, we can assume that whenever two vectors are parallel, they are positive scalar multiples of each other. Since our coordinatization is unimodular, we know that two vectors are equal to each other if and only if they are parallel. Moreover, since our matroid is loopless, all of the vectors are non-zero. So, for every parallel class $\overline{x}$ there is a vector $v_{\overline{x}} \in \RR^n$ such that $v_y = v_{\overline{x}}$ if and only if $y \in \overline{x}$. Recall that for all $0 \leq k \leq n$, we have that 
	\[
		N_k = \binom{n}{i} \mathsf{V}_n \left ( \underbrace{Z(R), \ldots, Z(R)}_{k \text{ times}}, \underbrace{Z(Q), \ldots, Z(Q)}_{n-k \text{ times}} \right ). 
	\]
	Thus the equality $\widetilde{N}_1^n = \widetilde{N}_0^{n-1} \widetilde{N}_n$ is equivalent to the equality in Minkowski's inequality (Theorem~\ref{minkowski-inequality}):
	\[
		\mathsf{V}_n \left ( Z(R), \underbrace{Z(Q), \ldots, Z(Q)}_{n-1 \text{ times}} \right )^n \geq \Vol_n (Z(R))^{n-1} \cdot \Vol_n(Z(Q)).
	\]
	From the equality case of Theorem~\ref{minkowski-inequality}, we know that one of the following three things occur:
	\begin{enumerate}[label = (\arabic*)]
		\item $\dim Z(R) \leq n-2$. 
		\item $\dim Z(R)$ and $\dim Z(Q)$ lie in parallel hyperplanes. 
		\item $Z(R)$ and $Z(Q)$ are homothetic. 
	\end{enumerate}
	If $\dim Z(R) \leq n-2$ or $\dim Z(R)$ and $\dim Z(Q)$ lie in parallel hyperplanes, then we know that $\Vol_n (Z(R)) = \Vol_n(Z(Q)) = 0$. This implies that $\widetilde{N}_1 = 0$. Now suppose that $Z(R)$ and $Z(Q)$ are homothetic. Let $R = \{v_1, \ldots, v_r\}$ and $Q = \{w_1, \ldots, w_s\}$ with $r + s = |E|$ be the vectors in the coordinations. Let $R' = \{v_1', \ldots, v_t'\}$ and $Q' = \{w_1', \ldots, w_u'\}$ be the result of adding parallel vectors together. The collection $R'$ consists of vectors of the form $|R \cap \overline{x}| v_{\overline{x}}$ for all $x$ satisfying $|R \cap \overline{x}| > 0$. The same is true for the collection $Q'$. Since these zonotopes are homothetic, we know from Lemma~\ref{lem-condition-zonotope-homothety} that $t = u$ and there exists a $\lambda > 0$ so that after a relabeling of indices we have $v_i' = \lambda s_i w_i'$ for all $1 \leq i \leq t$ where $s_i \in \{\pm 1\}$. From our construction, we are forced to have $s_i = 1$. Thus, it must be the case that equality holds if and only if there exists $\lambda > 0$ so that $|R \cap \overline{x}| = \lambda |Q \cap \overline{x}|$ for all $x \in E$. This suffices for the proof. 
\end{proof}

\begin{remark}
	In his paper \cite{STANLEY}, Stanley gives a combinatorial proof of Theorem~\ref{stanley-equality-cases-matroid-thm} by proving Lemma~\ref{lem-condition-zonotope-homothety} with the additional hypothesis that all vectors lie in the same orthant. Because the vectors are forced to be in the same positive orthant, he can conclude that $v_i = \lambda w_i$ without any sign constraints from the $s_i$. However, it is not clear to us why he can apply this fact to Theorem~\ref{stanley-equality-cases-matroid-thm}. In particular, Stanley did not prove that for any regular matroid, it is possible to find a unimodular coordinatization such that all of the vectors lie in the same orthant. 
\end{remark}

With the mixed discriminant perspective, we can extend Theorem~\ref{stanley-equality-cases-matroid-thm} slightly to the case where we are only guarenteed that one of the inequalities $\widetilde{N}_k^2 \geq \widetilde{N}_{k-1} \widetilde{N}_{k+1}$ is equality. Indeed, by viewing the sequence $\widetilde{N}_k$ as a sequence of mixed discriminants, we can prove Lemma~\ref{connection-lemma}. 

\begin{lem} \label{connection-lemma}
	Let $M$ be a regular, loopless matroid of rank $n$. Suppose that $\widetilde{N}_0 > 0$ and $\widetilde{N}_n > 0$. Then $\widetilde{N}_1^n = \widetilde{N}_0^{n-1} \widetilde{N}_n$ if and only if $\widetilde{N}_k^2 = \widetilde{N}_{k-1} \widetilde{N}_{k+1}$ for some $k \in \{1, \ldots, n-1\}$.
\end{lem}

\begin{proof}
	Assuming that $\widetilde{N}_1^n = \widetilde{N}_0^{n-1} \widetilde{N}_n$, we know that $\widetilde{N}_k^2 = \widetilde{N}_{k-1} \widetilde{N}_{k+1}$ for all $k$ from our proof of equality conditions in Theorem~\ref{minkowski-inequality}. For the converse, suppose that $\widetilde{N}_k^2 = \widetilde{N}_{k-1} \widetilde{N}_{k+1}$ for some $k$. Recall that we can view 
	\[
		\widetilde{N}_k = \mathsf{D} (A_R [k], A_Q [n-k])
	\]
	where $A_R = X_R X_R^T$, $A_Q = X_Q X_Q^T$, and $X_R, X_Q$ are the matrices which consist of the unimodular representations of the elements in $R$ and $Q$ as columns. From Corollary~\ref{cor-one-implies-all-mixed-discriminants}, we get that $\widetilde{N}_k^2 = \widetilde{N}_{k-1} \widetilde{N}_{k+1}$ holds for all $k \in \{1, \ldots, n-1\}$. We can apply Corollary~\ref{cor-one-implies-all-mixed-discriminants} because the condition $\widetilde{N}_0 > 0$ and $\widetilde{N}_n > 0$ implies that the matrices $X_R X_R^T$ and $X_QX_Q^T$ are positive definite. 
\end{proof}

From Lemma~\ref{connection-lemma}, whenever $\widetilde{N}_0, \widetilde{N}_n > 0$ (which translate to $R$ and $Q$ both having full rank), we get that the necessary and sufficient conditions for equality in Theorem~\ref{stanley-equality-cases-matroid-thm} are exactly the necessary and sufficient conditions to guarentee equality at a single instance of the log-concavity inequality. Thus, we make the following conjecture for all matroids (not necessarily regular). 

\begin{conj} \label{matroid-conjecture}
	Let $M$ be a loopless matroid of rank $n$ on a set $E$. Let $R \subseteq E$ be a subset and let $Q = E \backslash R$. For $k$, $0 \leq k \leq n$, we define $N_k$ to be the number of bases of $M$ with $k$ elements in $R$. Define $\widetilde{N}_k = \frac{N_k}{\binom{n}{k}}$. Suppose that $R$ and $Q$ both have rank $n$. Then, the following are equivalent. 
	\begin{enumerate}[label = (\alph*)]
		\item $\widetilde{N}_k^2 = \widetilde{N}_{k-1} \widetilde{N}_{k+1}$ for some $k \in \{1, \ldots, n-1\}$. 

		\item $\widetilde{N}_k^2 = \widetilde{N}_{k-1} \widetilde{N}_{k+1}$ for all $k \in \{1, \ldots, n-1\}$.

		\item There are positive integers $r, q \geq 1$ such that $|\overline{x} \cap R| r = |\overline{x} \cap Q| q$ for all $x \in E$. 
	\end{enumerate}
\end{conj}

To show some evidence of Conjecture~\ref{matroid-conjecture}, we will prove one of the directions of Conjecture~\ref{matroid-conjecture}. This result is written in the statement of Theorem~\ref{one-direction-of-the-conjecture}.

\begin{thm} \label{one-direction-of-the-conjecture}
	Suppose that there are positive integers $r, s \geq 1$ such that $| \overline{x} \cap R| r = |\overline{x} \cap Q| q$. Then, $\widetilde{N}_k^2 = \widetilde{N}_{k-1} \widetilde{N}_{k+1}$ for all $k \in \{1, \ldots, n-1\}$. In particular, we have that $\widetilde{N}_k = \frac{q}{r} \widetilde{N}_{k-1}$. 
\end{thm}

\begin{proof}
	For all $x$, we have that $| \overline{x} \cap R| r = |\overline{x} \cap Q| q$. This implies that for $x \in E$, we can define an isomorphism $\varphi_x : (\overline{x} \cap R) \times [r] \to (\overline{x} \cap Q) \times [q]$ where $\varphi_x = \left(\varphi_x^1, \varphi_x^2\right)$. Let the inverse map be $\psi_x = (\psi_x^1, \psi_x^2)$. For all $k \in \{1, \ldots, n\}$, define the sets
	\begin{align*}
		\Omega_k & := \{(a, i, U) : U \in \mcB(M), |U \cap R| = k, i \in [r], a \in U \cap R\} \\
		\Omega_{k-1} & := \{(b, j, V) : V \in \mcB(M), |V \cap R| = k-1, j \in [q], b \in V \cap Q\}.
	\end{align*}
	To count the number of elements in $\Omega_k$ and $\Omega_{k-1}$, we count the number of elements in each set with a fixed $U$ and $V$. Then, we sum over all possible choices of $U$ and $V$. This gives the identity
	\begin{align*}
		|\Omega_k| & = N_k \cdot k \cdot r \\
		|\Omega_{k-1}| & = N_{k-1} \cdot (n-k+1) \cdot q.
	\end{align*}
	I claim that there is one-to-one correspondence between $\Omega_k$ and $\Omega_{k-1}$. We define maps $\varphi_{\downarrow} : \Omega_k \to \Omega_{k-1}$ and $\varphi_{\uparrow} : \Omega_{k-1} \to \Omega_k$ given by 
	\begin{align*}
		\varphi_{\downarrow}(a, i, U) & = (\varphi_1^x(a, i), \varphi_2^x (a, i) , (U \backslash a) \cup \varphi_1^x (a, i)) \\
		\varphi_{\uparrow} (b, j, V) & = (\psi_1^x(b, j), \psi_2^x(b,j), (V \backslash b) \cup \psi_1^x(b, j)).
	\end{align*}
	These two maps are well-defined because of interchangeability of parallel elements in indepenent sets (and in particular bases). From construction, the maps $\varphi_{\downarrow}$ and $\varphi_{\uparrow}$ are two-sided inverses of each other. This gives a one-to-one correspondence between $\Omega_k$ and $\Omega_{k-1}$. From our computation of the cardinalities of these sets, we have that 
	\[
		N_k \cdot k \cdot r = |\Omega_k| = |\Omega_{k-1}| = N_{k-1} (n-k+1) q.
	\]
	This implies that $\widetilde{N}_k = \frac{q}{r} \widetilde{N}_{k-1}$. This proves the theorem. 
\end{proof}
In the subsequent subsection, we prove Stanley's matroid inequality without the regularity condition using the technology of Lorentzian polynomials. 

\subsection{Lorentzian Perspective of the Basis Counting Number} \label{subsection-matroid-lorentzian-stanley-general}

One essential hypothesis in Stanley's matroid inequality was that the matroid $M = (E, \mcB)$ had to be regular. This condition was needed in order to have a unimodular coordinatization. The unimodular coordinatization was at the heart of the mixed discriminant and mixed volume constructions. In this section, we generalize Stanley's matroid inequality by removing the hypothesis of unimodularity using the theory of Lorentzian polynomials. Our proof will be based on the fact that the basis generating polynomial of a matroid is Lorentzian. For simplicity, suppose that the ground set of $M = (E, \mcI)$ is $E = [n]$ and $\rank (M) = r$. Recall that the basis generating polynomial of $M$ is defined as 
\[
	f_M(x_1, \ldots, x_n) = \sum_{B \in \mcB} x^B = \sum_{\substack{1 \leq i_1 < \ldots < i_r \leq n \\ \{i_1, \ldots, i_r\} \in \mcB (M)}} x_{i_1} \ldots x_{i_r}.
\]
For any sequence of subsets $\mathcal{T} := (T_1, \ldots, T_m) \subseteq E$, we can define a modification of the basis generating polynomial which gives us a better handle on Stanley's basis counting numbers. 

\begin{defn}
	Let $\mathcal{T} := (T_1, \ldots, T_m)$ be a sequence of subsets of $E$. Then, consider the polynomial 
	\[
		g_M^\mathcal{T} (y_1, \ldots, y_m) := f_M \left (x_e = \sum_{i : e \in T_i} y_i \right )
	\]
	where in the right hand side, we consider the basis generating polynomial where we replace each instance of the coordinate $x_e$ with the linear form $\sum_{i : e \in T_i} y_i$. 
\end{defn}

For $a_1, \ldots, a_m \geq 0$ satisfying $a_1 + \ldots + a_m = r$ we define the number $N(a_1, \ldots, a_m)$ as the number of ways to pick subsets $Q_i \subset T_i$ with $|Q_i| = a_i$ such that $Q_1 \cup \ldots \cup Q_m$ is a basis of $M$. 
\begin{lem} \label{lem-simplifying-N-to-B}
	For $T_1, \ldots, T_m \subseteq E$ and $a_1, \ldots, a_m \geq 0$ satisfying $a_1 + \ldots + a_m = r$, we have that 
	\[	
		N(a_1, \ldots, a_m) = \frac{B(T_1[a_1], \ldots, T_m[a_m])}{a_1! \ldots a_m!}.
	\]
\end{lem}

\begin{proof}
	Any choice of subsets $Q_i \subseteq T_i$ with $|Q_i| = a_i$ and $Q_1 \cup \ldots \cup Q_r$ a basis of $M$ gives $a_1! \ldots a_m!$ sequences which are included in the count of $B(T_1[a_1], \ldots, T_m[a_m])$. Conversely, every sequence determines subsets $Q_i \subseteq T_i$ with $|Q_i| = a_i$ and $Q_1 \cup \ldots \cup Q_r$ a basis of $M$. It is clear that this is a $1 : a_1! \ldots a_m!$ correspondence. This suffices for the proof of the lemma. 
\end{proof}

\begin{prop}
	For $T_1, \ldots, T_m \subseteq E$, we have that 
	\[	
		g_M^{T_1, \ldots, T_m} (y_1, \ldots, y_m) = \sum_{a_1 + \ldots + a_m = r} \frac{B(T_1[a_1], \ldots, T_m[a_m])}{a_1! \ldots a_m!} \cdot y_1^{a_1} \ldots y_m^{a_m}.
	\]
\end{prop}

\begin{proof}
	By substituting $x_e = \sum_{i : e \in \mathcal{T}_i} y_i$ in the formula for the basis generating polynomial, we get that 
	\begin{align*}
		g_M^{\mathcal{T}}(y_1, \ldots, y_m) & = \sum_{\substack{1 \leq i_1 < \ldots < i_r \leq m \\ \{i_1, \ldots, i_r\} \in \mcB(M)}} \prod_{k = 1}^r \left ( \sum_{i_k \in T_j} y_j \right ) \\
		& = \sum_{\substack{1 \leq i_1 < \ldots < i_r \leq m \\ \{i_1, \ldots, i_r\} \in \mcB(M)}} \left (\sum_{a_1 + \ldots + a_m = r} N^{\{i_1, \ldots, i_r\}}(a_1, \ldots, a_m) \right ) y_1^{a_1} \ldots y_m^{a_m}
	\end{align*}
	where $N^B(a_1, \ldots, a_m)$ is the number of ways to pick subsets $Q_i \subseteq T_i$ with $|Q_i| = a_i$ and $Q_1 \cup \ldots \cup Q_m = B$. In particular, we have that 
	\[	
		\sum_{B \in \mcB(M)} N^B (a_1, \ldots, a_m) = N(a_1, \ldots, a_m). 
	\]
	This allows us to simplify the equation after changing the order of summations: 
	\begin{align*}
		g_M^{\mathcal{T}}(y_1, \ldots, y_m) & = \sum_{a_1 + \ldots + a_m = r} \left ( \sum_{B \in \mcB(M)} N^B (a_1, \ldots, a_m) \right ) y_1^{a_1} \ldots y_m^{a_m} \\
		& = \sum_{a_1 + \ldots + a_m = r} N(a_1, \ldots, a_m) \cdot y_1^{a_1} \ldots y_m^{a_m}.
	\end{align*}
	From Lemma~\ref{lem-simplifying-N-to-B}, we have proven the lemma. 
\end{proof}

\begin{lem}
	For $T_1, \ldots, T_m \subseteq E$, the polynomial $g_M^{T_1, \ldots, T_m}$ is Lorentzian. 
\end{lem}
\begin{proof}
	This follows from Theorem~\ref{basis-generating-polynomial-is-lorentzian-thm} and Proposition~\ref{proposition-properties-of-Lorentzian-polynomials}. 
\end{proof}

We can now prove the main theorem for this section. This theorem will be a generalization of Stanley's matroid inequality in Theorem~\ref{stanley-matroid-inequality-full-generality}. 

\begin{thm} \label{main-thm-for-matroids}
	Let $M = (E, \mcB)$ be any matroid of rank $r$. Let $\mathcal{T} = (T_1, \ldots, T_{r-m})$ be a sequence of subsets in $E$ and let $Q, R \subseteq E$ be subsets. Define the sequence
	\[	
		B_k(\mathcal{T}, Q, R) := B(T_1, \ldots, T_{r-m}, Q[k], R [m-k]).
	\]
	Then the sequence $B_0(\mathcal{T}, Q, R), \ldots, B_m(\mathcal{T}, Q, R)$ is log-concave. 
\end{thm}

\begin{proof}
	Consider the Lorentzian polynomial $g_M^{T_1, \ldots, T_{r-m}, Q, R}$. The result then follows from Proposition~\ref{proposition-log-concavity-property-of-lorentzian-polynomial}. 
\end{proof}

Using Theorem~\ref{main-thm-for-matroids}, we can also remove the regularity hypothesis from Corollary 2.4 in \cite{STANLEY}. 

\begin{cor} \label{cor-ultra-log-concavity-of-some-sequence-related-to-bases}
	Let $M = (E, \mcB)$ be any matroid of rank $n$ and let $T_1, \ldots, T_r, Q, R$ be pairwise disjoint subsets of $E$ whose union is $E$. Fix non-negative integers $a_1, \ldots, a_r$ such that $m = n - a_1 - \ldots - a_r \geq 0$, and for $0 \leq k \leq m$ define $f_k$ to be the number of bases $B$ of $M$ such that $|B \cap T_i| = a_i$ for $1 \leq i \leq r$, and $|B \cap R| = k$ (so $|B \cap Q| = m-k$). Then the sequence $f_0, \ldots, f_m$ ultra-log-concave. 
\end{cor}
	
\begin{proof}
	Let $\mathcal{T} = (T_1[a_1], \ldots, T_r[a_r])$. Then we have $B_k(\mathcal{T}, Q, R) = a_1! \ldots a_r! k! (m-k)! f_k$ where we use the same notation as in Theorem~\ref{main-thm-for-matroids}. Thus, we have that
	\[	
		\frac{f_k}{\binom{m}{k}} = \frac{B_k(\mathcal{T}, Q, R)}{a_1 ! \ldots a_r!}.
	\]
	The ultra-log-concavity of $f_k$ then follows from Theorem~\ref{main-thm-for-matroids}. 
\end{proof}

According to Stanley in \cite{STANLEY}, Theorem~\ref{main-thm-for-matroids} would imply the first Mason conjecture. Indeed, this is the content of Theorem 2.9 in \cite{STANLEY}. We rewrite the proof now for the sake of completeness. 

\begin{thm}[Mason's Conjecture]
	Let $M = (E, \mcI)$ be a matroid of rank $n$. For all $0 \leq k \leq n$, let $I_k$ be the number of independent sets of $M$ of rank $k$. Then $I_k^2 \geq I_{k-1} I_{k+1}$ for all $2 \leq k \leq n-1$. 
\end{thm}

\begin{proof}
	Let $B_n$ be the boolean matroid on $n$ elements. Consider $T^n (B_n + M)$ the level $n$ truncation of the matroid sum $B_n + M$. Let $f_k$ be the number of bases of $T^n(B_n + M)$ which shares $k$ elements with $E(M)$. Then, we have that $f_k = I_k \binom{n}{n-k}$. From Corollary~\ref{cor-ultra-log-concavity-of-some-sequence-related-to-bases}, we have that $I_k$ is log-concave. This suffices for the proof. 
\end{proof}

\begin{remark}
	In \cite{lorentzian-polynomials}, Br\"aden-Huh prove the strongest Mason conjecture using Lorentzian polynomials. Their proof involves proving that the homogeneous multivariate Tutte polynomial of a matroid is Lorentzian. 
\end{remark}

\begin{remark}
	In hindsight, it seems that Stanley would have been unable to prove his matroid inequality for arbitrary matroids using mixed volumes without the hypothesis that the matroid is regular. Indeed, from Remark 4.3 in \cite{lorentzian-polynomials}, the basis generating polynomial of a matroid on $[n]$ is the volume polynomial of $n$ convex bodies precisely when the matroid is regular.
\end{remark}

\chapter{Hodge Theory for Matroids} \label{chap:hodge-theory-for-matroids}

In this chapter we will study cohomology rings associated to matroids. In the general case, our cohomology ring will be a graded $\RR$-algebra of the form $A^\bullet = \bigoplus_{i = 0}^d A^i$ with a top-degree isomorphism: $\deg : A^d \to \RR$. If in addition, the natural pairing $A^i \times A^{d-i} \to A^d \to \RR$ is non-degenerate for all $i$, we will call our graded algebra $A^\bullet$ a \textbf{Poincar\'e duality algebra}. In the finite dimensional case, this will automatically imply that $\dim A^k = \dim A^{d-k}$ for all $k$. Cohomology rings with these properties show up naturally in topology and algebraic geometry. For an overview of such examples, we refer the reader to \cite{Huh2016TropicalGO}. Given a cohomology ring associated to a matroid, we will study when this graded object satisfies Poincar\'e duality, Hard Lefschetz, and Hodge-Riemann relations. If it satisfies all three properties in all degrees ($\leq \frac{d}{2}$) we say that it satisfies the K\"ahler package. For the development of the Hodge theory of matroids, we refer the reader to the papers \cite{AHK,huh-semi-small,ardila2022lagrangian,eur2022stellahedral,Feichtner_2004} where the notions of the Chow ring, intersection cohomology, and conormal Chow ring are studied. Through this development, many difficult and long standing conjectures in combinatorics have been solved by proving that a suitable matroid cohomology satisfies the K\"ahler package. In the next three sections, we define the notions of the graded M\"obius algebra, the Chow ring of a matroid, and the augmented Chow ring of a matroid. We discuss these notions at a surface level, and the sections primarily serve as a place to write down definitions and semi-small decomposition results for future sections (see Section~\ref{sec:future-work}). In subsequent sections, we define the Gorenstein ring associated to the basis generating polynomial of a matroid. It is known from \cite{MNY} that this ring satisfies the Hard Lefschetz property and Hodge-Riemann relations of degree $1$ on the positive orthant. We prove necessary and sufficient conditions for these properties to hold on the boundary of the positive orthant. We also make some progress in proving the complete K\"ahler package. Our progress is summarized in Section~\ref{sec:future-work}.

\section{Cohomology Rings for Matroids} \label{sec:cohomology-rings-for-matroids}
	
In this section, we will review the definitions and properties of three graded commutative algebras associated to matroids: the graded M\"obius algebra, the Chow ring, and the augmented Chow ring. The Chow ring and the augmented Chow ring of a matroid are automatically Poincar\'e duality algebras. The graded M\"obius algebra is in general not a Poincar\'e duality algebra. 

\subsection{Graded M\"obius Algebra}

Let $M$ be a matroid on ground set $E$ and let $\mcL$ be its lattice of flats. For $k \geq 0$, let $\mcL^k(M)$ be the set of flats of rank $k$. For every $k$, we can define the real vector space $\opH^k(M)$ given by 
\[
	\opH^k(M) := \bigoplus_{F \in \mcL^k(M)} \RR y_F
\]
where we have a variable $y_F$ for every flat $F \in \mcL^k(M)$. We can then define a graded multiplicative structure $\opH^k(M) \times \opH^l(M) \to \opH^{k+l}(M)$ where for $F_1 \in \mcL^k(M)$ and $F_2 \in \mcL^l(M)$ we have 
\begin{equation}\label{eqn:multiplicative-structure-on-graded-mobius-algebra}
	y_{F_1} \cdot y_{F_2} = 
	\begin{cases}
		y_{F_1 \vee F_2} & \text{if $\rank_M(F_1) + \rank_M(F_2) = \rank_M(F_1 \vee F_2)$}, \\
		0 & \text{if $\rank_M(F_1) + \rank_M(F_2) > \rank_M(F_1 \vee F_2)$}.
	\end{cases} 
\end{equation}
\begin{defn}
	For every matroid $M$, we define the \textbf{graded M\"obius algebra} to be the graded ring $\opH(M) = \bigoplus_{k \geq 0} \opH^k(M)$ equipped with the multiplicative structure defined in Equation~\ref{eqn:multiplicative-structure-on-graded-mobius-algebra}. 
\end{defn}
For every $e \in E$ in the ground set, we can define a graded ring homomorphism $\theta_e : \opH(M \backslash e) \to \opH(M)$ where $\theta_e (y_F) = y_{\text{clo}_M(F)}$ for all $F \in \mcL(M \backslash e)$. Explicitly, we have that $\clo_M(F)$ is $F$ if $F \in \mcL(M)$ and $F \cup \{e\}$ if $F \notin \mcL(M)$. Since the flats of $M \backslash e$ are of the form $F - \{e\}$ where $F$ is a flat of $M$, this map is well-defined. It is also a degree preserving map since $\rank (S) = \rank (\clo(S))$. It is also easy to see that $\theta_i$ is injective. We prove that $\theta_e$ is a ring homomorphism in Lemma~\ref{lem:theta-is-a-ring-homomorphism}. 

\begin{lem}\label{lem:theta-is-a-ring-homomorphism}
	For every $e \in M$, the map $\theta_e : \opH(M \backslash i) \to \opH(M)$ is a ring homomorphism. 
\end{lem}

\begin{proof}
	Let $F_1 \in \mcL^k(M \backslash e)$ and $F_2 \in \mcL^l(M \backslash e)$ be flats $M \backslash e$. We want to prove that $\theta_e (y_{F_1} \cdot y_{F_2}) = \theta_e(y_{F_1}) \theta_e(y_{F_2})$. Since the multiplicative structure of our ring depends on the structure of our flats, we will prove the homomorphism property in separate cases. 
	\begin{enumerate}[label = (\alph*)]
		\item Suppose that $y_{F_1} \cdot y_{F_2} = 0$, then we have $\rank_M(F_1 \vee F_2) < \rank_M(F_1) + \rank_M(F_2)$.
		\begin{enumerate}[label = (\roman*)]
			\item Additionally, suppose that $F_1, F_2 \notin \mcL(M)$. Then $\theta (y_{F_1}) = y_{F_1 \cup i}$ and $\theta(y_{F_2}) = y_{F_2 \cup i}$. To prove that $\theta_{y_{F_1}} \theta_{y_{F_2}} = 0$, it is enough to prove 
			\[
				\rank_M(F_1 \cup F_2 \cup e) < \rank_M(F_1 \cup e) + \rank_M(F_2 \cup e).
			\]
			If $e \in \clo_M(F_1 \cup F_2)$, then this result follows from $\rank_M(F_1 \vee F_2) < \rank_M(F_1) + \rank_M(F_2)$. If $e \notin \clo_M(F_1 \cup F_2)$, then there is a basis of $F_1 \cup F_2$ such that $I \cup e$ is a basis of $F_1 \cup F_2 \cup e$. We then have 
			\begin{align*}
				\rank_M(F_1 \cup F_2 \cup e) = 1 + |I| \leq 1 + |I \cap F_1| + |I \cap F_2| < \rank_M(F_1) + \rank_M(F_2).
			\end{align*}
			This proves the homomorphism property for $F_1, F_2 \notin \mcL(M)$.

			\item Now, suppose that $F_1 \notin \mcL(M)$ and $F_2 \in \mcL(M)$. Then $\rank (F_1) = \rank (F_1 \cup i)$ and $\rank (F_1 \cup F_2) < \rank (F_1) + \rank (F_2)$. We have that 
			\begin{align*}
				\rank \left((F_1 \cup e) \cup F_2\right) & = \rank (F_1 \cup F_2) \\
				& < \rank (F_1) + \rank (F_2) \\
				& = \rank (F_1 \cup e) + \rank (F_2). 
			\end{align*}
			This proves the homomorphism property for $F_1 \notin \mcL(M)$ and $F_2 \in \mcL(M)$. 

			\item Suppose that $F_1, F_2 \in \mcL(M)$. Then we automatically get the homomorphism property. 
		\end{enumerate}

		\item Now, suppose that $y_{F_1} \cdot y_{F_2} = y_{F_1 \vee F_2}$. Then, we have $\rank (F_1 \vee F_2) = \rank (F_1) + \rank (F_2)$. 
		\begin{enumerate}[label = (\roman*)]
			\item Suppose that $F_1, F_2 \notin \mcL(M)$. Then, we have that $\rank (F_1) = \rank(F_1 \cup e)$ and $\rank (F_2) = \rank (F_2 \cup e)$. This implies that $\rank (F_1 \cup F_2 \cup e) = \rank (F_1 \cup F_2)$. Thus, we have that 
			\begin{align*}
				\rank \left((F_1 \vee e) \vee (F_2 \vee e)\right) & = \rank (F_1 \cup F_2) \\
				& = \rank (F_1) + \rank (F_2) \\
				& = \rank (F_1 \vee e) + \rank (F_2 \vee e). 
			\end{align*}
			This proves the homomorphism property for $F_1, F_2 \notin \mathcal{L}(M)$. 

			\item Suppose that $F_1 \notin \mcL(M)$ and $F_2 \in \mcL(M)$. Then we have that $\rank (F_1 \cup e) = \rank (F_1)$ and $\rank (F_2 \cup e) = 1 + \rank (F_2)$. We also have $\rank (F_1 \cup F_2 \cup i) = \rank (F_1 \cup F_2)$. Thus, we have 
			\begin{align*}
				\rank((F_1 \cup e) \cup F_2 ) & = \rank (F_1 \cup F_2) \\
				& = \rank (F_1) + \rank (F_2) \\
				& = \rank (F_1 \cup e) + \rank (F_2). 
			\end{align*}
			This proves the homomorphism property ofr $F_1 \notin \mcL(M)$ and $F_2 \in \mcL(M)$. 

			\item Finally, suppose that $F_1, F_2 \in \mcL(M)$. Then the homomorphism property follows automatically. 
		\end{enumerate}
	\end{enumerate} 
	This suffices for the proof. 
\end{proof}

\subsection{Chow Ring of a Matroid} \label{sec:chow-ring-matroid}

In this section, we give a brief overview of the Chow ring associated to a matroid. We will not prove many of the claims we make, and refer the reader to \cite{huh-semi-small} for the details. Let $M = (E, \mcI)$ be a loopless matroid. According to \cite{huh-semi-small}, when $M$ is realizable over a field $k$, then the Chow ring of $M$ is isomorphic to the Chow ring of a smooth projective variety over $k$. The Chow ring of a matroid was originally introduced by Feichtner and Yuzvinsky in the paper \cite{Feichtner_2004}. It is defined as a quotient ring of the polynomial ring 
\[	
	\underbar{S}_M := \RR \left [ x_F : F \text{ is a nonempty proper flat of } M \right ].
\]
We can define the two ideals of $\underbar{S}_M$ given by 
\begin{align*}
	\underbar{I}_M & := \left \langle \sum_{i_1 \in F} x_F - \sum_{i_2 \in F} x_F : \text{ for all } i_1, i_2 \in E \right \rangle = \left \langle \sum_{i_1 \notin F} x_F - \sum_{i_2 \notin F} x_F : \text{ for all } i_1, i_2 \in E \right \rangle,  \\
	\underbar{J}_M & := \left \langle x_{F_1} x_{F_2} : \text{ $F_1$ and $F_2$ are incomparable}  \right \rangle.
\end{align*}

\begin{defn} \label{def:chow-ring-matroid}
	For a matroid $M$, we define the \textbf{Chow ring} of $M$ to be the quotient algebra 
	\[
		\underbar{CH} (M) := \frac{\underbar{S}_M}{\underbar{I}_M + \underbar{J}_M}. 
	\]
\end{defn}

The Chow ring of a matroid is a graded algebra where the grading is inherited from the grading on $\underbar{S}_M$. When $M$ has rank $d$, the top dimension of $\underbar{CH} (M)$ is $d-1$. There exists an isomorphism 
\[
	\underbar{deg} : \underbar{CH} (M)^{d-1} \longrightarrow \RR, \quad \prod_{F \in \mathcal{F}} x_F \longmapsto 1
\]
whenever $\mathcal{F}$ is any complete flag of nonempty proper flats. Similar to the graded M\"obius algebra, there exists a natural embedding $\underbar{$\theta$}_i : \underbar{CH}(M \backslash e) \to \underbar{CH}(M)$ which gives $\underbar{CH}(M)$ a $\underbar{CH}(M \backslash i)$-module structure. From \cite{krull-schmidt}, the category of graded $\underbar{CH}(M \backslash i)$ modules is a Krull-Schmidt category. Therefore, there is a unique way to decompose $\underbar{CH}(M)$ into indecomposable $\underbar{CH}(M \backslash i)$ modules. Theorem 1.2 in \cite{huh-semi-small} gives the decomposition in terms of smaller matroids. 

\begin{thm} [Theorem 1.2 in \cite{huh-semi-small}] \label{thm:chow-ring-decomp}
	If $i$ is not a coloop of $M$, there is a decomposition of $\underbar{$\CH$}(M)$ into indecomposable graded $\underbar{$\CH$}(M \backslash i)$-modules such that 
	\[
		\underbar{$\CH$}(M) = \theta_i \left ( \underbar{$\CH$} (M \backslash i) \right ) \oplus \bigoplus_{F \in \underbar{$\mathcal{S}$}_i} x_{F \cup i} \cdot \theta_i \left ( \underbar{$\CH$} (M \backslash i) \right ).
	\]
	Here, $\underbar{$\mathcal{S}$}_i$ consists of all non-empty proper flats $F$ of $E \backslash i$ such that $i$ is a coloop of $F \cup i$. 
\end{thm}

The semi-small decomposition in Theorem~\ref{thm:chow-ring-decomp} provides an avenue to inductively prove that the Chow ring satisfies the K\"ahler package. 

\subsection{Augmented Chow Ring of a Matroid}

The augmented Chow ring of a matroid is defined in \cite{huh-semi-small}. This ring is intimately related to the Chow ring of matroid. The augmented Chow ring is defined to be a quotient ring of the polynomial ring
\[
	\opS_M := \RR \left [x_F, y_i : i \in E, F \text{ a proper flat of } M \right ].
\]
Let $\opJ_M^{(1)}$ be the ideal generated by $x_{F_1} x_{F_2}$ where $F_1$ and $F_2$ are incomparable proper flats. Let $\opJ_M^{(2)}$ be the ideal generated by $y_i x_F$ where $i \in E$ and $F$ is a proper flat not containing $i$. Then, we have the two ideals which are used to define the augmented Chow ring. 
\begin{align*}
	\opI_M & := \left \langle y_i - \sum_{i \notin F} x_F : \text{ for all } i \in E \right \rangle,\\
	\opJ_M & := \opJ_M^{(1)} + \opJ_M^{(2)}.
\end{align*}

\begin{defn} \label{def:augmented-chow-ring-of-matroid}
	For a matroid $M$, we define the \textbf{augmented Chow ring} of $M$ to be the quotient algebra
	\[
		\CH(M) := \frac{\opS_M}{\opI_M + \opJ_M}.
	\]
\end{defn}
Like the Chow ring of a matroid, the augmented Chow ring of a matroid is a graded algebra with the grading inherited from the polynomial ring $\opS_M$. The top dimension of $\CH(M)$ will be the rank of the matroid $M$, and there is an isomorphism $\deg : \CH^d(M) \longrightarrow \RR$. Note that the subring generated by $\{y_i\}_{i \in E}$ is isomorphic to the graded M\"obius algebra. Hence, there is a embedding $\opH(M) \to \CH(M)$ which also gives $\CH(M)$ a $\opH(M)$-module structure. There is also a natural injection $\theta_i : \CH(M \backslash i) \to \CH(M)$ for any $i \in E$ which makes $\CH(M)$ into a $\CH(M \backslash i)$ module. The Krull-Schmidt decomposition of $\CH(M)$ into $\CH(M \backslash i)$ modules when $i$ is not a coloop is given by Theorem 1.5 in \cite{huh-semi-small}.

\begin{thm}[Theorem 1.5 in \cite{huh-semi-small}] \label{thm:augmented-chow-ring-krull-schmidt-decomp}
	If $i$ is not a coloop of $M$, then we can decompose $\CH(M)$ into indecomposable graded $\CH(M \backslash i)$-modules such that 
	\[
		\CH(M) = \theta_i (\CH(M \backslash i)) \oplus \bigoplus_{F \in \mathcal{S}_i} x_{F \cup i} \cdot \theta_i (\CH(M \backslash i)).
	\]
\end{thm}

This decomposition is similar to that of Theorem~\ref{thm:chow-ring-decomp}. This allows us to hope that similar decompositions hold for other cohomology rings associated to matroids. We discuss this point in Section~\ref{sec:future-work} when we discuss possible directions for future research. 

\section{Gorenstein Ring associated to a polynomial}

In this section, we examine a Poincar\'e algebra associated to a homogeneous polynomial. We call this ring the Gorenstein ring associated to a polynomial. As the name suggests, there is such a ring for every homogeneous polynomial $f \in \RR [x_1, \ldots, x_n]$. The case where the polynomial is the basis generating polynomial of a matroid was studied by Toshiaki Maeno and Yasuhide Numata in their paper \cite{MN-gorenstein}. In the case where the polynomial is the volume polynomial of a polytope, the ring was studied in \cite{riemann-roch-thm-for-virtual-polytopes} and \cite{Timorin_1999}. The idea of associating a Poincar\'e duality algebra to a polynomial is an old one due to Macaulay. 

\begin{defn}
	Let $f \in \RR[x_1, \ldots, x_n]$ be a homomgeneous polynoimal and let $S := \RR[\partial_1, \ldots, \partial_n]$ be the polynomial ring of differentials where $\partial_i := \partial_{x_i}$. Let $A_f^\bullet := S / \Ann_S (f)$. Then, we call $A_f^\bullet$ the \textbf{Gorenstein ring} associated to the polynomial $f$. Alternatively, we can view $\RR[x_1, \ldots, x_n]$ as a $\RR[X_1, \ldots, X_n]$ modules where the action is defined by the relation
	\[
		p(X_1, \ldots, X_n) \cdot q(x_1, \ldots, x_n) := p(\partial_1, \ldots, \partial_n) q(x_1, \ldots, x_n). 
	\]
	Then the Gorenstein ring is exactly $\RR[X_1, \ldots, X_n] / \Ann (f)$ where the annihilator is with respect to the $\RR[X_1, \ldots, X_n]$ action. 
\end{defn}

The Gorenstein ring associated to a polynomial has a natural grading with respect to the degree of the differential form. Before we prove that this actually gives $A_f^\bullet$ a graded ring structure, we first prove Lemma~\ref{homogeneous-parts}.

\begin{lem} \label{homogeneous-parts}
	Let $\xi \in \RR[\partial_1, \ldots, \partial_n]$ and $f \in \RR[x_1, \ldots, x_n]$ be a homogeneous polynomial. We can decompose $\xi = \xi_0 + \xi_1 + \ldots$ into its homogeneous parts. If $\xi (f) = 0$, then $\xi_d (f) = 0$ for all $d \geq 0$. 
\end{lem}

\begin{proof}
	Let $d = \deg (f)$. If $\xi_i (f) \neq 0$, then $i \leq d$ and $\xi_i(f)$ is a homomgeneous polynomial of degree $d-i$. Thus, $\xi (f) = \xi_0 (f) + \xi_1(f) + \ldots$ will be the homomgeneous decomposition of the polynomial $\xi(f)$. Since this is equal to $0$, all components of the decomposition are equal to zero. This suffices for the proof. 
\end{proof}

\begin{prop}
	The ring $A_f^\bullet$ is a graded $\RR$-algebra where $A_f^k$ consists of the forms of degree $k$. 
\end{prop}

\begin{proof}
	Let us define $A_f^k$ as in the statement of the lemma. Let $d = \deg (f)$ be the degree of the homomgeneous polynomial. Whenver $k > d$, the ring $A_f^k$ is clearly trivial. From Lemma~\ref{homogeneous-parts}, we have the direct sum decomposition 
	\[
		A_f^\bullet = \bigoplus_{k = 0}^d A_f^k.
	\]
	It is also clear that multiplication induces maps $A_f^r \times A_f^s \to A_f^{r+s}$ for all $r, s \geq 0$.
\end{proof}

Proposition~\ref{chow-ring-is-a-PD-algebra} states that the natural pairing in $A_f^\bullet$ equips the ring with a Poincar\'e duality algebra structure. The proposition follows from Theorem 2.1 in \cite{maeno2009lefschetz}. 
\begin{prop} \label{chow-ring-is-a-PD-algebra}
	Let $f$ be a homogeneous polynomial of degree $d$. Then, the ring $A_f^\bullet$ is a Poincar\'e-Duality algebra. That is, the ring satisfies the following two properties:
	\begin{enumerate}[label = (\alph*)]
		\item $A_f^d \simeq A_f^0 \simeq \RR$;

		\item The pairing induced by multiplication $A_f^{d-k} \times A_f^k \to A_f^d \simeq \RR$ is non-degenerate for all $0 \leq k \leq d$. 
	\end{enumerate}
\end{prop}

In Lemma~\ref{non-degeneracy-definition}, we give an characterization of non-degeneracy. From this characterization, it follows that the Hilbert polynomial of any Poincar\'e duality algebra is palindromic. In other words, $\dim A_f^k = \dim A_f^{d-k}$ for all $k$. 

\begin{lem} \label{non-degeneracy-definition}
	Let $B : V \times W \to k$ be a bilinear pairing between two finite-dimensional $k$-vector spaces $V$ and $W$. Then, any two of the following three conditions imply the third. 
		\begin{enumerate}[label = (\roman*)]
			\item The map $B_V : V \to W^*$ defined by $v \mapsto B(v, \cdot)$ has trivial kernel.
			\item The map $B_W : W \to V^*$ defined by $w \mapsto B(\cdot, w)$ has trivial kernel. 
			\item $\dim V = \dim W$. 
		\end{enumerate}
	\end{lem}

	\begin{proof}
		Condition (i) implies $\dim V \leq \dim W$ and Condition (ii) implies $\dim W \leq \dim V$. Thus (i) and (ii) both imply (iii). Now, suppose that (i) and (iii) are true. Then $B_V$ is an isomorphism between $V$ and $W^*$ (see 3.69 in \cite{axler}). Let $v_1, \ldots, v_n$ be a basis for $V$. Then $B_V(v_1), \ldots, B_V(v_n)$ is a basis of $W^*$. Let $w_1, \ldots, w_n$ be the dual basis in $W$ with respect to this basis of $W^*$. Suppose that $\sum \lambda_i w_i \in \ker B_W$. Then for all $v \in V$, we have 
		\[
			\sum_{i = 1}^n \lambda_i B_V(v)(w) = B \left ( v, \sum_{i = 1}^n \lambda_i w_i \right ) = 0. 
		\]
		By letting $v = v_1, \ldots, v_n$, we get $\lambda_i = 0$ for all $i$. 
	\end{proof}

We say a pairing between finite-dimensional vector spaces is non-degenerate whenever all three conditions in Lemma~\ref{non-degeneracy-definition} hold. As a consequence, we have Corollary~\ref{same-dimensions}.

\begin{cor} \label{same-dimensions}
	Let $f$ be a homogeneous polynomial of degree $d \geq 2$ and let $k$, $0 \leq k \leq d$, by a non-negative integer. Then $\dim_\RR A_f^k = \dim_\RR A_f^{d-k}$. 
\end{cor}

Given a homogeneous polynomial $f \in \RR[x_1, \ldots, x_n]$ of degree $d$, we can define an isomorphism $\deg_f : A_f^d \to \RR$ given by evaluation at $f$. This means that for any differential $d$-form $\xi \in A_f^d$, we have that $\deg_f (\xi) := \xi(f)$. Since $\xi$ and $f$ homogeneous of the same degree, the value of $\xi(f)$ will be a real number. Following the terminology in \cite{AHK}, we give the following definition. 

\begin{defn}
	Let $f$ be a homogeneous polynomial of degree $d$ and let $k \leq d/2$ be a non-negative integer. For an element $l \in A_f^1$, we define the following notions:
	\begin{enumerate}[label = (\alph*)]
		\item The \textbf{Lefschetz operator} on $A_f^k$ associated to $l$ is the map $L_l^k : A_f^k \to A_f^{d-k}$ defined by $\xi \mapsto l^{d-2k} \cdot \xi$. 

		\item The \textbf{Hodge-Riemann form} on $A_f^k$ associated to $l$ is the bilinear form $Q_l^k : A_f^k \times A_f^k \to \RR$ defined by $Q_l^k (\xi_1, \xi_2) = (-1)^k \deg (\xi_1 \xi_2 l^{d-2k})$.

		\item The \textbf{primitive subspace} of $A_f^k$ associated to $l$ is the subspace
		\[
			P_l^k := \{\xi \in A_f^k : l^{d-2k+1} \cdot \xi = 0\} \subseteq A_f^k.
		\]
	\end{enumerate}
\end{defn}

\begin{defn}
	Let $f$ be a homogeneous polynomial of degree $d$, let $k \leq d/2$ be a non-negative integer, and let $l \in A_f^1$ be a linear differential form. We define the following notions:
	\begin{enumerate}[label = (\alph*)]
		\item (Hard Lefschetz Property) We say $A_f$ satisfies $\HL_k$ with respect to $l$ if the Lefschetz operator $L_l^k$ is an isomorphism.

		\item (Hodge-Riemann Relations) We say $A_f$ satisfies $\HRR_k$ with respect to $l$ if the Hodge-Riemann form $Q_l^k$ is positive definite on the primitive subspace $P_l^k$. 
	\end{enumerate}
\end{defn}

We say that a homogeneous polynomial $f$ satisifes $\HL$ or $\HRR$ if the associated ring $A_f^\bullet$ satisfies $\HL$ or $\HRR$. For any $a \in \RR^n$, we can define the linear differential form $l_a := a_1 \partial_1 + \ldots + a_n \partial_n$. We say that $f$ satisfies $\HL$ or $\HRR$ with respect to $a$ if and only if it satisfies $\HL$ or $\HRR$ with respect to $l_a$. Since we have shown that $A_f^\bullet$ automatically satisfies Poincar\'e duality, we say that $f$ satisfies the K\"ahler package for a linear form $l$ if it satisfies $\HL_k$ and $\HRR_k$ with respect to $l$ for all $k \leq \frac{d}{2}$. 

\begin{prop} [Lemma 3.4 in \cite{MNY}] \label{conditions-for-HL-HRR}
	Let $f \in \RR[x_1, \ldots, x_n]$ be a homogeneous polynomial of degree $d \geq 2$ and $a \in \RR^n$. Assume that $f(a) > 0$. Then, 
	\begin{enumerate}[label = (\alph*)]
		\item $A_f$ has $\HL_1$ with respect to $l_a$ if and only $Q_{l_a}^1$ is non-degenerate. 

		\item Suppose that $A_f$ satisfies $\HL_1$. Then $A_f$ has $\HRR_1$ with respect to $l_a$ if and only if $-Q_{l_a}^1$ has signature $(+, -, \ldots, -)$. 
	\end{enumerate}
\end{prop}

\begin{proof}
	We include a proof for completeness. We first prove the statement in (a). Suppose that $A_f$ has $\HL_1$ with respect to $l_a$. We have the following commutative diagram:
	\[\begin{tikzcd}
	{A_f^1 \times A_f^1} && {A_f^1 \times A_f^{d-1}} \\
	& {\mathbb{R}}
	\arrow["{\text{id} \times L_{l_a}^1}", from=1-1, to=1-3]
	\arrow["{-Q_{l_a}^1}"', from=1-1, to=2-2]
	\arrow[from=1-3, to=2-2]
\end{tikzcd}\]
where the missing mapping is multiplication. If $A_f$ has $\HL_1$ with respect to $l_a$, then the top map between $A_f^1 \times A_f^1 \to A_f^1 \times A_f^{d-1}$ is a bijection. Thus the non-degeneracy of $Q_{l_a}^1$ follows from the non-degeneracy of the multiplication pairing as stated in Proposition~\ref{chow-ring-is-a-PD-algebra}. Now, suppose that $Q_{l_a}^1$ is non-degenerate. Then, the map $B : A_f^1 \to (A_f^1)^*$ defined by $\xi \mapsto -Q_{l_a}^1(\xi, \cdot)$ is given by $m(L_{l_a}^1 \xi, \cdot)$ where $m : A_f^1 \to A_f^{d-1} \to \RR$ is the multiplication map. This is the composition of $A_f^1 \to A_f^{d-1} \to (A_f^1)_*$ where the first map is $L_{l_a}^1$ and the second map is injective from the non-degeneracy of the multiplication map. This proves that $L_{l_a}^1$ is injective. From Corollary~\ref{same-dimensions}, the map $L_{l_a}^1$ is an isomorphism. This suffices for the proof of (a). \\

To prove (b), consider the commutative diagram
\[\begin{tikzcd}
	{\mathbb{R}} & {A^0_f} & {A_f^1} & {A_f^{d-1}} & {A_f^d} & {\mathbb{R}}
	\arrow["{\times l_a}", from=1-2, to=1-3]
	\arrow["{L_{l_a}^1}", from=1-3, to=1-4]
	\arrow["{\times l_a}", from=1-4, to=1-5]
	\arrow["\simeq", from=1-1, to=1-2]
	\arrow["\simeq", from=1-5, to=1-6]
	\arrow["L_{l_a}^0"{description}, bend right= 18, from=1-2, to=1-5]
\end{tikzcd}\]
Note that $L_{l_a}^0$ is an isomorphism because 
\[
	\deg L_{l_a}^0 (1) = l_a^d (f)  = d! f(a) \neq 0.
\]
Thus, we have $A_f^1 = \RR l_a \oplus P_l^1$ where the direct sum is orthogonal over the Hodge-Riemann form by definition of the primitive subspace. Now, note that 
\[
	-Q_{l_a}^1(l_a, l_a) = l_a^d (f) = -d! f(a) > 0. 
\]
Thus, the signature of $-Q_{l_a}^1$ is $(+, -, \ldots, -)$ if and only if $Q_{l_a}^1$ is positive definite over the primitive subspace if and only if $A_f$ satisfies $\HRR_1$ with respect to $l_a$. 
\end{proof}

Recall from our definition of Lorentzian polynomials, we know that non-zero Lorentzian polynomials are log-concave at any point in $a \in \RR_{> 0}^n$. When $f(a) > 0$, the polynomial $f$ is log-concave at $a$ if and only if its Hessian has exactly one positive eigenvalue. Hence, for any element in the non-negative orthant, we know that the Hessians of Lorentzian polynomials have at most one positive eigenvalue. This fact and Sylvester's Law of Intertia gives a proof of Lemma~\ref{lorentzian-HL-iff-HRR}.

\begin{lem} \label{lorentzian-HL-iff-HRR}
	If $f \in \RR [x_1, \ldots, x_n]$ is Lorentzian, then for any $a \in \RR_{\geq 0}^n$ with $f(a) > 0$, $A_f^1$ has $\HL_1$ with respect to $l_a$ if and only if $f$ has the $\HRR_1$ with respect to $l_a$. 
\end{lem}

\begin{proof}
	See Lemma 3.5 in \cite{MNY}. 
\end{proof}

\section{Local Hodge-Riemann Relations}

This section illustrates an example of a general inductive technique to prove Hodge-Riemann relations. We define a local version of the Hodge-Riemann relations. If this local version is satisfied, then this will imply that the Hard Lefschetz property will be satisfied. In some situations, this is enough to imply that the original Hodge-Riemann relations are satisfied. We give an example of this inductive process in Lemma~\ref{local-hrr-mechanism}. We take our definition of the local Hodge-Riemann relations from \cite{MNY}. 

\begin{defn}
	A homogeneous polynomial $f \in \RR[x_1, \ldots, x_n]$ of degree $d \geq 2k+1$ satisfies the \textbf{local} $\HRR_k$ with respect to a form $l \in A_f^1$ if for all $i \in [n]$, either $\partial_i f = 0$ or $\partial_i f$ satisfies $\HRR_k$ with respect to $l$. 
\end{defn}

\begin{lem} [Lemma 3.7 in \cite{MNY}] \label{local-hrr-mechanism}
	Let $f \in \RR_{\geq 0}[x_1, \ldots, x_n]$ be a homogeneous polynomial of degree $d$ and $k$ a positive integer with $d \geq 2k+1$, and $a = (a_1, \ldots, a_n) \in \RR^n$. Suppose that $f$ has the local $\HRR_k$ with respect to $l_a$. 
	\begin{enumerate}[label = (\roman*)]
		\item If $a \in \RR_{> 0}^n$, then $A_f$ has the $\HL_k$ with respect to $l_a$. 
		\item If $a_1 = 0, a_2, \ldots, a_n > 0$ and $\{\xi \in A_f^k : \partial_i \xi = 0 \text{ for $i = 2, \ldots, n$}\} = \{0\}$, then $A_f$ has the $\HL_k$ with respect to $l_a$. 
	\end{enumerate}
\end{lem}

With Lemma~\ref{local-hrr-mechanism}, we can inductively prove that all Lorentzian polynomials satisfy $\HRR_1$ with respect to $l_a$ for all $a \in \RR_{> 0}^n$. We can prove this directly for all small Lorentzian polynomials. For an arbitrary Lorentzian polynomial, we know that all of its partial derivatives are Lorentzian. Thus, by induction, all of the partial derivatives satisfy $\HRR_1$. This means that $f$ satisfies the local $\HRR_1$. From lemma~\ref{local-hrr-mechanism}, we know that $f$ satisfies $\HL_1$. But we have shown that $\HL_1$ and $\HRR_1$ are equivalent for Lorentzian polynomials. This gives a proof of Theorem~\ref{lorentzian-satisfies-HRR} using the inductive procedure that we alluded to. 

\begin{thm} [Theorem 3.8 in \cite{MNY}] \label{lorentzian-satisfies-HRR}
	If $f \in \RR[x_1, \ldots, x_n]$ is Lorentzian, then $f$ has $\HRR_1$ with respect to $l_a$ for any $a \in \RR_{> 0}^n$. 
\end{thm}

The next result in Corollary~\ref{partial-independent-implies-hessian} gives a condition for general homogeneous polynomials to satisfy $\HRR_1$ in terms of the signature of its Hessian. This result will hold for any $a \in \RR^n$ satisfying the conditions in the statement.

\begin{cor} \label{partial-independent-implies-hessian}
	Let $f$ be a homogeneous polynomial of degree $d \geq 2$. If $\partial_1 f, \ldots, \partial_n f$ are linearly independent in $\RR[x_1, \ldots, x_n]$ and $f(a) > 0$ for some $a \in \RR^n$, then $A_f$ satisfies $\HRR_1$ with respect to $l_a$ if and only if $\Hess_f|_{x = a}$ has signature $(+, -, \ldots, -)$. 
\end{cor}

\begin{proof}
	Since $\partial_1 f_1, \ldots, \partial_n f_n$ are linearly independent, the partials $\partial_i$ form a basis for $A_f^1$. Thus, the signature of $Q_{l_a}^1$ is actually the signature of matrix of $Q_{l_a}^1$ with respect to the set $\{\partial_1, \ldots, \partial_n\}.$ We have 
	\[
		-Q_{l_a}^1(\partial_i, \partial_j) = \partial_i \partial_j l_a^{d-2} f =  l_a^{d-2} \partial_i \partial_j f = (d-2)! \partial_i \partial_j f(a).
	\]
	This proves that the signature of $-Q^1_{l_a}$ is the same as the signature of $\Hess_f |_{x = a}$.  We are done from Lemma~\ref{conditions-for-HL-HRR}(b).
\end{proof}

\section{The Gorenstein Ring associated to the Basis Generating Polynomial of a Matroid} \label{sec:gorenstein-ring-of-matroid}

In this section, we specialize the Gorenstein ring associated to a polynomial to the Gorenstein ring associated to the basis generating polynomial of a matroid. This graded $\RR$-algebra is another cohomology ring associated to a matroid. Thus, it is interesting study the Hard Lefschetz property and Hodge-Riemann relations on this algebra. The Gorenstein ring associated to the basis generating polynomial of a matroid is intimately related to graded M\"obius algebra. In particular, we can realize the Gorenstein ring associated to the basis generating polynomial as a quotient ring of $\opH(M)$ by quotienting out the kernel of the Poincar\'e pairing.

\begin{defn} \label{defn:basis-cohomology}
	Let $M = (E, \mcI)$ be a matroid. We let $\A(M)$ be the Gorenstein ring associated to the basis generating polynomial of $M$. Explicitly, we define the ring $\opS_M := \RR [X_e : e \in E]$ and the ideal $\Ann_M = \Ann_{\opS_M}(f_M)$ where $f_M$ is the basis generating polynomial of $M$. Then, the ring $\A(M)$ is equal to $\A(M) = \opS_M / \Ann_M$.
\end{defn}

For a matroid $M$, we will prove that the ring $\A(M)$ will depend only on its simplification. This is a reasonable claim because whenever we have two elements $e, f \in E$ which are parallel, the differential $\partial_e - \partial_f$ will be in the annihilator of $f_M$. Indeed, the elements $e$ and $f$ are interchangeable in any independent set. Moreover, if $e \in E(M)$ is a loop then $\partial_e$ is clearly in the annihilator. This claim follows directly from the fact that $\A(M)$ is a quotient ring of the graded M\"obius algebra. Since the graded M\"obius algebra only depends on the lattice of flats which only depends on the simplification of the matroid, it follows that $\A(M)$ should only depend on the simplification of the matroid. However, we want to prove a more delicate isomorphism which tells us the image of linear forms under this isomorphism. This is the content of Theorem~\ref{only-simplification-matters}. Before proving this result, we first describe some common elements in the ideal $\Ann_M$. 

\begin{prop}
	Let $M = (E, \mcI)$ be a matroid. For any subsets $S, T \subseteq E$, we write $S \sim T$ if and only if $\overline{S} = \overline{T}$. Let $\Lambda_M^{(1)}, \Lambda_M^{(2)}$, $\Lambda_M^{(3)}$ be three subsets of $\opS_M$ given by 
	\begin{align*}
		\Lambda_M^{(1)} & := \left \{X_e^2 : e \in E(M) \right \}, \\
		\Lambda_M^{(2)} & := \left \{X^S : S \notin \mcI(M) \right \}, \\
		\Lambda_M^{(3)} & := \left \{X^S - X^T : S, T \in \mcI(M) \text{ such that } S \sim T \right \}.
	\end{align*}
	Let $\Lambda_M = \Lambda_M^{(1)} \cup \Lambda_M^{(2)} \cup \Lambda_M^{(3)}$. Then $\Lambda_M \subseteq \Ann_M$. 
\end{prop}

\begin{proof}
	See Proposition 3.1 in \cite{MN-gorenstein}
\end{proof}

In general, it is not true that the elements in $\Lambda_M$ generate the annihilator $\Ann_M$. In fact, if we let $I(\Lambda_M)$ be the ideal generated by $\Lambda_M$, then $\opS_M / I(\Lambda_M)$ is exactly the graded M\"obius algebra $\operatorname{H}(M)$. The graded M\"obius algebra is in general not even a Poincar\'e duality algebra, hence it cannot be $\A(M)$. We can find an explicit counterexample in \cite{MN-gorenstein}. Consider the matrix given by 
\[
	A = \begin{bmatrix}
		1 & 0 & 0 & 1 & 0 \\
		0 & 1 & 0 & 1 & 1 \\
		0 & 0 & 1 & 0 & 1
	\end{bmatrix}.
\]
The basis of the linear matroid generated by $A$ are $\{123, 125, 134, 135, 145\}$. From Example 3.5 in \cite{MN-gorenstein}, we have that 
\[
	\Ann_M =  I(\Lambda_M \cup (X_1X_3 + X_4X_5 - X_1X_5 - X_3X_4)).
\]
Let $M = (E, \mathcal{I})$ be a matroid and $\widetilde{M}$ be its simplification. We can define the maps $\phi : \opS_M \to \opS_{\widetilde{M}}$ and $\psi : \opS_{\widetilde{M}} \to \opS_M$ by 
\begin{align*}
	\phi (\partial_{x_e}) := \partial_{\overline{x_e}}, \quad \text{and} \quad \psi (\partial_{\overline{x}}) := \frac{1}{|\overline{x}|} \sum_{e \in \overline{x}} \partial_{x_e}.
\end{align*}
and then extending to the whole polynomial ring using the universal property of polynomial rings.

\begin{thm} \label{only-simplification-matters}
	The maps $\phi : S_M \to S_{\widetilde{M}}$ and $\psi : S_{\widetilde{M}} \to S_M$ induce isomorphisms between $A(M)$ and $A(\widetilde{M})$. 
\end{thm}

\begin{proof}
	We first prove that $\phi$ and $\psi$ induce homomorphisms between the rings $\A(M)$ and $\A(\widetilde{M})$. To show that $\phi$ induces a homomorphism, consider the diagram in Equation~\ref{extension-1}.
	\begin{equation} \label{extension-1}
			\begin{tikzcd}
				{S_M} && {S_{\widetilde{M}}} && {A(\widetilde{M})} \\
				\\
				&& {A(M)}
				\arrow["\phi", from=1-1, to=1-3]
				\arrow["{\pi_{\widetilde{M}}}", from=1-3, to=1-5]
				\arrow["{\pi_M}"', from=1-1, to=3-3]
				\arrow["{\exists ! \Phi}"', dotted, from=3-3, to=1-5]
			\end{tikzcd}
	\end{equation}

Let $\xi \in S_M$ be an element satisfying $\xi (f_M) = 0$. We will prove that $\phi(\xi) (f_{\widetilde{M}}) = 0$. In other words, we want to prove that $\Ann_M \subseteq \ker \pi_{\widetilde{M}} \circ \phi$. From Proposition~\ref{homogeneous-parts}, it suffices to consider the case where $\xi$ is homogeneous. Let $e_1, \ldots, e_s$ be representatives of all parallel classes. Then, we have that $M = \overline{e_1} \sqcup \ldots \sqcup \overline{e_s} \cup E_0$ where $E_0$ denotes the loops in $M$. In terms of the basis generating polynomials, we have 
\begin{align*}
	f_{\widetilde{M}}(x_{\overline{e_1}}, \ldots, x_{\overline{e_s}}) & = \sum_{\substack{1 \leq i_1 < \ldots < i_r \leq s \\ \{\overline{e_{i_1}}, \ldots, \overline{e_{i_r}}\} \in \mathcal{B}(\widetilde{M})}} x_{\overline{e_{i_1}}} \ldots x_{\overline{e_{i_r}}} \\ 
	f_M(x_1, \ldots, x_n) & = f_{\widetilde{M}} \left ( y_1, \ldots, y_s \right )
\end{align*}
where for $1 \leq i \leq s$, we define $y_i := \sum_{e \in \overline{e_i}} x_e$. Since $\xi$ is homogeneous of degree $k$, we can write it in the form $\xi = \sum_{\substack{\alpha \subseteq [n] \\ |\alpha| = k}} c_\alpha \partial^\alpha$. Then, we have 
\begin{align} \label{karen}
	\xi (f_M) & = \sum_{\beta \in \mathcal{B}} \xi (x^\beta) = \sum_{\beta \in \mathcal{B}(M)} \sum_{\substack{\alpha \subseteq [n] \\ |\alpha| = k}} c_\alpha \partial^\alpha x^\beta = \sum_{\gamma \in \mathcal{I}_{r-k}(M)} \left ( \sum_{ \substack{\alpha \in \mathcal{I}_k \\ \alpha \cup \gamma \in \mathcal{I}_r(M)}} c_\alpha  \right ) x^\gamma. 
\end{align}

Since $\xi(f_M) = 0$, we know that all of the coefficients on the right hand side of Equation~\ref{karen} are equal to $0$. Thus, we have that
\[
	\sum_{\substack{\alpha \in I_k \\ \alpha \cup \gamma \in \mathcal{I}_r(M)}} c_\alpha = 0, \text{ for all $\gamma \in \mathcal{I}_{r-k}(M)$}.
\]
On the other hand, we have 
\[
	\phi (\xi) = \sum_{\substack{\alpha \subseteq [n] \\ |\alpha| = k}} c_\alpha \prod_{e \in \alpha} \partial_{\overline{e}} = \sum_{\beta \in \mathcal{I}_{k}(\widetilde{M})} \left ( \sum_{\alpha \in \text{fiber}(\beta)} c_\alpha \right ) \partial^\beta. 
\]
We can let this differential act on $f_{\widetilde{M}}$ to get the expression
\begin{align} \label{karen-math}
	\phi(\xi) (f_{\widetilde{M}}) & = \sum_{\gamma \in \mathcal{I}_{r-k}(\widetilde{M})} \left ( \sum_{\substack{\beta \in \mcI_k(\widetilde{M}) \\ \beta \cup \gamma \in \mathcal{B}(\widetilde{M})}} \sum_{\alpha \in \text{fiber}(\beta)} c_\alpha \right ) x^\gamma = \sum_{\gamma \in \mathcal{I}_{r-k}(\widetilde{M})} \left ( \sum_{\substack{\alpha \in \mcI_k(M) \\ \alpha \cup \gamma_0 \in \mathcal{B}(M)}} c_\alpha \right ) x^\gamma = 0.
\end{align}
In Equation~\ref{karen-math}, the independent set $\gamma_0 \in \text{fiber}(\gamma)$ is an arbitrary element in the fiber of $\gamma$. This proves that $\Ann_M \subseteq \ker \pi_{\widetilde{M}} \circ \phi$. Thus, there is a unique ring homomorphism $\Phi : A(M) \to A(\widetilde{M})$ which makes Equation~\ref{extension-1} commute. \\

To prove that $\psi$ induces a ring homomorphism, consider the diagram in Equation~\ref{extension-2}. 
\begin{equation} \label{extension-2}
\begin{tikzcd}
	{S_{\widetilde{M}}} && {S_M} && {A(M)} \\
	\\
	&& {A(\widetilde{M})}
	\arrow["{\pi_{\widetilde{M}}}"', from=1-1, to=3-3]
	\arrow["\psi", from=1-1, to=1-3]
	\arrow["{\pi_M}"', from=1-3, to=1-5]
	\arrow["{\exists! \Psi}"', dotted, from=3-3, to=1-5]
\end{tikzcd}
\end{equation}

Consider a differential $\xi \in S_{\widetilde{M}}$ satisfying $\xi(f_{\widetilde{M}}) = 0$. We can write this as $\xi = \sum_{\alpha \in \mathcal{I}_k (\widetilde{M})} c_\alpha \partial^\alpha$ for some real constants $c_\alpha$. Then, its image under $\psi$ is equal to
\[
	\psi(\xi) = \sum_{\alpha \in \mathcal{I}_k (\widetilde{M})} \frac{c_\alpha}{\prod_{e \in \alpha} |e|}\sum_{\beta \in \text{fiber}(\alpha)} \partial^\beta.
\]
Fix a $\alpha \in \mathcal{I}_k(\widetilde{M})$ and a $\beta \in \text{fiber}(\alpha)$. Since $\partial y_i / \partial x_e = \1_{e \in \overline{e_i}}$, we have  
\begin{align*}
	\partial^\beta f_M (x_1, \ldots, x_n) & = \partial^\beta f_{\widetilde{M}} (y_1, \ldots, y_s) = \partial^\alpha f_{\widetilde{M}} (x_1, \ldots, x_s) |_{x_1 = y_1, \ldots, x_s = y_s} = 0.
\end{align*}
Thus, we have that $\psi (\xi)(f_M) = 0$ and $\Ann_{\widetilde{M}} \subseteq \ker \pi_M \circ \psi$. This proves that there is a unique ring homomrphism $\Psi : A(\widetilde{M}) \to A(M)$ which causes the diagram in Equation~\ref{extension-2} to commute. Since the maps $\Psi$ and $\Phi$ are inverses of each other, they are both isomorphisms. This suffices for the proof.
\end{proof}

From Theorem~\ref{only-simplification-matters}, we get Corollary~\ref{complex-isomorphism} and Corollary~\ref{simplification-of-HRR} immediately. 

\begin{cor} \label{complex-isomorphism}
	Let $M = (E, \mathcal{I})$ be a matroid. For any $a \in \RR^E$, we can define the linear form $l_a := \sum_{e \in E} a_e \cdot \partial_{x_e} \in A^1(M)$. Let $\widetilde{l_a} := \Phi(l_a) = \sum_{e \in E} a_e \cdot \partial_{x_{\overline{e}}} \in A^1(\widetilde{M})$. Then, the following diagram commutes:
	\[\begin{tikzcd}
	\ldots & {A^{i-1}(M)} & {A^i(M)} & {A^{i+1}(M)} & \ldots \\
	\ldots & {A^{i-1}(\widetilde{M})} & {A^i(\widetilde{M})} & {A^{i+1}(\widetilde{M})} & \ldots
	\arrow["{\times l_a}", from=1-1, to=1-2]
	\arrow["{\times \widetilde{l_a}}", from=2-1, to=2-2]
	\arrow["\Phi"', from=1-2, to=2-2]
	\arrow["{\times l_a}", from=1-2, to=1-3]
	\arrow["{\times \widetilde{l_a}}", from=2-2, to=2-3]
	\arrow["{\times l_a}", from=1-3, to=1-4]
	\arrow["{\times \widetilde{l_a}}", from=2-3, to=2-4]
	\arrow["\Phi"', from=1-3, to=2-3]
	\arrow["\Phi"', from=1-4, to=2-4]
	\arrow["{\times l_a}", from=1-4, to=1-5]
	\arrow["{\times \widetilde{l_a}}", from=2-4, to=2-5]
\end{tikzcd}\]
 \end{cor}

\begin{cor} \label{simplification-of-HRR}
	Let $M = (E, \mathcal{I})$ be a matroid. Then $A(M)$ satisfies $\HRR_k$ with respect to $l$ if and only if $A(\widetilde{M})$ satisfies $\HRR_k$ with respect to $\Phi(l)$. 
\end{cor}

From Theorem~\ref{only-simplification-matters}, Corollary~\ref{complex-isomorphism}, and Corollary~\ref{simplification-of-HRR}, when we study Hodge-theoretic properties of $\A(M)$ it is enough to assume that $\widetilde{M}$ is simple. We may not always want to do this, but the translation from a matroid to its simplification is useful nonetheless. Not only do we know that the two rings are isomorphic, but we know how multiplication by $1$-forms translate from one ring to the other. This means that we can relate the Hodge theoretic properties of the two rings to each other. We also understand some properties of $\A(M)$ better when we know that $M$ is simple. For example, the vector space structure of $\A^1(M)$ is well-understood for simple $M$. Satoshi Murai, Takahiro Nagaoka, and Akiko Yazawa prove in \cite{MNY} that if $M$ is a simple matroid on $[n]$, then $\dim A^1(M) = n$. In other words, the spanning set $\partial_1, \ldots, \partial_n$ is a basis of $\A^1(M)$. We write this result in Lemma~\ref{simple-partial-independent}. 

\begin{lem} \label{simple-partial-independent}
	If $M = ([n], \mcI)$ is simple, then $\partial_1, \ldots, \partial_n$ is a basis of $A^1(M)$.
\end{lem}

\begin{proof}
	See Theorem 2.5 in \cite{MNY}
\end{proof}	

From Corollary~\ref{partial-independent-implies-hessian}, this implies that when $M$ is simple, the Hessian of the basis generating polynomial $f_M$ has signature $(+,-, \ldots, -)$ at every point in the positive orthant. Thus, the basis generating polynomial for a simple matroid is strictly log-concave on the positive orthant, which was one of the main results in \cite{MNY}. In the same paper, they formulate Conjecture~\ref{conj:mny}. One of my goals was to work towards a proof of Conjecture~\ref{conj:my-conjecture}. 

\begin{conj} \label{conj:mny}
	Let $M = (E, \mcI)$ be a matroid of rank $d$. The ring $\A(M)$ satisfies $\HL_k$ for some $a \in \RR^E_{> 0}$ for all $k \leq \frac{d}{2}$.  
\end{conj}

\begin{conj} \label{conj:my-conjecture}
	Let $M = (E, \mcI)$ be a matroid of rank $d$. The ring $\A(M)$ satisfies the K\"ahler package for all $a \in \RR^E_{> 0}$.
\end{conj}

Unfortunately, Conjecture~\ref{conj:my-conjecture} is false. The counterexample was given to us through private communication between June Huh and Connor Simpson. We discuss the counterexample in more detail in Section~\ref{sec:future-work}. For the rest of the thesis, we analyze when $\A(M)$ satisfies $\HRR_1$ on the facets of the positive orthant. We will resolve this question completely in Section~\ref{sec:HRR-on-facets}. Before moving on to this question, we first give necessary and sufficient conditions on $\A(M)$ for $\HRR_1$ with respect to $a \in \RR^n$ in the case where $f(a) > 0$ and $M$ is simple.

\begin{cor} \label{simple-matroids-HRR-condition}
	Let $M = (E, \mcI)$ be a simple matroid of rank $r \geq 2$. If $a \in \RR_{\geq 0}^E$ satisfies $f_M(a) > 0$, then $A(M)$ satisfies $\HRR_1$ with respect to $l_a$ if and only if $\Hess_{f_M} |_{x = a}$ is non-singular.
\end{cor} 

\begin{proof}
	This follows immediately from Lemma~\ref{lorentzian-HL-iff-HRR}, Corollary~\ref{partial-independent-implies-hessian}, and Lemma~\ref{simple-partial-independent}.
\end{proof}

\begin{lem} \label{block-matrix-determinant-calculation}
		When $D$ is invertible and $A$ is a square matrix, we have
		\[
			\det \begin{bmatrix} A & B \\ C & D \end{bmatrix} = \det (A - B D^{-1} C) \det (D).
		\]
	\end{lem}
	\begin{proof}
		See Section 5 in \cite{block-matrices}.
	\end{proof}

\begin{thm} \label{thm:simple-inverse-hessian-property}
	Let $M = (E, \mcI)$ be a simple matroid of rank $r \geq 2$. If $a \in \RR_{\geq 0}^E$ satisfies $f_M(a) > 0$ and $a_e = 0$ for some $e \in E$ which is not a co-loop, then $A(M)$ satisfies $\HRR_1$ with respect to $l_a$ if and only if 
	\[
		\left ( \nabla f_{M / e}^{\mathsf{T}} \cdot \Hess^{-1}_{f_{M \backslash e}} \cdot \nabla f_{M / e} \right ) |_{x = a} \neq 0.
	\]
\end{thm}

\begin{proof}
	Without loss of generality, we can assume that $E(M) = [n]$ and $e = n$. In particular, this means that $a = (a_1, \ldots, a_{n-1}, 0) \in \RR^n$ with $a_i \geq 0$ for all $i$. From Corollary~\ref{simple-matroids-HRR-condition}, $\HRR_1$ is satisfied if and only if the Hessian is non-singular. To compute the Hessian at $x = a$, note that because $n$ is a co-loop, we can write the basis generating polynomial as $f_M = x_n f_{M / n} + f_{M \backslash n}$. Using this equation, we see that the Hessian of $f_M$ at $(a_1, \ldots, a_{n-1}, 0)$ is equal to 
	\[
		\Hess_{f_M} = \begin{bmatrix}
			\Hess_{f_{M \backslash n}} & \nabla f_{M/n} \\
			(\nabla f_{M/n})^{\mathsf{T}} & 0
		\end{bmatrix}.
	\]
	Since $M$ is simple, we know that $M \backslash n$ is simple. Thus, the matrix $\Hess_{f_{M \backslash n}}$ is invertible as it has the same signature as the Hodge-Riemann form. From the Lemma~\ref{block-matrix-determinant-calculation}, we have that
	\[
		\det \Hess_{f_M} = \left ( \nabla f_{M / n}^{\mathsf{T}} \cdot \Hess^{-1}_{f_{M \backslash n}} \cdot \nabla f_{M / n} \right ).
	\]
	This suffices for the proof. 
\end{proof}

\section{Hodge-Riemann Relations on the Facets of the Positive Orthant} \label{sec:HRR-on-facets}

In this section, we will prove necessary and sufficient conditions for $\A(M)$ to satisfy $\HRR_1$ on the facets of the positive orthant. From Lemma~\ref{lorentzian-satisfies-HRR}, we already know that $A(M)$ satisfies $\HRR_1$ on $\RR_{> 0}^n$. For a matroid $M = (E, \mathcal{I})$, we can decompose the boundary set of $\RR^{E}_{\geq 0}$ as
\[	
	\bd \RR_{\geq 0}^E = \bigcup_{e \in E} H_e
\]
where $H_e := \{x \in \RR_{\geq 0}^E : x_e = 0 \}$. Before stating our main result in this section, we first prove a few technical lemmas that are necessary for us to apply the inductive procedure modeled in Lemma~\ref{local-hrr-mechanism}. 

\begin{lem} \label{contracting-stays-not-a-coloop}
	Let $M = (E, \mathcal{I})$ be a matroid satisfying $\rank (M) \geq 2$. If $e \in M$ is not a coloop of $M$, then $e$ will not be a co-loop of $M / i$ for any $i \in E \backslash e$. 
\end{lem}

\begin{proof}
	If $e$ or $i$ is a loop, then the statement is vacuously true. Now, suppose that $e$ and $i$ are both not loops. Suppose for the sake of contradiction that $e$ is a coloop of $M/i$. This implies that any basis of $M$ containing $i$ must contain $e$. But, since $e$ is not a coloop of $M$, the matroid $M \backslash e$ has the same rank of $M$. Moreover, since $i$ is independent in $M \backslash e$, there is a basis of $M \backslash e$ which contains $i$. But this is automatically a basis of $M$ that contains $i$ but doesn't contain $e$. This is a contradiction, and suffices for the proof. 
\end{proof}

\begin{thm}[Degree 1 Socles] \label{socle-socle-socle}
	Let $M = (E, \mcI)$ be a matroid satisfying $\rank (M) \geq 3$. Let $S \subseteq E(M)$ be a subset with $\rank (S) \leq \rank (M)-2$. Then 
	\[
		\{\xi \in A_f^1 : \xi (\partial_i f) = 0 \text{ for } i \in E \backslash S\} = \{0\}.
	\]
\end{thm}

\begin{proof}
	Let $r = \rank (M) \geq 3$. We want to prove that if a linear form $\xi = \sum_{e \in E} c_e \cdot \partial_e$ satisfies $\xi (\partial_e f_M) = 0$ for all $e \in E \backslash S$, then we have $\xi (f_M) = 0$. For $i \in E \backslash S$, we have 
	\begin{equation} \label{socle-calculation-1}
		0 = \xi (\partial_i f_M) = \sum_{e \in E} c_e \partial_e \partial_i f_M = \sum_{e \in E} c_e \sum_{\substack{\alpha \in \mcI_{r-2}(M) \\ \alpha \cup \{e,i\} \in \mcI_r (M)}} x^\alpha = \sum_{\alpha \in \mcI_{r-2}(M)} \left ( \sum_{\substack{e \in E \\ \alpha \cup \{i, e\} \in \mcI_r(M)}} c_e \right ) x^\alpha. 
	\end{equation}
	By setting all of the coefficients of the right hand side of Equation~\ref{socle-calculation-1}, we have that 
	\[
		\sum_{\substack{e \in E \\ \alpha \cup \{i, e\} \in \mcI_r (M)}} c_e = 0 \quad \text{for all $\alpha \in \mcI_{r-2}(M)$ and $i \in E \backslash S$}.
	\]
	For any $\beta \in \mcI_{r-1}(M)$, we know that $\beta \not \subseteq S$ since $\rank (\beta) = r-1 > \rank (S)$. There exists some $i \in \beta \backslash S$. Thus, we can write $\beta = \alpha \cup \{i\}$ where $\alpha \in \mcI_{r-2}(M)$ and $i \in E \backslash S$. This proves that for any $\beta \in \mcI_{r-1}(M)$, we have
	\[
		\sum_{\substack{e \in E \\ \beta \cup \{e\} \in \mcI_{r}}} c_e = \sum_{\substack{e \in E \\ \alpha \cup \{i, e\} \in \mcI_{r}}} c_e = 0.
	\]
	Finally, we have that 
	\[
		\xi(f_M) = \sum_{e \in E} c_e \partial_e f_M = \sum_{e \in E} c_e \sum_{\substack{\beta \in \mcI_{r-1} \\ \beta \cup \{e\} \in \mcI_r}} x^\beta = \sum_{\beta \in \mcI_{r-1}} \left ( \sum_{\substack{e \in E \\ \beta \cup \{e\} \in \mcI_r}} c_e \right ) x^\beta = 0.
	\]
	This suffices for the proof. 
\end{proof}

\begin{thm}[Higher Degree Socles] \label{higher-degree-socles}
	Let $M = (E, \mcI)$ be a matroid. Let $S \subseteq E$ be a subset such that $\rank(S) \leq \rank(M) - k - 1$. Then, 
	\[
		\{\xi \in A^k(M) : \xi (\partial_e f_M) = 0 \text{ for all } e \in E \backslash S\} = \{0\}.
	\]
\end{thm}

\begin{proof}
	Any $\xi \in A^k(M)$ can be written as $\xi = \sum_{\alpha \in \mcI_k} c_\alpha \partial^\alpha$. For any $e \in E \backslash S$, we have 
	\begin{align*}
		0 = \xi \partial_i f_M = \sum_{\alpha \in \mcI_k} c_\alpha \partial_e \partial^\alpha f_M = \sum_{\alpha \in \mcI_k} c_\alpha \sum_{\substack{\gamma \in \mcI_{r-k-1} \\ \gamma \cup \alpha \cup \{e\} \in \mcI_r}} x^\gamma = \sum_{\gamma \in \mcI_{r-k-1}} \left ( \sum_{\substack{\alpha \in \mcI_k \\ \gamma \cup \alpha \cup \{e\} \in \mcI_r}} c_\alpha \right ) x^\gamma.
	\end{align*}
	This implies that for all $\gamma \in \mcI_{r-k-1}$ and $e \in E \backslash S$, we have 
	\[
		\sum_{\substack{\alpha \in \mcI_k \\ \gamma \cup \alpha \cup \{e\} \in \mcI_r}} c_\alpha = 0 \quad \text{for all $\gamma \in \mcI_{r-k-1}$ and $e \in E \backslash S$.}
	\]
	Let $\beta \in \mcI_{r-k}$ be an arbitrary independent set of rank $r-k$. Note that $\rank (\beta) = r-k > \rank (S)$. Thus, we cannot have $\beta \subseteq S$. This implies that we can find $e \in \beta \backslash S$ such that $\beta = \alpha \cup \{e\}$ for $e \in E \backslash S$. Thus, 
	\[
		\sum_{\substack{\alpha \in \mcI_k \\ \beta \cup \alpha \in \mcI_r}} c_\alpha = 0 \quad \text{for all $\beta \in \mcI_{r-k}$}.
	\]
	Then, we have 
	\[
		\xi (f_M) = \sum_{\alpha \in \mcI_k} c_\alpha \partial^\alpha f_M = \sum_{\alpha \in \mcI_k} c_\alpha \sum_{\substack{\beta \in \mcI_{r-k} \\ \beta \cup \alpha \in \mcI_r}} x^\beta = \sum_{\beta \in \mcI_{r-k}} \left ( \sum_{\substack{\alpha \in \mcI_k \\ \beta \cup \alpha \in \mcI_r}} c_\alpha \right ) x^\beta = 0,
	\]
	which suffices for the proof. 
\end{proof}

We are now ready to state our main theorem about $\HRR_1$ on the facets of $\RR_{\geq 0}^E$. We study the properties of $\A(M)$ on the relative interiors of the facets $H_e$ whenever $e$ is not a coloop. 
\begin{thm} \label{main-theorem}
	Let $M = (E, \mcI)$ be a matroid which satisfies $\rank (M) \geq 2$. For any $e \in E(M)$, the basis generating polynomial $f_M$ satisfies $\HRR_1$ on $\relint (H_e)$ if and only if $e$ is not a co-loop.
\end{thm}

\begin{proof}
	Without loss of generality, let $M$ be a matroid on the set $[n]$ and let $e = 1$. Then, we want to prove that whenever $a_2, \ldots, a_n > 0$, the ring $A(M)$ satisfies $\HRR_1$ on $a = (0, a_2, \ldots, a_n)$ if and only if $e$ is not a co-loop. We first prove that if $1$ is not a co-loop, then $A(M)$ satisfies $\HRR_1$. To prove this, we induct on the rank of $M$. For the base case $\rank (M) = 2$, Corollary~\ref{simplification-of-HRR} implies that it suffices to prove that $A(\widetilde{M})$ satisfies $\HRR_1$ on $\Phi(l_a)$. The only simple matroid of rank $2$ is the uniform matroid of rank $2$. From Corollary~\ref{simple-matroids-HRR-condition}, it suffices to check that the signature of the Hessian is $(+, -, \ldots, -)$. But the Hessian of a rank $2$ simple matroid at every point is the same matrix. If the matroid is on $n$ elements, the Hessian matrix is 
	\[
		A(K_n) = \begin{bmatrix} 
			0 & 1 & 1 & \ldots & 1 \\
			1 & 0 & 1 & \ldots & 1 \\
			1 & 1 & 0 & \ldots & 1 \\
			\vdots & \vdots & \vdots & \ddots & \vdots \\
			1 & 1 & 1 & \ldots & 0
		\end{bmatrix}.
	\]
	This happens to be the adjancency matrix of a complete graph. From Proposition 1.5 in \cite{Stanley-alg-combo}, this matrix has an eigenvalue of $-1$ with multiplicity $n-1$ and an eigenvalue of $n-1$ with multiplicity $1$. Hence, its signature is $(+, -, \ldots, -)$ which proves the base case. Now suppose that the claim is true for all matroids of rank less than $r$. Let $M$ be a matroid of rank $r$. We want to prove that the claim is true for $M$. By the same reasoning as in the base case, we can assume that our matroid is simple. In the simplification, the codimension of our $1$-form will not increase. For all $i \in E \backslash \{1\}$, we know that $e$ is not a co-loop of $M / i$ from Lemma~\ref{contracting-stays-not-a-coloop}. Thus, the simplified $l_a$ will either be all positive (in which case the claim is known to be true) or $l_a$ will be all positive except possibly at $e$. Since $e$ is known not to be a coloop of $M / i$, the simplification of $l_a$ will be a $1$-form which satisfies the hypothesis in the inductive hypothesis. Since $M$ is simple, we know that $\rank (M / i) = \rank (M) - 1 < \rank (M)$. Hence, from the inductive hypothesis, we know that $A(M/i) = A_{\partial_i f_M}^\bullet$ satisfies $\HRR_1$ for $l_a$. \\

	Now, we have enough information to directly prove that $A(M)$ satisfies $\HRR_1$ with respect to $l_a$. From Lemma~\ref{lorentzian-HL-iff-HRR}, it suffices to prove that $A(M)$ satisfies $\HL_1$ with respect to $l_a$. Since $\dim_\RR A^1(M) = \dim_\RR A^{r-1}(M)$ from properties of Poincar\'e Duality algebras, it suffices to prove that the Lefschetz operator $L_{l_a}^1 : A^1(M) \to A^{r-1}(M)$ is injective. Let $\Xi \in A^1(M)$ be the kernel of $L_{l_a}^1$. Note that this is equivalent to requiring $\Xi l_a^{r-2} = 0$ in $A(M)$. We want to prove that $\Xi = 0$ in $A^1(M)$. Since $\Xi l_a^{r-2} = 0$ in $A(M)$, we have 
	\[
		0 = -Q_{l_a}^1 (\Xi, \Xi) = \deg_M (\Xi^2 \cdot l_a^{r-2}) = \sum_{i = 2}^n a_i \deg_M (\Xi^2 \cdot l_a^{r-3} \cdot \partial_i).
	\]
	Note that 
	\[
		\deg_M (\Xi^2 \cdot l_a^{r-3} \cdot \partial_i) = (\Xi^2 \cdot l_a^{r-3}) (\partial_i f_M) = \deg_{M / i} (\Xi^2 \cdot l_a^{r-3}) = - Q_{M/i} (\Xi, \Xi).
	\]
	where $Q_{M/i}$ is the Hodge-Riemann form of degree $1$ with respect to $l$ associated with $A(M/i)$. Thus, 
	\[
		\sum_{i = 2}^n a_i Q_{M/i}(\Xi, \Xi) = 0.
	\]
	In $A(M/i)$, the linear form $\Xi$ is in the primitive subspace. Since we know that $A(M/i)$ satisfies $\HRR_1$ with respect to $l$, we know that the Hodge-Riemann form $Q_{M/i}$ is negative-definite on $\RR \Xi$. Since $a_i > 0$ for $i = 2, \ldots, n$, this implies that $Q_{M/i}(\Xi, \Xi) = 0$ for all such $i$. Hence $\Xi = 0$ in $A(M/i)$ for all $i \in [2, n]$. In terms of polynomials, this means that $\Xi (\partial_i f_M) = 0$ for all $i \in [2, n]$. From Theorem~\ref{socle-socle-socle}, this implies that $\Xi (f_M) = 0$. Thus, $A(M)$ satisfies $\HRR_1$ with respect to $l$. \\

	For the other direction, we will prove that if $e$ is a co-loop of $M$, then $f_M$ does not satisfy $\HRR_1$ on $\relint (H_e)$. Without loss of generality, we can suppose $E(M) = [n]$ and $e = 1$. The linear differential can be written as $l_a$ where $a = (0, a_2, \ldots, a_n)$ for $a_2, \ldots, a_n \geq 0$. It suffices to consider the case when $M$ is simple because the coefficient of $\partial_1$ in the simplification of $l_a$ will remain $0$. This is because co-loops have no parallel elements. In the simple case, the bottom $(n-1) \times (n-1)$ sub-matrix of $\Hess_{f_M}$ will be entirely $0$. This means that the Hodge-Riemann form will singular and $f_M$ cannot satisfy $\HRR_1$ on $\relint (H_e)$. This suffices for the proof. 
\end{proof}

With a same proof, we also have Theorem~\ref{thm:main-result-general-to-lower-facet}. 

\begin{thm} \label{thm:main-result-general-to-lower-facet}
	Let $M = (E, \mcI)$ be a matroid which satisfies $\rank (M) \geq 2$. For any $S \subseteq E$, if $\rank(S) \leq \rank (M) - 2$ and $S$ contains no co-loops, then $f_M$ satisfies $\HRR_1$ on $\relint (H_S)$. 
\end{thm}

Theorem~\ref{thm:main-result-general-to-lower-facet} in tandem with Theorem~\ref{thm:simple-inverse-hessian-property} immediately gives a somewhat random computational result that under the hypothesis in Theorem~\ref{thm:main-result-general-to-lower-facet}, we have that
\[
	\nabla f_{M / e}^T \cdot \Hess_{f_{M \backslash e}}^{-1} \cdot \nabla f_{M/e} |_{x = a} \neq 0.
\]
Using our result about high degree socles in Theorem~\ref{higher-degree-socles}, we can prove an inductive procedure for local $\HRR_k$ in Theorem~\ref{thm:inductive-procedure} for points $a \in \RR_{\geq 0}^E$ whose support is large enough. 
\begin{thm} \label{thm:inductive-procedure}
	Let $M = (E, \mcI)$ be a matroid and let $S \subseteq E$ be a subset such that $\rank (S) \leq \rank (M) - k - 1$. Let $l := \sum_{i \in E \backslash S} a_i \cdot e_i$ where $a_i > 0$. If $\partial_e f_M$ is $0$ or satisfies $\HRR_k$ and with respect to $l$ for all $e \in E \backslash S$, then $A(M)$ satisfies $\HL_k$ with respect to $l$. 
\end{thm}

\begin{proof}
	Note that $\partial_e f = 0$ if and only if $e$ is a loop. In this case, the variable $x_e$ doesn't appear in the basis generating polynomial. Hence, without loss of generality, we can suppose that $M$ is loopless and $\partial_e f \neq 0$ for all $e \in E$. Since $A(M)$ is a Poincar\'e duality algebra, it suffices to prove for $\Xi \in A^k(M)$ that if $\Xi l^{d-2k} = 0$ in $A^{r-k}(M)$, then $\Xi = 0$ in $A^k(M)$. We know that $\Xi$ is in the primitive subspaces of $\partial_e f$ for all $e \in E$ and we can compute the formula
	\[
		0 = Q^k (\Xi, \Xi) = \sum_{i \in E \backslash S} a_i Q_{\partial_i f} (\Xi, \Xi).
	\]
	From the positive definiteness on the primitive subspaces, we have $\Xi = 0$ in $A^k_{\partial_i f}$ for all $i \in E\backslash S$. From Theorem~\ref{higher-degree-socles}, we get $\Xi = 0$ in $A^k(M)$. This suffices for the proof. 
\end{proof}

From Theorem~\ref{thm:inductive-procedure}, it is enough to show that if $\A(M)$ satisfies $\HRR_i$ for $1 \leq i \leq k-1$ and $\HL_i$ for $1 \leq i \leq k$, then it satisfies $\HRR_k$. Under these conditions, satisfying $\HRR_k$ is equivalent to a statement about the net signature of the Hodge-Riemann form. We define the notion of net signature in Definition~\ref{def:net-signature}. 

\begin{defn} \label{def:net-signature}
	Let $B : V \times V \to k$ be a symmetric bilinear form on a finite dimensional vector space $V$. Suppose that the signature of $B$ has $n_+$ positive eigenvalues and $n_-$ negative eigenvalues. Then, we define the \textbf{net signature} to be $\sigma(B) = n_+ - n_-$. 
\end{defn}

When our bilinear form $B : V \times V \to k$ is non-degenerate, then the net signature $\sigma (B)$ determines the exact signature of the form. Indeed, in the non-degenerate case, we have $n_+ - n_- = \sigma(B)$ and $n_+ + n_- = \dim V$. 
\begin{lem} \label{formula-for-dimension-when-HRR-HL-are-satisfied}
	Let $A(M)$ satisfies $\HRR_i$ and $\HL_i$ with respect to $l$ for $1 \leq i \leq k$, then, we have 
	\[
		\sigma \left ( (-1)^kQ_l^k \right ) = \sum_{i = 0}^k (-1)^i (\dim A^i(M) - \dim A^{i-1}(M)).
	\]
\end{lem}

\begin{proof}
	We induct on $k$. For the base case, we have $k = 1$ and the claim follows from Proposition~\ref{conditions-for-HL-HRR}. Suppose that the claim holds for $k-1$. Consider the composition of maps given by the following commutative diagram:
	\begin{equation} \label{important-composition}
		\begin{tikzcd}
			{A^{k-1}(M)} && {A^k(M)} && {A^{d-k}(M)} && {A^{d-k+1}(M)}
			\arrow["{\times l_a}", from=1-1, to=1-3]
			\arrow["{\times l_a^{d-2k}}", from=1-3, to=1-5]
			\arrow["{\times l_a}", from=1-5, to=1-7]
			\arrow["{\psi_a}", bend right = 30, from=1-3, to=1-7]
		\end{tikzcd}
	\end{equation}
	This diagram exhibits an isomorphism between $A^{k-1}(M)$ and $A^{d-k+1}(M)$ from $\HL_{k-1}$. This implies that we can decompose $A^k(M)$ into 
	\[
		A^k(M) = l \cdot A^{k-1}(M) \oplus \ker \psi_a
	\] 
	where the two summands in the direct sum are orthogonal with respect to the Hodge-Riemann form $Q_l^k$. Let $u_1, \ldots, u_m \in A^{k-1}(M)$ be a basis for $A^{k-1}(M)$. Then, $l \cdot u_1, \ldots, l \cdot u_m$ is a basis for $A^k(M)$ and
	\[
		(-1)^{k} Q_{l}^k (l \cdot u_i, l \cdot u_j) = (-1)^{k-1} Q_{l}^{k-1} (u_i, u_j). 
	\]
	Thus, the signature of $Q_{l}^k$ on $l \cdot A^{k-1}(M)$ should be the negative of the signature of $Q_{l}^{k-1}$ on $A^{k-1}(M)$. Since $A(M)$ satisfies $\HRR_k$, we know that the signature of $Q_l^k$ on $\ker \psi$ is $\dim \ker \psi$. This gives us the formula
	\begin{align*}
		\sigma ((-1)^k Q_{l_a}^k) & = \sigma \left((-1)^{k-1} Q_{l_a}^{k-1}\right) + (-1)^k (\dim A^k(M) - \dim A^{k-1}(M)) \\
		& = \sum_{i = 0}^k (-1)^i (\dim A^i(M) - \dim A^{i-1}(M))
	\end{align*}
	from the inductive hypothesis. This suffices for the proof. 
\end{proof}
\begin{lem} \label{sufficient-conditions-for-higher-HRR}
	Let $k \geq 2$ and $A(M)$ satisfies $\HRR_{i}$ and $\HL_{i}$ with respect to $l = l_a$ for some $a \in \RR^n$ for all $1 \leq i \leq k-1$. Then
	\begin{enumerate}[label = (\alph*)]
		\item $A(M)$ satisfies $\HL_k$ with respect to $l$ if and only if $Q_l^k$ is non-degenerate on $A^k(M)$. 

		\item Suppose that $\HL_k$ is satisfied. Then $A(M)$ satisfies $\HRR_k$ with respect to $l$ if and only if 
		\[
			\sigma ((-1)^kQ_l^k) = \sum_{i = 0}^k (-1)^i (\dim A^i(M) - \dim A^{i-1}(M)).
		\]
	\end{enumerate}
\end{lem}

\begin{proof}
	The argument for (a) is exactly the same as the argument for (a) in Proposition~\ref{conditions-for-HL-HRR}. For part (b), consider the composition in Equation~\ref{important-composition}. Recall from the proof of Lemma~\ref{formula-for-dimension-when-HRR-HL-are-satisfied}, we have that 
	\[
		\sigma ((-1)^k Q_{l_a}^k) = \sigma \left((-1)^{k-1} Q_{l_a}^{k-1}\right) + \sigma \left((-1)^k Q_{l_a}^k |_{\ker \psi_a}\right).
	\]
	Since $A(M)$ satisfies $\HRR_k$ if and only if $Q_{l_a}^k$ is positive definition on $\ker \psi_a$, we know that $A(M)$ satisfies $\HRR_k$ if and only if $\sigma (Q_{l_a}^k) = \dim \ker \psi_a = \dim A^k(M) - \dim A^{k-1}(M)$. Then Lemma~\ref{formula-for-dimension-when-HRR-HL-are-satisfied} completes the proof.  
\end{proof}

In the situation where we know that $\A(M)$ satisfies $\HRR_i$ for $1 \leq i \leq k-1$ and $\HL_i$ for $1 \leq i \leq k$, we already know that $\HRR_k$ is equivalent to a statement about the signature of the Hodge Riemann form. Note that as a function of the point $a \in \RR^E$ in $l_a$, the Hodge-Riemann form is continuous. In particular, the eigenvalues of the Hodge-Riemann form are continuous functions. This property will allow us to prove Lemma~\ref{one-element-is-enough}. This result tells us that it is enough to prove that $\HRR_k$ and $\HL_k$ holds on at least one boundary point. 

\begin{lem} \label{one-element-is-enough}
	Let $\Omega \subseteq \RR^n$ be a subset such that for all $x \in \Omega$, $A(M)$ satisfies $\HRR_i$ for $i \leq k-1$ and $\HL_i$ for $i \leq k$. Suppose that there is a continuous path $\gamma : [0, 1] \to \RR^n$ with image $\varphi ([0, 1]) \subseteq \Omega$, suct that $A(M)$ satisfies $\HRR_k$ with respect to $l_{\gamma(0)}$. Then $A(M)$ satisfies $\HRR_k$ with respect to $l_{\gamma(1)}$. 
\end{lem}

\begin{proof}
	Let $\lambda_i(a)$ be the $i$th largest eigenvalue of $Q_{l_a}^k$ as a function in $a$ for $1 \leq i \leq \dim A^k(M)$. Then for all $1 \leq i \leq \dim A^k(M)$, the $i$th eigenvalue $\lambda_i(a)$ is a continuous function in $a$. Along the path $\gamma$, we know that $Q_{l_a}^k$ is non-degenerate. Hence, none of the functions $\lambda_i(a)$ cross zero. This implies that on the path, the signature $\sigma \left((-1)^k Q_{l_a}^k\right)$ remains constant. From Lemma~\ref{sufficient-conditions-for-higher-HRR}(b), this completes the proof.
\end{proof}

\section{Future Work} \label{sec:future-work}

As noted in the thesis, the Conjecture~\ref{conj:my-conjecture} was recently to be found erroneous. However, we will still describe our thought process in our attempt to prove it. We first proved Theorem~\ref{thm:main-result-general-to-lower-facet}. This describes the exact conditions needed for $\A(M)$ to satisfy $\HRR_1$ on the faces of the positive orthant. Theorem~\ref{thm:inductive-procedure} tells us that for general $k$, we have a mechanism from going from local $\HRR_k$ to $\HL_k$ on the positive orthant and possibly its boundary depending on some dimension conditions. This means that to prove the general conjecture, it suffices to prove that we can go from $\HRR_i$ for $1 \leq i \leq k-1$ and $\HL_i$ for $1 \leq i \leq k$ to $\HRR_k$. In Lemma~\ref{one-element-is-enough}, we further reasoned that given $\HRR_i$ for $1 \leq i \leq k-1$ and $\HL_i$ for $1 \leq i \leq k$ as the inductive hypothesis, it suffices to prove that $\HRR_k$ is satisfied on a suitable boundary point. The reason why we might want to pick a boundary point even though it may seem more accessible is because we may hope for a semi-small decomposition for $\A(M)$ as in Theorem~\ref{thm:chow-ring-decomp} and Theorem~\ref{thm:augmented-chow-ring-krull-schmidt-decomp}. Indeed, if we were to have such a decomposition, we could reduce $\HRR_k$ on the boundary for $\A(M)$ to $\HRR_k$ for $\A(M \backslash i)$. The conjecture would then follow from a more elaborate induction argument. Before we can contemplate such an argument, there are a few questions we must answer. 
\begin{question}
	Does there exist a graded ring homomorphism $\theta : \A(M \backslash i) \to \A(M)$ for every $i \in E$? 
\end{question}
In order to have a decomposition from which we can extract inductive information about the Hodge structures, we must write $\A(M)$ as the direct sum of $\A(M \backslash i)$ submodules. Even before this, we would require that $\A(M)$ be a $\A(M \backslash i)$-module. Note that there is a well-defined surjection $\Theta_M : \operatorname{H}(M) \to \A(M)$ between the graded M\"obius algebra and the ring $\A(M)$ defined by 
\[
	\Theta_{M}(y_F) = x^I, \quad \text{ where $I$ is a basis of $F$}.
\]
A reasonable guess for an injection $\A(M \backslash i) \to \A(M)$ would be the map which makes the diagram in Equation~\ref{eqn:mobius-basis-ring-injection} commute. 
\begin{equation} \label{eqn:mobius-basis-ring-injection}
	\begin{tikzcd}
	{\operatorname{H}(M\backslash i)} && {\operatorname{H} (M)} \\
	\\
	{\operatorname{A}(M\backslash i)} && {\operatorname{A}(M)}
	\arrow[from=1-1, to=1-3]
	\arrow[from=1-3, to=3-3]
	\arrow[from=1-1, to=3-1]
	\arrow[from=3-1, to=3-3]
\end{tikzcd}
\end{equation}
This map is forced to send the equivalence class of $\A(M \backslash i)$ represented by a polynomial $f$ to the equivalence class of $\A(M)$ represented by the same polynomial. Unfortunately, it is not clear that this map is well-defined. In order to be well-defined, Question~\ref{question:question2} must have an affirmative answer. 

\begin{question} \label{question:question2}
	Let $M$ be a matroid on the ground set $[n]$. Let $\xi \in \RR[\partial_1, \ldots, \partial_{n-1}]$ such that $\xi (f_{M \backslash n}) = 0$. Is it true that $\xi (f_M) = 0$? 
\end{question}

In the case where $n$ is a coloop, the answer to Question~\ref{question:question2} is Yes. When $n$ is not a coloop, we can write $f_M = x_n f_{M / n} + f_{M \backslash n}$. In this case, Question~\ref{question:question2} becomes equivalent to Question~\ref{question:question3}. 

\begin{question} \label{question:question3}
 	Let $M$ be a matroid and suppose that $e \in E(M)$ is not a coloop. Is it the case that $\Ann (f_{M \backslash e}) \subseteq \Ann (f_{M / e})$? 
\end{question}

At this time, we have not found a counter-example to Question~\ref{question:question3}. Note that the operation $M \backslash e \to M / e$ is called a \textbf{matroid quotient}. Assuming that the answers to all of these questions are Yes, then a decomposition of $\A(M)$ may possibly descend from a decomposition of $\operatorname{H}(M)$. Even if the mathematics doesn't work in this way, the question of having a semi-small decomposition of $\operatorname{H}(M)$ is an interesting question. We know that $\opH(M)$ is a $\opH(M \backslash i)$ module. From the Krull-Schmidt theorem, there is some unique $\opH(M \backslash i)$-module decomposition of the form
\[
	\opH(M) = S_1 \oplus \ldots \oplus S_m
\]
where the $S_i$'s are indecomposable $\opH(M \backslash i)$-modules. If we look at the $0$-degree parts of these modules, we know that $H^0(M) \cong \RR$. Thus, all but one of the $S_i$'s must have trivial $0$-degree parts. This means that $1 \in S_i$ for some $i \in [m]$. Since they are $\opH(M \backslash i)$-modules, this means that $\opH (M \backslash i) \subseteq S_i$ for some $i \in [m]$. This gives us some idea on how a semi-small decomposition $\opH(M)$ could be constructed. However, we discovered that $\opH(M)$ does not necessary have a semi-small decomposition. Indeed, consider the matroid that comes from the affine configuration of $\{A, B, C, D, E\}$ in Figure~\ref{fig:indecomposable}. One can calculate that as $\opH (M \backslash A)$ modules, the ring $\opH(M)$ is indecomposable. This does not bode well for the possibility for a general semi-small decomposition. 
\begin{figure}[h]
	\begin{center}
		\includegraphics[scale = 0.5]{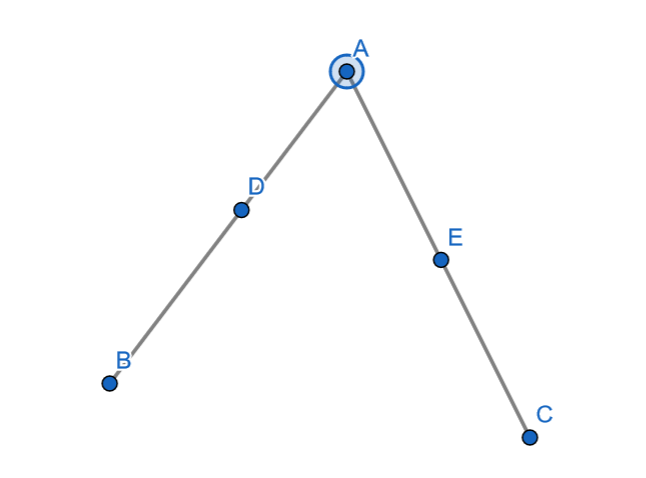}
		\caption{$\opH(M)$ is indecomposable}
		\label{fig:indecomposable}
	\end{center}
\end{figure}
However, this may be a result of picking a bad point to delete. Perhaps, there always exists a point that we can delete to give a nice decomposition. Another possibility is that there may be a decomposition in low dimensions ($\leq \frac{d}{2}$). For example, one conjecture is that of Question~\ref{question:question4}, which asks about the equivalence of $\A(M)$ and $\opH(M)$ in small dimensions. 

\begin{question} \label{question:question4}
	Is it true that $\opH^k(M) \to \A^k(M)$ is an isomorphism when $k \leq \frac{d}{2}$?
\end{question}

For any graded ring $R$ of top dimension $d$, we let $D^k R$ be the graded ring of dimension $k$ where we annihilate all elements of degree greater than $k$. Question~\ref{question:question5} asks if there is a possibility of a semi-small decomposition of $\opH(M)$ for a small enough $k$? 

\begin{question} \label{question:question5}
	Is $D^k\opH(M \backslash i)$ a summand in the Krull-Schmidt decomposition of $D^k \opH(M)$ for $k \leq \frac{d}{2}$? 
\end{question}

However, much of this strategy will not work because of a counter-example to Conjecture~\ref{conj:my-conjecture} that was provided to us by Connor Simpson. This counter-example also gives a negative answer to Question~\ref{question:question4}. Let $G = K_{2, 3}$ be the complete bipartite graph on $2+3$ vertices. Let $M = M(G)$ be the graphic matroid of rank $4$ with $G$ as the underlying graph. Then, $M$ has $15$ flats of rank $2$. The pairing on $\opH^2(M)$ has $6$ positive eigenvalues, $6$ negative eigenvalues, and $3$ zero eigenvalues. The zero eigenvalues implies that the answer of Question~\ref{question:question4} is \textbf{no}. The negative eigenvalues gives the contradiction to Conjecture~\ref{conj:my-conjecture}. However, there is still a possibility for Conjecture~\ref{conj:mny} to be true. If it were, this would provide an example of a polynomial algebra for which Hard Lefschetz holds but in general the Hodge-Riemann relations do not.

\chapter{Appendix}

\section{Brunn-Minkowski and the Base Case of the Alexandrov-Fenchel Inequality}

\begin{prop} \label{base-case-for-AF-inequality}
	Let $K, L \subseteq \RR^2$ be two convex bodies in the plane. Then $\mathsf{V}_2(K, L)^2 \geq \mathsf{V}(K, K) \mathsf{V}(L, L)$. 
\end{prop}

\begin{proof}
	From Theorem~\ref{brunn-minkowski}, we have $\sqrt{\Vol_2(K + L)} \geq \sqrt{\Vol_2(A)} + \sqrt{\Vol_2(B)}$. Squaring both sides, we get the equivalent inequality
	\[	
		\Vol_2(K+L) - \Vol_2(K) - \Vol_2(L) \geq 2 \sqrt{\Vol_2(K) \Vol_2(L)} = 2 \sqrt{\mathsf{V}_2(K, K) \mathsf{V}_2(L, L)}.
	\]
	From Theorem~\ref{mixed-volume-polynomial-expansion-FINAL}, the left hand side is equal to 
	\[		
		\Vol_2(K+L) - \Vol_2(K) - \Vol_2(L) = 2 \mathsf{V}_2(K, L).
	\]
	This suffices for the proof. 
\end{proof}

\section{Computations in the Kahn-Saks Inequality}

Let $(P, \leq)$ be a finite poset and fix two elements $x, y \in P$. For the sake of simplicity, we can assume that $x \leq y$. In the proof of Theorem~\ref{kahn-saks}, we defined special cross-sections of the order polytope for every $\lambda \in [0, 1]$ given by
\[
	K_\lambda := \{t \in \mcO_P : t_y - t_x = \lambda\}.
\]
We give a proof of Lemma~\ref{cross-section-mixed-lemma-computation-that-wasnt-in-kahn-saks}. This result is important in the proof of Theorem~\ref{kahn-saks} as it gives a way to interpret the volume of $K_\lambda$ in terms of mixed volumes. The proof of Lemma~\ref{cross-section-mixed-lemma-computation-that-wasnt-in-kahn-saks} was stated in \cite{balancing-poset-extensions} but was not proven. We include the proof for the sake of completeness. 

\begin{lem} \label{cross-section-mixed-lemma-computation-that-wasnt-in-kahn-saks}
	For $\lambda \in [0, 1]$ we have that $K_\lambda = (1-\lambda) K_0 + \lambda K_1$. 
\end{lem}

\begin{proof}
	For all $\lambda \in \RR$, we can define the hyperplanes $H_\lambda = \{t \in \RR^n : t_y - t_x = \lambda\}$. Then $K_\lambda = H_\lambda \cap \mcO_P$. It is a simple computation to prove that $H_\lambda = (1-\lambda) H_0 + \lambda H_1$. For any linear extension $\sigma \in e(P)$, we define the polytopes:
	\begin{align*}
		\Delta_\sigma & := \{0 \leq t_{\sigma^{-1}(1)} \leq \ldots \leq t_{\sigma^{-1}(n)} \leq 1\} \subseteq \mcO_P \\
		\Delta_\sigma (\lambda) & := \{t \in \Delta_\sigma: t_y - t_x = \lambda\} = K_\lambda \cap \Delta_\sigma = H_\lambda \cap \Delta_\sigma.
	\end{align*}
	Then, we can decompose $K_\lambda$ into 
	\[
		K_\lambda = H_\lambda \cap \mcO_P = H_\lambda \cap \bigcup_{\sigma \in e(P)} \Delta_e = \bigcup_{\sigma \in e(P)} \Delta_\sigma(\lambda).
	\]
	where the polytopes in the unions are disjoint up to a set of measure zero. Now, note that we have 
	\begin{align*}
		(1-\lambda) K_0 + \lambda K_1 & \subseteq (1-\lambda) \mcO_P + \lambda \mcO_P = \mcO_P \\
		(1-\lambda)K_0 + \lambda K_1 & \subseteq (1-\lambda) H_0 + \lambda H_1 = H_\lambda.
	\end{align*} 
	Thus, we have that $(1-\lambda) K_0 + \lambda K_1 \subseteq \mcO_P \cap H_\lambda = K_\lambda$. It suffices to prove the opposite inclusion. Let $t \in K_\lambda$ be an arbitrary point. From our decomposition of $K_\lambda$ into polytopes $\Delta_\sigma (\lambda)$, we know that there is some linear extension $\sigma$ satisfying $\sigma(x) < \sigma(y)$ and $t \in \Delta_\sigma (\lambda)$. The linear extension $\sigma$ induces a total order $<_\sigma$ on $P$ defined by $z_1 <_\sigma z_2$ if and only if $\sigma(z_1) < \sigma(z_2)$. In particular, when $\omega_1 \geq_\sigma \omega_2$, we have that $t_{\omega_1} \geq t_{\omega_2}$. Now, we can define two points $P_0, P_1 \in \RR^n$ so that for all $\omega \in P$, we have 
	\begin{align*}
		(P_0)_\omega & = \begin{cases}
			\frac{t_\omega}{1-\lambda} & \text{if $\omega \leq_\sigma x$} \\
			\frac{t_x}{1-\lambda} & \text{if $x <_\sigma \omega <_\sigma y$} \\
			\frac{t_\omega - \lambda}{1-\lambda} & \text{if $\omega \geq_\sigma y$}.
		\end{cases} \\
		(P_1)_\omega & = \begin{cases}
			0 & \text{if $\omega \leq_\sigma x$} \\
			\frac{t_\omega - t_x}{\lambda} & \text{if $x <_\sigma \omega <_\sigma y$} \\
			1 & \text{if $\omega \geq y$}.
		\end{cases}
	\end{align*}
	Then $P_0 \in K_0$, $P_1 \in K_1$, and $t = (1-\lambda) P_0 + \lambda P_1$. Since $t$ was arbitrary, this completes the proof that $K_\lambda = (1-\lambda) K_0 + \lambda K_1$.
\end{proof}

\begin{lem} \label{computation-of-mixed-volume-in-kahn-saks-case}
	For $i \in \{1, \ldots, n-1\}$, we have that $N_i = (n-1)! \mathsf{V}_{n-1}(K_0[n-i], K_1[i-1])$. 
\end{lem}
\begin{proof}
	From Theorem~\ref{mixed-volume-polynomial-expansion} and Lemma~\ref{cross-section-mixed-lemma-computation-that-wasnt-in-kahn-saks}, we have that 
	\begin{align*}
		\Vol_{n-1}(K_\lambda) & = \Vol_{n-1}  ( (1-\lambda) K_0 + \lambda K_1 ) \\
		& = \sum_{j = 0}^{n-1} \binom{n-1}{j} \mathsf{V}_{n-1}(K_0[n-j-1], K_1[j]) (1-\lambda)^{n-j-1} \lambda^j.
	\end{align*}
	Alternatively, we can compute the volume $\Vol_{n-1}(K_\lambda)$ as 
	\begin{align*}
		\Vol_{n-1} (K_\lambda) & = \Vol_{n-1} \left ( \bigcup_{\substack{\sigma \in e(P) \\ \sigma(x) < \sigma(y)}} \Delta_\sigma (\lambda) \right ) \\
		& = \sum_{\substack{\sigma \in e(P) \\ \sigma(x) < \sigma(y)}} \Vol_{n-1}(\Delta_\sigma (\lambda)) \\
		& = \sum_{k = 1}^{n-1} \sum_{\substack{\sigma \in e(P) \\ \sigma(y) - \sigma(x) = k}} \Vol_{n-1}(\Delta_\sigma(\lambda)).
	\end{align*}
	To compute $\Vol_{n-1}(\Delta_\sigma (\lambda))$, let $P = \{z_1, \ldots, z_n\}$ with $x = z_i$, $y = z_j$ with $\sigma(z_k) = k$ for $k \in [n]$. In this notation, we have the explicit description of $\Delta_\sigma (\lambda)$ as 
	\[
		\Delta_\sigma (\lambda) = \{t \in \RR^n : 0 \leq t_1 \leq \ldots \leq t_n \leq 1 \text{ and } t_j - t_i = \lambda \}.
	\]
	This set of inequalities is equivalently decribed by the inequalities
	\begin{align*}
		0 & \leq t_{i+1} - t_i \leq \ldots \leq t_{j-1} - t_i \leq \lambda \\
		0 & \leq t_1 \leq \ldots \leq t_{i-1} \leq t_i \leq t_{j+1} - \lambda \leq \ldots \leq t_n - \lambda \leq 1 - \lambda.
	\end{align*}
	Thus, the volume of $\Delta_n (\lambda)$ ends up being the product of two simplices. Specifically, we have the equation 
	\[
		\Vol_{n-1}(\Delta_\sigma(\lambda)) = \frac{\lambda^{j-i-1}}{(j-i-1)!} \cdot \frac{(1-\lambda)^{n-(j-i)}}{(n-(j-i))!}. 
	\]
	We can finish the computation with 
	\begin{align*}
		\Vol_{n-1}(K_\lambda) & = \sum_{k = 1}^{n-1} \sum_{\substack{\sigma \in e(P) \\ \sigma(y) - \sigma(x) = k}} \frac{\lambda^{k-1}}{(k-1)!} \cdot \frac{(1-\lambda)^{n-k}}{(n-k)!} \\
		& = \sum_{k = 1}^{n-1} \binom{n-1}{k-1} \frac{N_k}{(n-1)!} \cdot \lambda^{k-1}(1-\lambda)^{n-k}.
	\end{align*}
	Thus, we have that 
	\[
		\frac{N_k}{(n-1)!} = \mathsf{V}_{n-1}(K_0[n-k], K_1[k-1]) \implies N_k = (n-1)! \mathsf{V}_{n-1}(K_0[n-k], K_1[k-1]).
	\]
	This suffices for the proof of the Lemma. 
\end{proof}

\begin{lem} \label{affine-hulls}
    The affine hulls of $K_0$, $K_1$, and $K_0 + K_1$ are given by 
    \begin{align*}
        \aff (K_0) & = \RR \left [ \sum_{\omega \in P_{x \leq \cdot \leq y}} e_\omega \right ] \oplus \bigoplus_{\omega \in P \backslash P_{x \leq \cdot \leq y}} \RR[e_\omega] \\
        \aff (K_1) & = \sum_{\omega \in P_{\geq y}} e_\omega + \bigoplus_{\omega \in P \backslash {P_{\leq x} \cup P_{\geq y}}} \RR[e_\omega] \\
        \aff (K_0 + K_1) & = \sum_{\omega \in P_{\geq y}} e_\omega + \RR \left [ \sum_{\omega \in P_{x \leq \cdot \leq y}} e_\omega \right ] \oplus \bigoplus_{\omega \in P \backslash \{x, y\}} \RR[e_\omega]
    \end{align*}
    where $\RR[v]$ denotes the linear span of the vector $v$. As an immediate corollary, we have
    \begin{align*}
        \dim K_0 & = n - |P_{x < \cdot < y}| - 1 \\
        \dim K_1 & = n - |P_{< x}| - |P_{> y}| - 2 \\
        \dim (K_0 + K_1) & = n-1.
    \end{align*}
\end{lem}

\begin{proof}
    From Proposition~\ref{extension-narrow-exists}, there is a linear extension $f : P \to [n]$ such that $f(y) - f(x) = |P_{x < \cdot < y}| + 1$. Then, there are elements $\alpha, \beta_j, \gamma_k$ for $1 \leq i \leq a$, $1 \leq j \leq b$, and $1 \leq k \leq c$ such that 
    \[
        f(\alpha_1) < \ldots < f(\alpha_a) < f(x) < f(\beta_1) < \ldots < f(\beta_b) < f(y) < f(\gamma_1) < \ldots < f(\gamma_c).
     \] 
    This allows us to compute 
    \begin{align*}
        K_0 \cap \Delta_f = \{t \in [0, 1]^n : t_{\alpha_1} \leq \ldots \leq t_{\alpha_a} \leq t_x = \ldots = t_y \leq t_{\gamma_1} \leq \ldots \leq t_{\gamma_c} \}.
    \end{align*}
    In particular, by taking affine spans, we have 
    \[
        \aff K_0 \supseteq \aff (K_0 \cap \Delta_f) = \RR \left [\sum_{\omega \in P_{x \leq \cdot \leq y}} e_\omega \right ] \oplus \bigoplus_{\omega \in P \backslash P_{x \leq \cdot \leq y}} \RR[e_\omega].
    \]
    To prove the other inclusion, let $v \in K_0$ be an arbitrary vector. Since we are working in a subset of the order polytope, it must be the case that $v_x \leq v_\omega \leq v_y$ for all $\omega \in P_{x \leq \cdot \leq y}$. But since $v_x = v_y$, this means that $v_{\omega} = v_x = v_y$ are all $\omega \in P_{x \leq \cdot \leq y}$. Thus
    \[
        v \in \RR \left [ \sum_{\omega \in P_{x \leq \cdot \leq y}} e_\omega \right ] \oplus \bigoplus_{\omega \in P \backslash P_{x \leq \cdot \leq y}} \RR[e_\omega].
    \]
    Since $v$ is arbitrary, we have proved the given formula for $\aff K_0$. \\

    To prove the formula for $K_1$, we first use Proposition~\ref{extension-wide-exists} to construct a linear extension $g : P \to [n]$ satisfying $g(x) = 1 + |P_{< x}$ and $g(y) = n - |P_{> y}|$. Note that 
    \[
        K_1 = \{t \in \mathcal{O}_P : t_y - t_x = 1 \} = \{t \in \mathcal{O}_P : t_y = 1, t_x = 0\}.
    \]
    In particular, for any $v \in K_1$, for all $\omega \leq x$ and $\eta \geq y$ we must have $v_\omega = 0$ and $v_\eta = 1$. Thus, 
    \[
        K_1 \subseteq \sum_{\omega \in P_{ \geq y}} e_\omega + \bigoplus_{\omega \in P \backslash (P_{\leq x} \cup P_{\geq y})} \RR[e_\omega].
    \]
    On the other hand, if we have $P_{\leq x} = \{\alpha_1, \ldots, \alpha_a, x\}$ ordered by according to $g$ and $P_{\geq y} = \{y, \beta_1, \ldots, \beta_b\}$ ordered according to $g$, then
    \[
        K_1 \cap \Delta_g = \{t \in [0, 1]^n : 0 = t_{\alpha_1} \leq \ldots \leq t_{\alpha_a} \leq t_x \leq \ldots \leq t_y = t_{\beta_1} = \ldots = t_{\beta_b} = 1 \}.
    \]
    Taking the affine span, we have 
    \[
        \aff K_1 \supseteq \aff (K_1 \cap \Delta_g) = \sum_{\omega \in P_{\geq y}} e_{\omega} + \bigoplus_{\omega \in P \backslash (P_{\leq x} \cup P_{\geq y}} \RR[e_\omega]. 
    \]
    This completes the proof for the given formula for $\aff K_1$. The third formula follows from the fact that $\aff (K+L) = \aff K + \aff L$ for any non-empty sets $K$ and $L$. 
\end{proof}

\subsection{Modifications and linear extensions of posets}

In this section, we will define certain linear extensions and linear extension modifications to help with the analysis of the polytopes associated to the Kahn-Saks sequence. In particular, it will also be useful to have examples of extremal linear extensions which place $x$ and $y$ close to each or very far from each other. It will also be useful to modify an existing linear extension to make a covering relation adjacent in a linear extension with changing too many terms in the original extension. 

\begin{prop} \label{extension-exists}
    The poset $P$ as at least one linear extension.
\end{prop}

\begin{proof}
    We induct on the size of $P$. If $|P| = 1$, then the partial order is already a total order. Now suppose the claim is true for posets of size $n-1$ and let $P$ be a poset of size $n$. Let $m \in P$ be a maximal element. Let $P'$ be the poset which you get by deleting $m$. From the inductive hypothesis, there is a linear extension $g : P' \to [n]$. Then, the map $f : P \to [n]$ defined by 
    \[
        f(\omega) = 
        \begin{cases}
            g(\omega) & \text{ if $\omega \neq m$}, \\
            n & \text{ if $\omega = m$}
        \end{cases}
    \]
    is a linear extension of $P$.
\end{proof}

\begin{prop} \label{extension-narrow-exists}
    There exists a linear extension $f : P \to [n]$ satisfying $f(y) - f(x) = |P_{x < \cdot < y}| + 1$. 
\end{prop}

\begin{proof}
    Let $\mathscr{L}$ denote the set of linear extensions of $P$. From Proposition~\ref{extension-exists}, there exists a linear extension $f_{\mathsf{min}} \in \mathscr{L}$ which minimizes $f (y) - f(x)$ out of all $f \in \mathscr{L}$. Since $x \leq y$, we know 
    \[
        f_{\mathsf{min}}(y) - f_{\mathsf{min}}(x) \geq 1 + |P_{x < \cdot < y}|.
    \]
    If $f_{\mathsf{min}}(y) - f_{\mathsf{min}}(x) = 1 + |P_{x < \cdot < y}|$, then we proposition is proved. Otherwise, we can assume $f_{\mathsf{min}}(y) - f_{\mathsf{min}}(x) \geq 2 + |P_{x < \cdot < y}|$. Define the set
    \begin{align*}
        \mathsf{M}_{f_{\mathsf{min}}} & := \{z \in P : f_{\mathsf{min}} (x) < f_{\mathsf{min}}(z) < f_{\mathsf{min}}(y) \} \\
        & = f_{\mathsf{min}}^{-1} \left \{ ( f_{\mathsf{min}} (x) , f_{\mathsf{min}} (y) ) \right \}
    \end{align*}
    which consist of the elements of the poset which appear between $x$ and $y$ in the linear extension $f_{\mathsf{min}}$. From the assumption $f_{\mathsf{min}}(y) - f_{\mathsf{min}}(x) \geq 2 + |P_{x < \cdot < y}|$, we know there is at least one element in $\mathsf{M}_{f_{\mathsf{min}}}$ which is not comparable to both $x$ and $y$. This is because any element that is comparable to both $x$ and $y$ must be in the set $P_{x < \cdot < y}$ which doesn't have enough elements to satisfy the inequality. Without loss of generality, suppose that there exists an element which is not comparable to $x$. Let $z \in \mathsf{M}_{f_{\mathsf{min}}}$ be such an element which minimizes $f_{\mathsf{min}}(z)$. In other words, in the linear extension, it is the leftmost element of this type. All elements between $x$ and $z$ in the linear extension must be greater than $x$ from our choice of $z$. In particular, those elements cannot be less than $z$ since that would imply the relation $x < z$. This means that we can move $z$ to the immediate left of $x$ and still have a working linear extension. But the resulting linear extension will have a smaller value of $f(y) - f(x)$. This contradicts the minimality of $f_{\mathsf{min}}$ and completes the proof of the proposition.
\end{proof}

\begin{prop} \label{extension-wide-exists}
    There exists a linear extension $f : P \to [n]$ satisfying $f(x) = |P_{< x}| + 1$ and $f(y) = n - |P_{> y}|$. 
\end{prop}

\begin{proof}
    Consider any linear extensions of the three subposets: $P_{\leq x}$, $P \backslash (P_{\leq x} \cup P_{\geq y})$, and $P_{\geq y}$. By appending the extensions in this order, we get an extension of $P$ of the desired form. 
\end{proof}

\begin{prop} \label{covering-modification}
    Let $f : P \to [n]$ be a linear extension and let $a, b \in P$ be elements such that $a \lessdot b$. Then there is a linear extension $\tilde{f} : P \to [n]$ satisfying $\tilde{f}(b) - \tilde{f}(a) = 1$ and $\tilde{f}(z) = f(z)$ whenever $f(z) < f(a)$ or $f(z) > f(b)$. 
\end{prop}

\begin{proof}
    Consider the sub-poset consisting of $a$, $b$, and the elements in between $a$ and $b$ in the linear extension $f$. By Proposition~\ref{extension-narrow-exists}, there is a linear extension of this sub-poset so that $a$ lies to the immediate left of $b$. By replacing the portion of the extension $f$ containing the elements of this sub-poset by this linear extension, we get the desired linear extension. 
\end{proof}

\begin{lem} \label{add-new-relation}
    Let $(P, \leq)$ be a poset and let $x, y \in P$ be two elements such that $x \not \geq y$. Suppose that $x$ and $y$ are incomparable with respect to the order $\leq$. Let $\leq_*$ be a binary relation defined by $z_1 \leq_* z_2$ if and only if $z_1 \leq z_2$ or $z_1 \leq x$ and $z_2 \geq y$. Then $\leq_*$ is a partial order on $P$. 
\end{lem}

\begin{proof}
    In the following, we separate the proofs of reflexivity, anti-symmetry, and transitivity for the binary relation $\leq_*$. Let $a, b, c \in P$ be arbitrary points in the poset. 
    \begin{enumerate}[label = (\roman*)]
        \item Since $a \leq a$, we must have $a \leq_* a$. This proves reflexivity of $\leq_*$.

        \item If $a \leq_* b$ and $b \leq_* a$, then there are four possibilities. The first possibility is $a \leq b$ and $b \leq a$. In this case, we get $a = b$. The second possibility is $a \leq x, y \leq b$ and $b \leq x, y \leq a$. In this case, $y \leq a \leq x$ so $y \leq x$. But this contradicts our assumption that $x \not \geq y$. The third possibility is $a \leq b$ and $b \leq x$, $y \leq a$. In this case, we have $y \leq a \leq b \leq x$ which cannot happen for the same reason the second possibility cannot happen. Similarly, the fourth possibility cannot occur. Thus, in any case, $\leq_*$ satisfies the anti-symmetry property.

        \item For the final property, suppose that $a \leq_* b$ and $b \leq_* c$. Again, there are four possibilities. The first possibility is $a \leq b$ and $b \leq c$. Then $a \leq c$ which implies $a \leq_* c$. The second possibility is $a \leq b$, $b \leq x$, and $y \leq c$. In this case, $a \leq x$ and $y \leq c$. Thus $a \leq_* c$ holds. The third possibility is $a \leq x$, $y \leq b$, and $b \leq c$. In this case $a \leq x$ and $y \leq c$. Thus $a \leq_* c$ holds. The fourth possibility is $a \leq x, y \leq b$ and $b \leq x, y \leq c$. But this implies that $y \leq x$ which cannot happen. In all cases, $\leq_*$ satisfies the transitivity property.  
    \end{enumerate} 
\end{proof}

\begin{lem} \label{doesn't-change}
    Let $(P, \leq)$ be a poset and $x, y \in P$ be elements such that $x \not \geq y$. Let $\leq_*$ be the partial order as defined in Lemma~\ref{add-new-relation}. Let $\{N_k\}$ be the Kahn-Saks sequence with respect to $\leq$ and let $\{\widetilde{N}_k\}$ be the Kahn-Saks sequence with respect to $\leq_*$. Then $N_k = \widetilde{N}_k$ for all $k$. 
\end{lem}

\begin{proof}
    We prove the stronger statement that a bijective map $f : P \to [n]$ is a linear extension of $\leq$ satisfying $f(y) - f(x) = k$ if and only if it is a linear extension of $\leq_*$ satisfying $f(y) - f(x) = k$. Clearly any linear extension of $\leq_*$ is a linear extension of $\leq$. Now suppose that $f$ is a linear extension of $\leq$ satisfying $f(y) - f(x) = k$. Suppose for the sake of contradiction that it is not a linear extension of $\leq_*$. This means that there are two elements $a, b \in P$ such that $a <_* b$ and $f(b) < f(a)$. Since $a <_* b$, we either have $a < b$ or $a \leq x$ and $y \leq b$. In the first case, we would have $f(a) < f(b)$ since $f$ is a linear extension of $\leq$. This would be give a contradiction. For the second case, we would have $f(a) \leq f(x) < f(y)\leq f(b)$, which is also a contradiction. This suffices for the proof. 
\end{proof}

Lemma~\ref{linear-algebraic-lemma} is a linear algebraic result which allows us skimp on calculating the exact affine span of our faces. 

\begin{lem} \label{linear-algebraic-lemma}
    Let $W_1$ and $W_2$ be finite dimensional affine spaces. Let $V_1 \subseteq W_1$ and $V_2 \subseteq W_2$ be affine subspaces such that $\dim V_1 = \dim W_1 - a$ and $\dim V_2 = \dim W_2 - b$ for some $a, b \geq 0$. Then $\dim (V_1 + V_2) \geq \dim (W_1 + W_2) - a - b$. 
\end{lem}

\begin{proof}
    After suitable translations, it suffices to prove the result when all our spaces are vector spaces. Then, we have 
    \begin{align*}
        \dim (V_1 + V_2) & = \dim V_1 + \dim V_2 - \dim (V_1 \cap V_2) \\
        & \geq \dim W_1 + \dim W_2 - \dim (W_1 \cap W_2) - a - b \\
        & = \dim (W_1 + W_2) - a - b. 
    \end{align*}
    This completes the proof of the lemma. 
\end{proof}

Our main application of Lemma~\ref{linear-algebraic-lemma} is when $W_1$ and $W_2$ are the affine spans of $K_0$ and $K_1$ while $V_1$ and $V_2$ are the affine spans of faces of $K_0$ and $K_1$. 

\section{Stanley Basis Inequality for Graphic Matroids}

\begin{lem} \label{lem:base-case}
	Theorem~\ref{stanley-matroid-thm-for-graphic-matroids-equality} is true when $G = (V, E)$ is a connected graph on $|V| = 3$ vertices. 
\end{lem}

\begin{proof}
	The only non-trivial case to check is for $n = 3$. We can define $r_1 = e_R(v_2, v_3)$, $r_2 = e_R(v_1, v_3)$, and $r_3 = e_R(v_1, v_2)$. We can then define $q_1, q_2$, and $q_3$ similarly. From the same reasoning as in Theorem~\ref{stanley-matroid-thm-for-graphic-matroids-equality}, we get that
	\begin{align*}
		r_1 & = \lambda_1 q_1,  r_2 = \lambda_2 q_2, r_3 = \lambda_3 q_3 
	\end{align*}
	and 
	\begin{align*}
		r_1 + r_2 & = \lambda_1 (q_1 + q_2) = \lambda_2 (q_1 + q_2) \\
		r_2 + r_3 & = \lambda_2 (q_2 + q_3) = \lambda_3 (q_2 + q_3) \\
		r_3 + r_1 & = \lambda_3 (q_3 + q_1) = \lambda_1 (q_3 + q_1)
	\end{align*}
	for some constants $\lambda_1, \lambda_2, \lambda_3 \geq 0$. All of these equations imply that $\lambda_1 = \lambda_2 = \lambda_3$. This suffices for the proof. 
\end{proof}
\bibliographystyle{plain}
\bibliography{ref}
\end{document}